\documentclass{article}

\newcommand{\Addresses}{{
		\bigskip
		\footnotesize
		
		\textsc{Department of Mathematics, Technion - Israel Institute of Technology, Haifa, Israel}\par\nopagebreak
		\textit{E-mail address:} \texttt{ofir.gor@technion.ac.il}

        \medskip

        \textsc{Department of Mathematical Sciences, Durham University, Stockton Road, Durham DH1 3LE}\par\nopagebreak
		\textit{E-mail address:} \texttt{mo-dick.wong@durham.ac.uk}
}}

\title{Multiplicative chaos measure for multiplicative functions: the $L^1$-regime}

\author{Ofir Gorodetsky, Mo Dick Wong}
\date{}

\usepackage[margin=1in]{geometry}
\usepackage{amssymb}
\usepackage{amsfonts}
\usepackage{amsthm}
\usepackage{thmtools}
\usepackage{amsmath}
\usepackage{hyperref}
\usepackage{cleveref}
\usepackage{mathtools}
\usepackage{url}
\usepackage{xcolor}
\usepackage{enumitem}

\theoremstyle{plain}
\newtheorem{thm}{Theorem}[section]
\newtheorem{lem}[thm]{Lemma}  
\newtheorem{proposition}[thm]{Proposition}
\newtheorem{cor}[thm]{Corollary}
\newtheorem{definition}[thm]{Definition}
\theoremstyle{remark}
\newtheorem{rem}{Remark}[section]

\newcommand{\PP}{\mathbb{P}}
\newcommand{\RR}{\mathbb{R}}
\newcommand{\CC}{\mathbb{C}}
\newcommand{\EE}{\mathbb{E}}
\newcommand{\NN}{\mathbb{N}}
\newcommand{\ZZ}{\mathbb{Z}}

\newcommand{\Fa}{\mathcal{F}}
\newcommand{\Ga}{\mathcal{G}}
\newcommand{\Na}{\mathcal{N}}
\newcommand{\Ma}{\mathcal{M}}
\newcommand{\Ea}{\mathcal{E}}
\newcommand{\Sa}{\mathcal{S}}
\newcommand{\Wa}{\mathcal{W}}
\newcommand{\Pa}{\mathcal{P}}
\newcommand{\Ia}{\mathcal{I}}

\newcommand{\OurEpsilon}{\varepsilon}
\newcommand{\weight}{\Phi}
\newcommand{\fname}{H}

\numberwithin{equation}{section}

\begin{document}

\maketitle

\begin{abstract}
Let $\alpha$ be a Steinhaus random multiplicative function. For a wide class of multiplicative functions $f$ we construct a multiplicative chaos measure arising from the Dirichlet series of $\alpha f$, in the whole $L^1$-regime. Our method does not rely on the thick point approach or Gaussian approximation, and uses a modified second moment method with the help of an approximate Girsanov theorem. We also employ the idea of weak convergence in $L^r$ to show that the limiting measure is independent of the choice of the approximation schemes, and this may be seen as a non-Gaussian analogue of Shamov's characterisation of multiplicative chaos.

Our class of $f$-s consists of those for which the mean value of $|f(p)|^2$ lies in $(0,1)$. In particular, it includes the indicator of sums of two squares. As an application of our construction, we establish a generalised central limit theorem for the (normalised) sums of $\alpha f$, with random variance determined by the total mass of our measure.
\end{abstract}
\section{Introduction}
Let $\alpha\colon \NN \to \CC$ be a Steinhaus random multiplicative function. This is a multiplicative function such that $(\alpha(p))_{p\text{ prime}}$ are i.i.d.~random variables uniformly distributed on $\{ z \in \CC: |z|=1\}$, and $\alpha(n) := \prod_{p}\alpha(p)^{a_p}$ if $n$ factorises into primes as $n=\prod_{p} p^{a_p}$. This function satisfies the orthogonality relation
\begin{equation}\label{eq:orth}
\EE [\alpha(n)\overline{\alpha}(m)]=\delta_{n,m}
\end{equation}
for all $n,m \in \NN$. We denote by $\Ma$ the set of (deterministic) multiplicative functions on $\NN$,
\[ \Ma := \{ f \colon \NN \to \CC : (n,m)=1 \implies f(nm)=f(n)f(m)\}.\]
Given $f\in \Ma$ we consider the random Euler product
\[A_y(s) := \prod_{p \le y} \bigg(\sum_{k \ge 0} \frac{\alpha(p^k) f(p^k)  }{p^{ks}}\bigg)=\sum_{\substack{n\ge 1\\p\mid n \implies p \le y}} \frac{\alpha(n)f(n)}{n^s}\]
where $y \ge 2$ and $\Re s >0$. 
Motivated by questions concerning the asymptotics for partial sums of random multiplicative functions (see \Cref{sec:app-rs}), we provide two constructions of a multiplicative chaos measure associated with $A_y$
in the entire subcritical ($L^1$) regime, one based on truncated Euler product (martingale approach) and the other based on approximation away from the critical line. In addition, we show that the two constructions give rise to the same limiting measure, based on a novel modified second moment method (see \Cref{sec:main-ideas} for further details) without Gaussian approximation or thick point analysis. Our result may also be seen as an example/partial answer to an open problem raised by Junnila in \cite{Jun2020} regarding the universality of non-Gaussian multiplicative chaos.

To describe the functions $f\in \Ma$ to which our construction applies we introduce a set of functions. Given $\theta\in \CC$ we denote by $\mathbf{P}_\theta$ the set  of functions
$g\colon \NN \to \CC$ that satisfy the following: the sum $\pi_{g}(t):=\sum_{p\le t} g(p)$ can be written as
\begin{equation}\label{eq:condition-P}
\pi_{g}(t) = \theta \frac{t}{\log t} + \Ea_g(t),
\end{equation}
where $\Ea_g(t) = o(t / \log t)$ as $t \to \infty$ and $\int_2^\infty t^{-2} |\Ea_g(t)|dt < \infty$.\footnote{
Let $\mathrm{Li}(t)=\int_{2}^{t} \tfrac{dv}{\log v}=\tfrac{t}{ \log t}  + O(\tfrac{t}{\log^2 t})$. Abusing notation, we may write \eqref{eq:condition-P} equivalently as
 $\pi_{g}(t) = \theta \mathrm{Li}(t) + \Ea_g(t)$ for $t\ge 2$, where $\Ea_g(t) = o(t / \log t)$ is such that $\int_2^\infty t^{-2} |\Ea_g(t)|dt < \infty$. We will use both formulations interchangeably.}
\begin{thm}\label{thm:mc-L1}
Let $\theta \in (0, 1)$, and $f\in \Ma$ be a function such that $|f|^2 \in \mathbf{P}_\theta$ and
\begin{equation}\label{eq:f-summability}
  \sum_p \left[\frac{|f(p)|^3}{p^{3/2}} \log^3 p
+ \left(\frac{|f(p^2)|^2}{p^{2}} +\sum_{k \ge 3} \frac{|f(p^k)|}{p^{k/2}} \right)\log^2 p\right] < \infty.
\end{equation}
Write 
\[\sigma_t = \frac{1}{2}\left(1 + \frac{1}{\log t}\right) \quad \text{and} \quad     m_{y, t}(ds) := \frac{|A_y(\sigma_t + is)|^2}{\EE\left[|A_y(\sigma_t + is)|^2\right]}ds, \quad s \in \RR.\]
Then there exists a non-trivial random Radon measure $m_\infty(ds)$ on $\RR$  such that the following are true: for any bounded interval $I$ and any test function $h \in C(I)$, we have
\[   (i) \quad m_{y, \infty}(h) \xrightarrow[y \to \infty]{L^r} m_\infty(h)
    \qquad \text{and} \qquad
  (ii) \quad   m_{\infty, t}(h) \xrightarrow[t \to \infty]{L^r} m_\infty(h)\]
for any $r \in [1, 1/\theta) \cap [1,2]$. In particular, the above convergence also holds in probability. Moreover,  the limiting measure $m_\infty(ds)$ is supported on $\RR$ and non-atomic almost surely.
\end{thm}
\begin{rem}
Since the case $\theta \in (0, 1/2)$ was treated in \cite{GW}, we shall focus on the $L^1$-regime $\theta \in [1/2, 1)$ and $L^r$-convergence with $r \in [1, 1/\theta)$ in this paper.
\end{rem}
As an immediate corollary, proved at the end of \Cref{sec:measure}, we have:
\begin{cor}\label{cor:mc-L1}
Under the same assumption on $f$ as in \Cref{thm:mc-L1}, let $t(y)$ be any sequence satisfying $0 < t(y) \to \infty$ as $y \to \infty$. Then for any $h \in L^1(\RR)$, we have
\[m_{y, t(y)}(h) \xrightarrow[y \to \infty]{L^1} m_\infty(h).\]
In particular, the convergence above also holds in probability.
\end{cor}
\Cref{thm:mc-L1} strengthens our earlier result \cite[Theorem 4.1]{GW} where the restriction $\theta < 1/2$ was imposed. Examples of functions satisfying the conditions of \Cref{thm:mc-L1} include
\begin{itemize}
    \item Any $f\in \Ma$ taking the values $0$ and $1$, such that $\sum_{f(p)=1,\,p\le x}1=(1+O(\log^{-\delta} x))\theta x/\log x$ holds for some $\theta \in (0,1)$ and $\delta>0$. An explicit example of such an $f$ is the indicator of sums of two squares, with $\theta=1/2$. Indeed, it is known to be multiplicative and it takes the value $1$ on primes $p\not\equiv 3\bmod 4$. The Prime Number Theorem in Arithmetic Progressions implies $\sum_{p\equiv 1\bmod 4,\, p \le x} 1 = (1+O(1/\log x))x/(2\log x)$ \cite[Theorem~11.20]{MV}.
    \item The function $f(n)=d_{\sqrt{\theta}}(n)$ for $\theta \in (0,1)$, where $d_z$ stands for the $z$-fold divisor function. Similarly, $f(n)=\theta^{\Omega(n)/2}$ and $f(n)=\theta^{\omega(n)/2}$ where $\theta \in (0,1)$, $\Omega$ counts prime factors with multiplicity and $\omega$ counts them without. These $f$ are in $\mathbf{P}_\theta$ as a direct consequence of the Prime Number Theorem with Error Term \cite[Theorem~6.9]{MV}.
\end{itemize}
\subsection{Application to random sums}\label{sec:app-rs}
Given $f \in \Ma$ we consider the random sum
\[S_x=S_{x,f}:=\frac{1}{\sqrt{\sum_{n\le x} |f(n)|^2}}\sum_{n\le x} \alpha(n) f(n).\]
We refer the reader to \cite{GW} for a discussion of the motivation for studying $S_x$ and overview of previous works, and also the recent paper of Hardy \cite{Har2025} for progress towards critical case. We describe the limiting distribution of $S_x$ as $x\to \infty$ for a wide class of $f$-s. 
\begin{thm}\label{thm:summain}
Let $f \in \Ma$. Suppose the following conditions are satisfied.
\begin{itemize}
\item[(a)] There is $\theta \in (0,1)$ such that $|f|^2 \in \mathbf{P}_{\theta}$.
\item[(b)] There is $c>0$ such that $|f(p^k)|^2=O(2^{k(1-c)})$ holds for all $k\ge 2$ and primes $p$, and $\sum_{p} |f(p)|^3\tfrac{\log^5p}{p^{3/2}}<\infty$.
\end{itemize}
Then we have
\[    S_x \xrightarrow[x \to \infty]{d} \sqrt{V_\infty} \ G\]
where $\displaystyle V_\infty:= \frac{1}{2\pi}\int_{\RR} \frac{m_\infty(ds)}{|\tfrac{1}{2} + is|^2}$ is almost surely finite and strictly positive ($m_\infty$ is as defined in \Cref{thm:mc-L1}), and is independent of $G \sim \Na_\CC(0, 1)$. 
Moreover, the convergence in distribution is stable in the sense of \Cref{def:stable}.
\end{thm}
We also obtain convergence of moments.
\begin{cor}\label{cor:mom-convergence}
Under the same setting as \Cref{thm:summain}, for any fixed $0 \le q \le 1$, 
\[     \lim_{x \to \infty} \EE\left[|S_x|^{2q}\right] = \Gamma(1+q) \EE\left[V_\infty^q\right]. \]
\end{cor}
Given \Cref{thm:summain}, the deduction of \Cref{cor:mom-convergence} is identical to the proof of Corollary 1.2 in \cite{GW}, and we omit the details here. \Cref{thm:summain} and \Cref{cor:mom-convergence} resolve Conjecture 1.5 from our earlier paper \cite{GW}, except for the range of $q$ in \Cref{cor:mom-convergence} being suboptimal.

The case where $f$ is taken to be the indicator of sums of two squares in \Cref{thm:summain} is of particular interest. Soundararajan and Xu \cite[Corollary~1.2]{SoundXu} proved that in this case, if we let $I=[x-h,x]$ where $h=h(x)$ is a function of $x$, then $\sum_{n \in I}\alpha(n)f(n) / \sqrt{\sum_{n \in I}|f(n)|^2}\xrightarrow[x \to \infty]{d}  G$ where $G \sim \Na_\CC(0, 1)$, as long as $h$ is sublinear in $x$ ($\lim_{x\to \infty} h/x=0$) and is not too small ($h>x^{1/3}$). \Cref{thm:summain} shows that once $h$ is linear in $x$ the distribution changes dramatically. In \Cref{thm:summainw} we prove a weighted version of \Cref{thm:summain} which allows one to restrict $n$ to lie in intervals of the shape $I=[cx,x]$ for any $c\in (0,1)$.
\subsection{Main ideas}\label{sec:main-ideas}
\paragraph{Construction of multiplicative chaos.}
To establish \Cref{thm:mc-L1}, it suffices to prove an analogous statement for the random measures
\[ \nu_{y, t}(ds) := \frac{\exp\left(\Ga_{y, t}(s)\right)}{\EE\exp\left(\Ga_{y, t}(s)\right)} ds
\qquad \text{with} \qquad 
\Ga_{y, t}(s) = \sum_{p \le y} G_{p, t}(s), \quad G_{p, t}(s) := 2 \Re \frac{f(p) \alpha(p)}{p^{\sigma_t + is}}, \]
i.e.~there exists a random measure $\nu_\infty$ such that $\nu_{y, \infty}, \nu_{\infty, y} \xrightarrow[y \to \infty]{p} \nu_\infty$ as $y \to \infty$. The claim for the sequence $(\nu_{y, \infty})_{y}$ (originating from truncated Euler product) readily follows from martingale convergence theorem provided one can establish uniform integrability (for non-trivial limits); it is the other approximation sequence $(\nu_{\infty, y})_y$ that requires extra work. To show both approximation sequences lead to the same limiting measure, it is very natural to apply a second moment method, i.e.~one would like to verify that
\[   \EE\left[ |\nu_{y, \infty}(h) - \nu_{\infty, y}(h)|^2\right] \xrightarrow[y \to \infty]{} 0 \]
by expanding the expectation on the left-hand side and evaluating limits separately. This approach was implemented in \cite{GW} when $\theta <1/2$. The issue about this approach is that beyond the $L^2$-regime (i.e.~when $\theta \in [1/2, 1)$) the individual terms blow up as $y \to \infty$ since e.g.
\[    \EE\left[ \nu_{y, t}(I)^2 \right] \asymp \int_{I \times I} \frac{du dv}{ \left(|u - v| \vee \log^{-1} \min(y, t)\right)^{2\theta}} \xrightarrow[y, t \to \infty]{} \infty. \]
To circumvent the problem, Berestycki \cite{Berestycki} used the idea of thick points (also known as barriers) in the construction of Gaussian multiplicative chaos (GMC) and introduced good events to avoid blow-up in moment computations. At a high level, this requires restricting the support of the approximating measures to points at which the value of the underlying fields are not exceptional large. Here we adopt instead a novel modified second moment method, which amounts to analysing expressions like
\begin{equation}\label{eq:mod2ndmom}
    \EE\left[ |\nu_{y, \infty}(h) - \nu_{\infty, y}(h)|^2 e^{-\nu_{y, \infty}(I) / K}\right]
\end{equation}
and showing that it vanishes as $y \to \infty$ for any fixed $K > 0$ (see \Cref{sec:approximate-away-critical} for the actual modified second moments studied). In simple terms, the introduction of the extra exponential factor automatically penalises events that are responsible for the blow-up of the usual 2nd moment (i.e.~when $\nu_{y, \infty}(h)$ or $\nu_{\infty, y}(h)$ are atypically large). Unlike Berestycki's approach which is based on imposing a thickness/barrier condition on the support of the limiting measure $\nu_\infty$, our analysis only requires the much weaker information that $\nu_\infty$ is non-atomic, which may be obtained for free from earlier uniform integrability estimates.

At first glance, it may come as a surprise that one could analyse \eqref{eq:mod2ndmom} by expanding the terms and evaluating limits separately given its complicated form. It turns out that this could be achieved by a simple change-of-measure argument in the spirit of Cameron--Martin--Girsanov theorem (in \Cref{sec:pf-ch-measure} we provide a further discussion of our philosophy before the detailed proof). Even though our underlying random variables are non-Gaussian, this could be made possible by our approximate Girsanov theorem (see \Cref{theo:Girsanov}) which may be of independent interest. Another innovation in our method is the use of the concept of weak convergence in $L^r$, which not only helps simplify computations but also provides a way to establish the universality of non-Gaussian multiplicative chaos (i.e.~its construction is independent of the choice of approximation schemes, see \Cref{sec:weakpf}). This strategy shares a similar philosophy with Shamov's characterisation of subcritical Gaussian multiplicative chaos \cite{Sha2016}, and can be seen as a substitute for his randomised shift in the non-Gaussian setup.

Our method does not rely on any Gaussian approximations, which is in contrast to various works on non-Gaussian multiplicative chaos in recent years. The work of Saksman and Webb \cite{sw}, for instance, studies the martingale sequence $(\nu_{y, \infty})_y$ when $f(p) \equiv \theta$ by constructing a coupling between $(\alpha(p))_p$ and a collection of i.i.d.~Gaussian random variables. With careful analysis, it is likely that their technique could be extended to more general twists $f$ like the ones considered in the current article, but the method does not work very well for the approximation away from the critical line since this would introduce a collection of couplings indexed by $y$, i.e.~they may not be defined on a common probability space and the convergence of $\nu_{\infty, y}$ as $y \to \infty$ would hold only in the sense of distribution (and additional efforts would be needed to identify and upgrade the limit). Resolving a conjecture of Kim and Kriechbaum \cite{KK2024}, Chowdhury and Ganguly \cite{CG2025} introduce a new Gaussian approximation technique, but their result does not cover Steinhaus random variables (the real and imaginary parts of which are orthogonal but not independent) and their proof technique does not extend immediately because of the use of Skorokhod embedding.
\paragraph{Martingale central limit theorem for random multiplicative functions.}
In \Cref{sec:randomsum} we undertake the proof of \Cref{thm:summain}, in which we take inspiration from the recent work of Najnudel, Paquette, Simm and Vu \cite{NPSV} on holomorphic multiplicative chaos. The sum $S_{x}$ has a natural martingale structure with respect to the filtration $\Fa_{p} := \sigma(\alpha(q), q \le p)$. Indeed, $S_x = \sum_{p \le x} Z_{x,p}$ where $(Z_{x,p})_{p\le x}$ is a martingale difference sequence, namely
\[ Z_{x,p} = \big(\sum_{n\le x} |f(n)|^2\big)^{-\frac{1}{2}}\sum_{\substack{n\le x \\P(n) = p}} \alpha(n)f(n)\]
where $P(n)$ stands for the largest prime factor of $n$. The random variable $ Z_{x,p}$ is determined by $\Fa_p$. It is then natural to apply a martingale central limit theorem to $(S_x)_x$; Harper was the first to successfully apply such a theorem when $f$ is the indicator of product of $k$ primes \cite{Harper}. The main difficulty in applying such a theorem is the computation of the so-called bracket process, which in our case is
\[T_{x} :=\sum_{p\le x} \EE\left[ |Z_{x,p}|^2 \mid \Fa_{p^-}\right] \approx (\sum_{n\le x}|f(n)|^2)^{-1} \sum_{p \le x} |f(p)|^2 \big|\sum_{\substack{m\le x/p\\P(m)<p}} f(m) \alpha(m)\big|^2.\]
If we introduce
\[ s_{t,y} := t^{-\frac{1}{2}}\sum_{\substack{n\le t\\ P(n) \le y}}\alpha(n) f(n)\]
then 
\[T_{x}\approx x\big(\sum_{n\le x}|f(n)|^2\big)^{-1} \sum_{p \le x} \frac{|f(p)|^2}{p} |s_{x/p,p}|^2.\]
If we let
\[R(t):=\sum_{p \le t}\frac{|f(p)|^2}{p} \approx \theta \log \log t\]
then, since $R'(t) \approx 1/(t\log t)$,
\begin{equation}\label{eq:Txapprox} 
T_x  \approx \theta x\big(\sum_{n\le x}|f(n)|^2\big)^{-1} \int_{2}^{x} \frac{|s_{x/t,t}|^2 dt}{t\log t}.
\end{equation}
From \Cref{cor:mc-L1} and Plancherel's theorem, we may obtain the limit of a somewhat similar expression, namely
\begin{equation}\label{eq:planc}
x\big(\sum_{n\le x}|f(n)|^2\big)^{-1} \int_{0}^{\infty} q(t^{1/\log x})\frac{|s_{x/t,x^a}|^2 dt}{t\log x}
\end{equation}
for any fixed $a>0$ and any `nice' function $q$, see \Cref{lem:plancherelapp} and \Cref{cor:sandwich}. The main difference between \eqref{eq:Txapprox} and \eqref{eq:planc}, however, is that in \eqref{eq:planc} the `smoothness parameter' in $s_{x/t,x^a}$, namely $x^a$, does not depend on the integration variable $t$, while in \eqref{eq:Txapprox} the smoothness parameter in $s_{x/t,t}$ is $t$, the integration variable itself. 

To circumvent this issue, we modify $S_x$. We divide the primes in $[2,x]$ into finitely many disjoint intervals $I_k=(x_k,x_{k+1}]$ ($x_k$ an increasing sequence, $0\le k \le K$), and if $n\le x$ has $P(n) \in I_k$, we `keep' this $n$ in the modified version of $S_x$ only if $k\ge 1$ and if $P(n/P(n))$ (the second largest prime factor of $n$, counting multiplicity) does not exceed $\inf I_k=x_k$. In this way, the new bracket process takes the shape
\begin{equation}\label{eq:newbracket}
T'_x \approx x \big(\sum_{n \le x}|f(n)|^2\big)^{-1} \sum_{k\ge 1} \sum_{p\in I_k} \frac{|f(p)|^2}{p} |s_{x/p, x_k}|^2 \approx \theta x  \big(\sum_{n \le x}|f(n)|^2\big)^{-1} \sum_{k\ge 1} \int_{x_k}^{x_{k+1}} \frac{|s_{x/t,x_k}|^2 dt}{t\log t} .
\end{equation}
For a given $k$, the smoothness parameter is now fixed within the integral, namely it is $x_k$. This allows us to handle the $k$th integral using \eqref{eq:planc}. One has to justify working with this modified $S_x$, and this is done in \Cref{lem:negdelta} and its proof. 

A complication that does not arise in \cite{NPSV} is the fact that $T'_x$ is a sum over primes, and we want to replace it by an integral against $dt/\log t$ (the density of primes); this is exactly the second approximation in \eqref{eq:newbracket}. We justify it in \Cref{lem:closel1}.

Before we continue, let us also mention a recent work of Garban and Vargas \cite{GV2023}, where a similar distributional result has been established in the high-frequency limit for the Fourier coefficients of real GMCs in the analogous subcritical regime. The approach there is based on a very different conditioning argument, but does not apply to our present problem due to their use of the cone construction/white-noise decomposition of logarithmically correlated fields. More specifically, this technique introduces independence at suitable scales crucial to their novel application of a result from the literature of Stein's method on weakly dependent variables. This decomposition breaks down immediately, however, when one studies non-Gaussian fields coming from e.g.~the Euler product associated to (twisted) Steinhaus multiplicative functions, and it would be challenging but also interesting to see how their philosophy could be adapted to the general setting.

\paragraph{Notation.}
Following the number-theoretic convention, we use Vinogradov's notation $A\ll B$ to mean the same thing as $A = O(B)$, i.e.~$|A|\le CB$ holds for some absolute constant $C$; we also write $A \ll_r B$ if the implicit constant depends on some parameter $r$.
\section{Part I: Construction of measure}\label{sec:measure}
This section is devoted to the proof of \Cref{thm:mc-L1} and is organised as follows:
\begin{itemize}
\item In \Cref{sec:elementary-estimates}, we compile a list of elementary estimates for the multiplicative function $f$ and exponential moments involving $G_{y, t}(s) := 2\Re \frac{f(p) \alpha(p)}{p^{\sigma_t + is}}$, which will be used throughout the entire section.
\item In \Cref{sec:mc-prepare}, we collect two probabilistic results which allow us to reformulate our problem in terms of new sequence of measures $\nu_{y, t}$ defined in \eqref{eq:reduced-measure} and \eqref{eq:def-nu} (see \Cref{lem:reduced-mc}).
\item In \Cref{sec:mc-moment}, we derive various moment estimates for the measure $\nu_{y, t}$, which will allow us to reduce $L^r$-convergence to convergence in probability by a standard argument of uniform integrability. We also explain a tensorisation trick (\Cref{lem:tensorise}) which will be used later in the modified second moment method.
\item In \Cref{sec:mc-martingale-support}, we quickly recall the almost sure convergence of $\nu_{y, \infty}(h)$ to some limiting variable $\nu_\infty(h)$, and explain the proof of the non-atomicity of the underlying random measure $\nu_\infty$.
\item In \Cref{sec:approximate-away-critical}, we establish $\nu_{\infty, y}(h) \xrightarrow[y \to \infty]{p} \nu_\infty(h)$ and explain our new modified second moment method with the help of an approximate Girsanov theorem (see \Cref{theo:Girsanov}).
\end{itemize}

\subsection{Elementary estimates}\label{sec:elementary-estimates}
\subsubsection{Estimates for multiplicative functions}
Our first estimate is borrowed from \cite[Lemma 4.1]{GW}.
\begin{lem}\label{lem:truncate}
Let $f\colon \NN \to \CC$ be a function satisfying \eqref{eq:f-summability} and $|f|^2 \in \mathbf{P}_\theta$ for some $\theta >0$. Then there exists some deterministic $y_0 = y_0(f) > 0$ such that
\[\left|1 + \sum_{k =1}^2 \frac{\alpha(p)^k f(p^k)}{p^{z}}\right|^{-2} \left|1 + \sum_{k \ge 1} \frac{\alpha(p)^k f(p^k)}{p^{z}}\right|^2
= 
1 + O\left(  \left|\frac{f(p)}{p^{1/2}}\right|^3 + \left|\frac{f(p^2)}{p}\right|^3 +  \sum_{k \ge 3} \left|\frac{f(p^k)}{p^{k/2}}\right| \right)\]
where the implicit constants on the right-hand side are uniform in $p \ge y_0$, $\Re z \ge \frac{1}{2}$, and also for any sequence $(\alpha(p))_{p}$ satisfying $\sup_p |\alpha(p)| \le 1$. In particular, there exists some deterministic constant $C = C(y_0) \in (0, \infty)$ such that
\[\prod_{p \ge y_0} \left|1 + \sum_{k \ge 1} \frac{\alpha(p)^k f(p^k)}{p^{kz}}\right|^2
\le C \prod_{p \ge y_0} \left|1 + \sum_{k = 1}^2 \frac{\alpha(p)^k f(p^k)}{p^{kz}}\right|^2\]
uniformly in $\Re z \ge \frac{1}{2}$.
\end{lem}
The next result is a small generalisation of \cite[Lemma 4.3]{GW}.
\begin{lem}\label{lem:new-martingaleL2est}
Let $g\colon \NN \to \CC$ be a function such that $g \in \mathbf{P}_\theta$ for some $\theta \in \CC$. We have
\begin{equation}\label{eq:new-exp-sum}
\sum_{p \le y} \frac{g(p)}{p^{1+\frac{c}{\log t}}} \cos(s \log p) 
= \theta \log\left[\min \left(|s|^{-1}, \log y, \log t\right)\right]  + O(1)
\end{equation}
uniformly in $y, t \ge 3$ and any compact subsets of $s \in \RR$ and $c \ge 0$. Moreover, for each fixed $c \ge 0$ there exists some constant $C_g = C_g(c) \in \CC$ such that
\begin{equation}\label{eq:new-log-sum-twist}
\sum_{p \le y} \frac{g(p)}{p^{1+\frac{c}{\log y}}}
= \theta \log\log y+ C_g + o(1)
\qquad \text{as $y \to \infty$.}
\end{equation}
\end{lem}
\begin{proof}
Let $\OurEpsilon := c / \log t  > 0$. We have
\begin{align*}
\sum_{p \le y} \frac{g(p)}{p^{1+\frac{c}{\log t}}}
& = \int_{2^-}^{y^+} \frac{\cos(s \log p)}{p^{1 + \frac{c}{\log t}}} \left[ \frac{\theta dp}{\log p} + d\Ea_g(p)\right]\\
& = \theta \int_{ \log 2}^{\log y} e^{-\OurEpsilon u} \cos (|s| u) \frac{du}{u} 
+ \int_{2^-}^{y^+} \frac{\cos(s \log p)}{p^{1 + \frac{c}{\log t}}}d\Ea_g(p)
=: I + II.
\end{align*}
We first treat the error term: using integration by parts,
\[II =  \frac{\cos(s \log p)}{p^{1 +\OurEpsilon}}\Ea_g(p) \Bigg|_{2_-}^{y^+} + \int_2^y \frac{\Ea_g(p)}{p^{2 + \OurEpsilon}} \left[  (1 +\OurEpsilon)\cos(s \log p) + s \sin (s \log p)\right] dp\]
which is uniformly bounded in $y, t \ge 3$ and $s$ in any compact subsets of $\RR$ because $\Ea_g(x) = o(x)$ and $\int_2^\infty x^{-2} |\Ea_g(x)| dx < \infty$.

As for the main term, we will use the observation that for any $a, b \ge 0$, the function $0 \le u \mapsto e^{-au} \cos(bu)$ has a Lipschitz constant bounded by $a + b$. We now verify \eqref{eq:new-exp-sum} case by case:
\begin{itemize}
\item When $\log y \le \min (\log t, |s|^{-1})$,
\[I = \theta \int_{\frac{\log 2}{\log y}}^{1} e^{-\OurEpsilon u \log y} \cos (u |s| \log y) \frac{du}{u} 
= \theta \int_{\frac{\log 2}{\log y}}^{1} \frac{du}{u}  + \theta \int_{\frac{\log 2}{\log y}}^{1} \left[e^{-\OurEpsilon u \log y} \cos (u |s| \log y) - 1\right]\frac{du}{u} .\]
The first integral on the right-hand side is equal to $\theta \log \frac{\log y}{\log 2}$, whereas the second integral is bounded with desired uniformity because $\OurEpsilon \log y = c\frac{\log y}{\log t} \in [0, c]$ and $|s| \log y \in [0,1]$.

\item When $|s|^{-1} \le \min(\log y, \log t)$, we have
\begin{align*}
I & =\theta \int_{|s| \log 2}^{|s| \log y} e^{-|s|^{-1} \OurEpsilon u } \cos(u) \frac{du}{u}\\
&= \theta \left\{\int_{|s| \log 2}^1 \frac{du}{u} + \int_{|s| \log 2}^{1} \left[e^{-|s|^{-1} \OurEpsilon u }\cos(u) - 1\right] \frac{du}{u} + \int_1^{|s| \log y} e^{-|s|^{-1} \OurEpsilon u } \cos(u) \frac{du}{u}\right\}
\end{align*}
where
\begin{itemize}
\item the first integral on the right-hand side is equal to $-\log(|s| \log 2)$;
\item the second integral is bounded uniformly for $s$ in any compact subset because $|s|^{-1} \OurEpsilon \in [0, c]$, $|s| \log 2$ is bounded, and the integrand is bounded uniformly near $u = 0$;
\item using the identity
\[\frac{d}{dx} \frac{e^{-bx}[a\sin(ax) - b \cos(ax)]}{a^2 + b^2}
= e^{-bx} \cos (ax)\]
and $|s| \log y \ge 1$, the third integral
\[\left[ e^{-|s|^{-1} \OurEpsilon u} \frac{\sin(u) - |s|^{-1} \OurEpsilon \cos(u)}{1 + (|s|^{-1} \OurEpsilon)^2}\right]_1^{|s| \log y}
+ \int_1^{|s| \log y} e^{-|s|^{-1} \OurEpsilon u} \frac{\sin(u) - |s|^{-1} \OurEpsilon \cos(u)}{1 + (|s|^{-1} \OurEpsilon)^2} \frac{du}{u^2}\]
is also bounded with the desired uniformity.
\end{itemize}
\item When $\log t \le \min(\log y, |s|^{-1})$ such that $|s| \log t \le 1$ and $\frac{\log y}{\log t} \ge 1$, 
\begin{align}
\notag
I &= \theta \int_{\frac{\log 2}{\log t}}^{\frac{\log y}{\log t}} e^{-c u} \cos(u |s| \log t) \frac{du}{u}\\
\label{eq:ymin-term}
& = \theta \left\{
\int_{\frac{\log 2}{\log t}}^{1}  \frac{du}{u}
+ \int_{\frac{\log 2}{\log t}}^{1} \left[e^{-c u} \cos(u |s| \log t) - 1\right]\frac{du}{u}
+ \int_{1}^{\frac{\log y}{\log t}} e^{-c u} \cos(u |s| \log t) \frac{du}{u}
\right\}
\end{align}
which is equal to $\theta \log \log t + O(1)$ by analysis similar to that in the previous case.
\end{itemize}
The claim \eqref{eq:new-log-sum-twist} follows immediately by inspecting the identity \eqref{eq:ymin-term} and setting $s = 0$ as well as $t = y$.
\end{proof}
The third result is again borrowed from \cite[Lemma 4.4]{GW}.
\begin{lem}\label{lem:martingaleL2est-2}
Let $g\colon \NN \to \CC$ be a function such that $g \in \mathbf{P}_\theta$ for some $\theta \in \CC$. Then
\[\sum_{p > y}\frac{g(p)\cos(s\log p)}{p^{1+ \frac{1}{\log y}}} = O(1)\]
uniformly in $y \ge 3$ and $s$ in any compact subset of $\RR$. Moreover,
\begin{itemize}
\item For $s$ in any compact subset of $\RR \setminus \{0\}$, we have uniformly
\[\sum_{p > y}\frac{g(p)\cos(s\log p)}{p^{1+ \frac{1}{\log y}}} = o(1)
\qquad \text{as $y \to \infty$}.\]
\item For $s = 0$, we have
\[    \lim_{y \to \infty} \sum_{p > y}\frac{g(p)}{p^{1+ \frac{1}{\log y}}} = \theta \int_1^\infty e^{-u} \frac{du}{u}.\]
\end{itemize} 
\end{lem}
\subsubsection{Estimates for exponential moments}
Next, we need to collect a few estimates for $G_{y, t}(s) := 2\Re \frac{f(p) \alpha(p)}{p^{\sigma_t + is}}$. We start with the asymptotics for the ``one-point" and ``two-point" functions.
\begin{lem}\label{lem:gf-estimates}
We have
\begin{align}
\label{eq:gf-1point}
\EE\left[\exp \left( G_{p, t}(s_1)\right)\right] & =  \exp\left(\frac{1}{2} \EE[G_{p, t}(s_1)^2] + O\left( \left| \frac{f(p)}{p^{\sigma_t}}\right|^3\right) \right)\\
\label{eq:gf-2point}
\text{and} \qquad \EE\left[\exp \left( G_{p, t}(s_1) + G_{p, t}(s_2)\right)\right] & =  \exp\left(\frac{1}{2} \EE\left[ \left(G_{p, t}(s_1) + G_{p, t}(s_2)\right)^2\right] + O\left( \left| \frac{f(p)}{p^{\sigma_t}}\right|^3\right) \right)
\end{align}
uniformly in $s_1, s_2 \in \RR$ and $t \ge 3$ as $p \to \infty$.
\end{lem}
\begin{proof}
Note that $\|G_{p, t}\|_\infty \le 2|f(p) / p^{\sigma_t}|^3$ and in particular it is uniformly bounded in $p$ due to our assumption on $f$. By Taylor series expansion, we obtain
\begin{align*}
\EE\left[\exp \left( G_{p, t}(s_1)\right)\right]
&= \EE\left[ 1 + G_{p, t}(s_1) + G_{p, t}(s_1)^2 + O\left(|f(p) / p^{\sigma_t}|^3\right) \right]\\
&= 1 +  \EE\left[G_{p, t}(s_1)^2\right] + O\left(|f(p) / p^{\sigma_t}|^3\right)
\end{align*}
which gives \eqref{eq:gf-1point}. The other estimate \eqref{eq:gf-2point} is similar.
\end{proof}
The next lemma explains how ``small tilting" affects mean values and exponential moments.
\begin{lem}\label{lem:tilt_exp}
Let $X$ be an $\RR^d$-valued random variable satisfying $\EE[X] = 0$ and $\PP(|X| \le r) = 1$ for some fixed $r > 0$. Then
\begin{align*}
\frac{\EE\left[\langle a_1, X \rangle e^{\langle a_2, X \rangle}\right]}{\EE[e^{\langle a_2, X\rangle}]} 
&= a_1^T \EE[XX^T]a_2 + O( |a_1||a_2|^2)\\
\text{and} \qquad \frac{\EE\left[e^{\langle a_1 + a_2, X \rangle}\right]}{\EE[e^{\langle a_2, X\rangle}]} 
&= \exp \left(\frac{1}{2}a_1^T \EE\left[ X X^T \right] (a_1 +2a_2) +O( |a_1||a_2|^2 + |a_1|^3)\right)
\end{align*}
uniformly for $a_1, a_2$ in compact subsets of $\RR^d$.
\end{lem}
\begin{proof}
Since $X$ is bounded, we have $e^{\langle a_1, X \rangle } = 1 + \langle a_1, X \rangle  + \frac{1}{2} \langle a_1, X \rangle ^2 + O(|a_1|^3)$ where the error can be bounded uniformly, and similarly for $e^{\langle a_2, X \rangle}$. In particular, it is straightforward to check that
\begin{align*}
\langle a_1, X \rangle \frac{e^{\langle a_2, X \rangle}}{\EE[e^{\langle a_2, X \rangle}]}
& =\langle a_1, X \rangle \left [1 +\langle a_2, X \rangle+  O( |a_2|^2)\right]\\
\text{and} \qquad \left(e^{\langle a_1, X \rangle} -1 \right) \frac{e^{\langle a_2, X \rangle}}{\EE[e^{\langle a_2, X \rangle}]}
& =\langle a_1, X \rangle  + \frac{1}{2} \langle a_1, X \rangle^2 +\langle a_1, X \rangle\langle a_2, X \rangle + O( |a_1||a_2|^2 + |a_1|^3),
\end{align*}
and the claim follows from taking expectations on both sides of the above estimates.
\end{proof}
\begin{lem}\label{lem:tilt_ratios}
Let $\OurEpsilon_{p, y} = p^{-1/2\log y} - 1$ and
\[c_{p, y}(s_1, s_2) 
:=  \frac{\EE\exp\left(G_{p, y}(s_1) + G_{p, y}(s_2)\right)}{\EE\exp\left(G_{p, \infty}(s_1) + G_{p, \infty}(s_2)\right)}
\left[\frac{\EE\exp\left(G_{p, y}(s_1)\right)\EE\exp\left(G_{p, y}(s_2)\right)}{\EE\exp\left(G_{p, \infty}(s_1)\right)\EE\exp\left(G_{p, \infty}(s_2)\right)}\right]^{-1}.\]
We have
\begin{equation}\label{eq:ratio_error}
c_{p, y}(s_1, s_2)  = \exp \left\{\frac{2|f(p)|^2}{p} \left(p^{-1/\log y} - 1\right) \cos(|s_1-s_2| \log p)+ O\left( \frac{|\OurEpsilon_{p, y}| |f(p)|^3}{p^{3/2}}\right)\right\}
\end{equation}
uniformly in $y \ge 3$, $p \le y$ and $s_1, s_2 \in \RR$. In particular, $C_{y, y}(s_1, s_2) / C_{y, \infty}(s_1, s_2) = \prod_{p\le y} c_{p, y}(s_1, s_2)$ is uniformly bounded in $y \ge 3, s_1, s_2 \in \RR$, and $C_{y, y}(s_1, s_2) / C_{y, \infty}(s_1, s_2) \to 1 $ as $y \to \infty$ uniformly for all $s_1, s_2 \in \RR$ satisfying $\delta \le |s_1 - s_2| \le \delta^{-1}$ for any fixed $\delta \in (0, 1)$.
\end{lem}
\begin{proof}
Applying \Cref{lem:tilt_exp} with
\[X = \begin{pmatrix} \Re \alpha(p) \\ \Im \alpha(p) \end{pmatrix},
\qquad 
a_1^{(i)} := \OurEpsilon_{p, y} a_2^{(i)}
\quad \text{and} \quad
a^{(i)}_2 := \frac{2|f(p)|}{p^{1/2}} \begin{pmatrix} \cos (s_i \log p) \\ \sin(s_i \log p)\end{pmatrix}\]
where $\OurEpsilon_{p, y} = p^{-1/2\log y} - 1$, we see that $c_{p, y}(s_1, s_2)$ is equal to
\begin{align*}
& \frac{\EE\left[\exp\left(\langle a_1^{(1)} + a_2^{(1)}  + a_1^{(2)} + a_2^{(2)} , X \rangle\right)\right]}{\EE[\exp\left(\langle a_2^{(1)} + a_2^{(2)}, X\rangle\right)]} 
\left\{\frac{\EE\left[\exp\left(\langle a_1^{(1)} +a_2^{(1)}, X \rangle\right)\right]}{\EE[\exp\left(\langle a_2^{(1)}, X\rangle\right)]}  \frac{\EE\left[\exp\left(\langle a_1^{(2)} +a_2^{(2)}, X \rangle\right)\right]}{\EE[\exp\left(\langle a_2^{(2)}, X\rangle\right)]} \right\}^{-1}\\
& = \exp \left\{\left(a_1^{(1)}\right)^T \EE[X X^T]  a_1^{(2)} + \left(a_1^{(1)}\right)^T \EE[X X^T]  a_2^{(2)} + \left(a_1^{(2)}\right)^T \EE[X X^T] a_2^{(1)} + O\left( \frac{|\OurEpsilon_{p, y}| |f(p)|^3}{p^{3/2}}\right)\right\}\\
& = \exp \left\{\frac{2|f(p)|^2}{p} (\OurEpsilon_{p, y}^2 +2 \OurEpsilon_{p, y}) [\cos(s_1 \log p)\cos(s_2 \log p) + \sin(s_1 \log p)\sin(s_2 \log p)]  + O\left( \frac{|\OurEpsilon_{p, y}| |f(p)|^3}{p^{3/2}}\right)\right\}
\end{align*}
which is \eqref{eq:ratio_error}. If we now consider
\begin{align*}
\log \frac{C_{y,y}(s_1, s_2)}{C_{y,\infty}(s_1, s_2)}
&= \sum_{p \le y} \log c_{p, y}(s_1, s_2)\\
& = \sum_{p \le y}\frac{2|f(p)|^2}{p} \left(p^{-1/\log y} - 1\right) \cos(|s_1-s_2| \log p)+ O\left( \sum_{p \le y} \frac{|\OurEpsilon_{p, y}| |f(p)|^3}{p^{3/2}}\right),
\end{align*}
we see that the error term on the right-hand side is uniformly bounded since $\sum_p |f(p)|^3 p^{-3/2} < \infty$, and converges to $0$ by dominated convergence since $\OurEpsilon_{p, y} \to 0$ for each fixed prime $p$. On the other hand, (recalling $\sum_{p \le y} |f(p)|^2 =: \pi_{|f|^2}(y)$) the first sum is equal to
\begin{align}
\notag &\int_{2^-}^{y^+} \frac{2}{p}  \left(p^{-1/\log y} - 1\right) \cos(|s_1 - s_2| \log p) d\left[\theta \mathrm{Li}(p) + \Ea_{|f|^2}(p)\right]\\
\notag 
& = 2\theta \int_2^y \frac{p^{-1/\log y} - 1}{p \log p}  \cos(|s_1 - s_2| \log p) dp + \int_{2^-}^{y^+} \frac{2}{p}  \left(p^{-1/\log y} - 1\right) \cos(|s_1 - s_2| \log p) d \Ea_{|f|^2}(p)\\
\label{eq:twopt-error1}
& = 2 \theta \int_{\frac{\log 2}{\log y}}^1 \frac{e^{-u} - 1}{u} \cos\left(|s_1 - s_2| u \log y\right) du
+\left[ \Ea_{|f|^2}(p) \frac{2}{p}\left(p^{-1/\log y} - 1\right) \right]_{2^{-}}^{y^+}\\
\label{eq:twopt-error2}
& +2 \int_2^y  \frac{\Ea_{|f|^2}(p)}{p^2} \left\{\left[ (p^{-1/\log y} - 1) + \frac{p^{-1/\log y}}{\log y}\right] \cos(|s_1 - s_2|\log p) + |s_1 - s_2| \sin(|s_1 - s_2|\log p) \right\}
dp.
\end{align}
After integration by parts, it is easy to see that the expressions in \eqref{eq:twopt-error1} are of size
\[ O \left( \frac{1}{\log y} + \frac{1}{\max(1, |s_1 - s_2| \log y)} + \frac{\Ea_{|f|^2}(y^+)}{y} \right)\]
which is uniformly bounded in $y\ge 3, s_1, s_2 \in \RR$, and converges to $0$ as $y \to \infty$ uniformly for  $s_1, s_2$ bounded away from each other. Meanwhile, the integral in \eqref{eq:twopt-error2} is uniformly bounded in $y \ge 3$ and $s_1, s_2$ in any compact interval since $\int_2^\infty dp |\Ea_{|f|^2}(p) |/ p^2  < \infty$, and its value converges to $0$ as $y \to \infty$ uniformly for all $s_1, s_2$ in any compact interval by dominated convergence. This concludes our proof.
\end{proof}
\begin{lem}
Let $\OurEpsilon_{p, y} = p^{-1/2\log y} - 1$. We have
\begin{align}
\notag &\frac{\EE\left[ G_{p, \infty}(u) \exp\left(G_{p, y}(s_1) + G_{p, y}(s_2) \right)\right]}{\EE\left[ \exp\left(G_{p, y}(s_1) + G_{p, y}(s_2) \right)\right]}
- \frac{\EE\left[ G_{p, \infty}(u) \exp\left(G_{p, \infty}(s_1) + G_{p, \infty}(s_2) \right)\right]}{\EE\left[ \exp\left(G_{p, \infty}(s_1) + G_{p, \infty}(s_2) \right)\right]} \\
\label{eq:estimate-newMean}
&\qquad =  \frac{2|f(p)|^2 \OurEpsilon_{p, y}}{p}\left[  \cos(|u-s_1| \log p) +  \cos(|u-s_2| \log p)\right]+ O\left(\frac{|f(p)|^3}{p^{3/2}}|\OurEpsilon_{p, y}|\right)
\end{align}
uniformly in  $y \ge 3$, $p \le y$ and $u, s_1, s_2 \in \RR$. 
\end{lem}
\begin{proof}
Let us choose $X$ as before, but now with
\[a_1 = \frac{2|f(p)|}{p^{1/2}} \begin{pmatrix} \cos(u \log p) \\ -\sin(u \log p)\end{pmatrix},
\qquad a_2(t) := \frac{2|f(p)|}{p^{\frac{1}{2} \left(1 + \frac{1}{\log t}\right)}} \begin{pmatrix} \cos(s_1 \log p) + \cos(s_2 \log p) \\ -\sin(s_1 \log p) -\sin(s_2 \log p)\end{pmatrix}.\]
Then the quantity we are interested in is indeed given by
\[\frac{\EE\left[\langle a_1, X \rangle e^{\langle a_2(y), X \rangle}\right]}{\EE[e^{\langle a_2(y), X\rangle}]} 
- \frac{\EE\left[\langle a_1, X \rangle e^{\langle a_2(\infty), X \rangle}\right]}{\EE[e^{\langle a_2(\infty), X\rangle}]} 
= \int_1^{1+\OurEpsilon_{p, y}} \left[\frac{d}{dt} \frac{\EE\left[\langle a_1, X \rangle e^{\langle ta_2(\infty), X \rangle}\right]}{\EE[e^{\langle t a_2(\infty), X\rangle}]} \right]dt.\]
But using \Cref{lem:tilt_exp}, we see that
\begin{align*}
 \frac{d}{dt}& \frac{\EE\left[\langle a_1, X \rangle e^{\langle ta_2(\infty), X \rangle}\right]}{\EE[e^{\langle t  a_2(\infty), X\rangle}]}\\
& = \frac{\EE\left[\langle a_1, X \rangle\langle a_2(\infty), X \rangle e^{\langle t a_2(\infty), X \rangle}\right]}{\EE[e^{\langle t a_2(\infty), X\rangle}]}
- \frac{\EE\left[\langle a_1, X \rangle e^{\langle ta_2(\infty), X \rangle}\right]}{\EE[e^{\langle t a_2(\infty), X\rangle}]}\frac{\EE\left[\langle a_2(\infty) X \rangle e^{\langle ta_2(\infty), X \rangle}\right]}{\EE[e^{\langle t a_2(\infty), X\rangle}]}\\
& =  a_1^T \EE[XX^T] a_2(\infty) + O(|a_1| |a_2(\infty)|^2)\\
& = \frac{2|f(p)|^2}{p}\left[  \cos(|u-s_1| \log p) +  \cos(|u-s_2| \log p)\right] + O\left(\frac{|f(p)|^3}{p^{3/2}}\right)
\end{align*}
uniformly for $p \le y$, $s_1, s_2 \in \RR$ and $t \le 2$, and thus we obtain \eqref{eq:estimate-newMean} with the desired uniformity.
\end{proof}
\begin{lem}\label{lem:M_y}
Let $\OurEpsilon_{p, y} = p^{-1/2\log y} - 1$ and
\[ \Ma_{y}(u; s_1, s_2):= \sum_{p \le y} \frac{2|f(p)|^2\OurEpsilon_{p, y}}{p}\left[  \cos(|u-s_1| \log p) +  \cos(|u-s_2| \log p)\right]. \]
The following statements are true.
\begin{itemize}
\item $\Ma_{y}(u; s_1, s_2)$ is uniformly bounded in $y \ge 3$ and $s_1, s_2 \in \RR$.
\item $\Ma_{y}(u; s_1, s_2) \to 0$ as $y \to \infty$, uniformly for all $u, s_1, s_2 \in \RR$ satisfying $|u-s_1|, |u-s_2| \in [\delta, \delta^{-1}]$ for any fixed $\delta \in (0, 1)$.
\end{itemize}
\end{lem}
\begin{proof}
The claims follow from computations similar to those concerning $\log \left[C_{y, y}(s_1, s_2) / C_{y, \infty}(s_1, s_2)\right] = \sum_{p \le y} \log c_{p, y}(s_1, s_2)$ in the proof of \Cref{lem:tilt_ratios}, and we omit the details here.
\end{proof}
\subsection{Some probabilistic results and reduction}\label{sec:mc-prepare}
In this subsection we collect a few probabilistic results. We shall often speak of convergence in probability of abstract random variables living in the space of continuous functions $C(I)$ or measures $\Ma(I)$ on some interval $I\subset \RR$ - see \Cref{app:abstract} for a short summary and references therein for further details.

Our first probabilistic result is a small generalisation of \cite[Lemma 3.1]{sw}.
\begin{lem}\label{lem:randomFourier}
Let $I \subset \RR$ be a compact interval, and consider the sequence of random functions $F_n(\cdot) := \sum_{k=1}^n U_k f_{n, k}(\cdot)$ where
\begin{itemize}
\item $(U_k)_{k}$ are i.i.d.~(real- or complex-valued) random variables that are either (i) standard normal or (ii) symmetric and uniformly bounded (i.e.~there exists some constant $C$ such that $|U_k| \le C$ almost surely);
\item $f_{n, k} \xrightarrow[n \to \infty]{} f_k$ in $C^1(I)$ for each $k \ge 1$, and
$
\sum_{k=1}^\infty \sup_{n \ge 1} \left[ \| f_{n,k}\|_{\infty}^2 + \| f_{n,k}'\|_{\infty}^2\right] < \infty.
$
\end{itemize}
Then $\sup_{n \ge 1} \EE\left[\exp\left( \lambda \|F_n\|_{\infty}\right)\right] < \infty$ for any $\lambda > 0$, and $F_n(\cdot) \xrightarrow[n \to \infty]{p} F(\cdot) := \sum_{k \ge 1} U_k f_k(\cdot)$ in $C(I)$. 
\end{lem}
\begin{proof}
Suppose $(U_k)_k$ and $(f_{n, k})_{n, k}$ are all real-valued without loss of generality (otherwise we treat the real and imaginary parts separately) and set $f_{n, k} \equiv 0$ for $n > k$. By an affine transformation of coordinates it suffices to establish the claim for $I = [0,1]$. Let us recall the Sobolev inequality (see \Cref{thm:sobolev})
\[ \|f\|_{\infty} \ll \|f\|_{W^{1,2}([0,1])} := \left( \int_0^1 [|f(t)|^2 + |f'(t)|^2 ]dt\right)^{\frac{1}{2}}. \]
We start with the estimate for exponential moments. Pick $a > 0$, and consider
\begin{align}
\EE\left[\exp\left( \lambda \|F_n\|_{\infty}\right)\right] 
\notag
&\le\EE\left[\exp\left( \lambda \|F_n\|_{\infty}\right)\mathbf{1}_{\{\|F_n\|_\infty \le 4\lambda / a\}}\right] + \EE\left[\exp\left( \lambda \|F_n\|_{\infty}\right) \mathbf{1}_{\{\|F_n\|_\infty > 4\lambda / a\}}\right]  \\
\label{eq:mgf_estimate}
& \le e^{4\lambda^2/a} + \EE\left[ \exp \left(a\|F_n\|_{W^{1,2}([0,1])}^2\right)\right]
\le e^{4\lambda^2/a} + \int_0^1 \EE\left[ e^{a\left(|F_n(t)|^2 + |F_n'(t)|^2\right)}\right]dt.
\end{align}
By Gaussian tail estimates or Hoeffding's inequality \cite{Hoeffding} (if $U_k$'s are bounded), we can verify that
\[ \PP\left(|F_n(t)| \ge u\right) \le 2 e^{-u^2 / (C \sum_{k \ge 1} |f_{n, k}(t)|^2)}  
\qquad \text{and} \qquad 
\PP\left(|F_n'(t)| \ge u\right) \le 2 e^{-u^2 /  (C \sum_{k \ge 1} |f_{n, k}'(t)|^2)} \]
where $C \in (0, \infty)$ is independent of $n \ge 1$ and $u \ge 0$. In particular, our assumption on $(f_{n,k})_{n, k}$ implies
\[ \sup_{n \ge 1}\sup_{t \in [0,1]} \EE\left[ e^{a\left(|F_n(t)|^2 + |F_n'(t)|^2\right)}\right]
\le \sup_{n \ge 1} \sup_{t \in [0,1]} \sqrt{\EE\left[ e^{2a|F_n(t)|^2}\right]\EE\left[ e^{2a|F_n'(t)|^2}\right]} < \infty \]
for sufficiently small $a > 0$. Substituting this into \eqref{eq:mgf_estimate}, we obtain $\EE\left[\exp\left( \lambda \|F_n\|_{\infty}\right)\right] < \infty$ for any $\lambda > 0$.

We now study the convergence of $F_n$. Since $(U_k)_{k}$ are i.i.d.~variables with zero mean and finite variance,
\[ \EE\left[\|F_n - F\|_{W^{1,2}([0,1])}^2\right]
\ll  \sum_{k \ge 1} \left(\|f_{n, k} - f_k\|_\infty^2 + \|f'_{n, k} - f'_k\|_\infty^2\right) =: \sum_{k \ge 1} I_{n, k}. \]
Note that $I_{n, k} \xrightarrow{n \to \infty} 0$ for each $k \ge  1$, and $I_{n, k} \le 2\sup_{n \ge 1} \left[ \| f_{n,k}\|_{\infty}^2 + \| f_{n,k}'\|_{\infty}^2\right]$ which is summable. It follows from dominated convergence that $\lim_{n \to \infty} \sum_k I_{n, k} = 0$ , and by Markov's inequality
\[ \lim_{n \to \infty} \PP\left( \|F_n - F\|_\infty > \OurEpsilon \right)
\le \lim_{n \to \infty} \OurEpsilon^{-2}\EE\left[\|F_n - F\|_{\infty}^2\right]
\ll \lim_{n \to \infty} \OurEpsilon^{-2} \cdot  \EE\left[\|F_n - F\|_{W^{1,2}([0,1])}^2\right] = 0 \qquad \forall \OurEpsilon > 0. \]
\end{proof}
The next result concerns convergence of measures in the presence of additional continuous densities.
\begin{lem}\label{lem:add-cont-density}
Let $I \subset \RR$ be a compact interval, and suppose
\[    X_n \xrightarrow[n \to \infty]{p} X\quad  \text{in $C(I)$} \qquad \text{and} \qquad
    \nu_n \xrightarrow[n \to \infty]{p} \nu \quad\text{in $\Ma(I)$}, \]
then $X_n \nu_n \xrightarrow[n \to \infty]{p} X \nu$ in $\Ma(I)$.
\end{lem}
\begin{proof}
We will use the subsequential limit characterisation of convergence in probability and reduces the problem to a deterministic one. Let $(n_k)_k$ be any subsequence. By \Cref{lem:subseq_rv} and \Cref{lem:subseq_measure}, there exists a further subsequence $(n_{m_k})_k$ along which we have
\[    X_{n_{m_k}} \xrightarrow[k \to \infty]{a.s.} X \quad \text{in $C(I)$} \qquad \text{ and } \qquad
    \nu_{n_{m_k}} \xrightarrow[k \to \infty]{a.s.} \nu
    \quad \text{in $\Ma(I)$}, \]
i.e.~there exists an event $E$ with $\PP(E) = 1$ such that for any $\omega \in E$ we have
\[    \lim_{k \to \infty} X_{n_{m_k}}(\cdot; \omega) = X(\cdot; \omega) \qquad \text{and} \qquad 
    \lim_{k \to \infty} \nu_{n_{m_k}}(\cdot; \omega) = \nu(\cdot; \omega).\]
Let us fix $\omega \in E$. For any given function $f \in C(I)$, consider
\begin{align*}
&\left| \int_I f(x) X_{n_{m_k}}(x; \omega) \nu_{n_{m_k}}(dx; \omega) - \int_I f(x) X(x;\omega) \nu(dx; \omega)\right|\\
& \qquad \le
\left| \int_I f(x) [X_{n_{m_k}}(x; \omega) - X(x;\omega)] \nu_{n_{m_k}; \omega}(dx)\right| + \left|\int_I f(x) X(x; \omega) [\nu_{n_{m_k}}(dx; \omega)- \nu(dx;\omega)]\right|.
\end{align*}
Observe that:
\begin{itemize}
    \item The first term on the right-hand side is bounded by $\|f\|_{\infty} \|X_{n_{m_k}}(\cdot; \omega) - X(\cdot; \omega)\|_{\infty} \nu_{n_{m_k}}(I)$. This vanishes as $k \to \infty$ since $\nu_{n_{m_k}}(I; \omega) \to \nu(I; \omega) < \infty$ and $ \|X_{n_{m_k}}(\cdot; \omega) - X(\cdot; \omega)\|_{\infty} \to 0$.
    \item The second term on the right-hand side also vanishes as $k \to \infty$ since we can take $f(\cdot) X(\cdot; \omega) \in C(I)$ as test function and use the convergence of $\nu_{n_{m_k}}(\cdot; \omega) \to \nu(\cdot; \omega)$ in $\Ma(I)$.
\end{itemize}
Since $f \in C(I)$ and $\omega \in E$ are arbitrary, we have shown that $X_{n_{m_k}} \nu_{n_{m_k}} \xrightarrow[k \to \infty]{a.s.} X \nu$ in $\Ma(I)$, and the claim follows by another application of \Cref{lem:subseq_measure}.
\end{proof}
Let $y_0 = y_0(f)$ be sufficiently large such that  
\[ \left|\frac{f(p)}{p^{1/2}}\right|  + \left|\frac{f(p^2)}{p}\right| \le \frac{1}{2}
\qquad \text{and thus} \qquad \sup_{s \in \RR} \left|\sum_{k=1}^2 \frac{\alpha(p)^k f(p^k)  }{p^{k(1/2+is)}}\right| \le \frac{1}{2} \]
for any $p \ge y_0$ (which is possible by \eqref{eq:f-summability}), and that the conclusion in \Cref{lem:truncate} holds. Rewrite
\begin{equation}\label{eq:reduced-measure}
m_{y, t}(ds) = X_{y, t}(s) \nu_{y, t}(ds)
\end{equation}
where
\begin{equation}\label{eq:def-nu}
\nu_{y, t}(ds) := \frac{\exp\left(\Ga_{y, t}(s)\right)}{\EE\exp\left(\Ga_{y, t}(s)\right)} ds
\qquad \text{with} \qquad 
\Ga_{y, t}(s) = \sum_{p \le y} G_{p, t}(s), \quad G_{p, t}(s) := 2 \Re \frac{f(p) \alpha(p)}{p^{\sigma_t + is}}
\end{equation}
and $X_{y, t}(s) = \prod_{j=0}^3 X_{y, t}^{(j)}(s)$ where
\begin{align*}
X_{y, t}^{(0)}(s) & := \frac{\EE\exp \left(\Ga_{y, t}(s)\right)}{\EE\left[|A_y(\sigma_t + is)|^2 \right]},
\qquad X_{y, t}^{(1)}(s) := \exp\bigg(- \sum_{p \le y_0} G_{p, t}(s)\bigg)  |A_{y_0}(\sigma_t + is)|^2, \\
X_{y, t}^{(2)}(s) & :=\prod_{y_0 < p \le y}  \left|1 + \sum_{k =1}^2 \frac{\alpha(p)^k f(p^k)  }{p^{k(\sigma_t + is)}}\right|^{-2}\left|1 + \sum_{k \ge 1} \frac{\alpha(p)^k f(p^k)  }{p^{k(\sigma_t+ is)}}\right|^2,  \\
X_{y, t}^{(3)}(s) & :=\prod_{y_0 < p \le y} \exp\left(-G_{p, y}(s)\right) \left|1 + \sum_{k =1}^2 \frac{\alpha(p)^k f(p^k)  }{p^{k(1/2 + is)}}\right|^{2}.
\end{align*}
We claim that:
\begin{lem}\label{lem:cont-density}
For any $r > 0$, 
\[ \sup_{y, t \ge 3} \EE\left[\|X_{y, t}\|_\infty^r\right] < \infty. \]
Moreover, both $X_{y, \infty}$ and $X_{\infty, y}$ converge in probability to some continuous field $X_\infty$ in $C(I)$ as $y \to \infty$.
\end{lem}
\begin{proof}
It suffices to verify the analogous claim for $X_{y, t}^{(j)}$ for $j = 0, \dots, 3$ which we perform below.
\begin{itemize}
\item $X_{y, t}^{(0)}(s)$ is just a sequence of numbers that are independent of $s \in I$ because of the rotational invariance of $\alpha(p)$. For this reason let us just set $s = 0$ and estimate
\[ \frac{\EE\exp \left(\Ga_{y, t}(s)\right)}{\EE\left[|A_y(\sigma_t + is)|^2 \right]}
=\prod_{p \le y} \frac{\EE\left[ \exp G_{p, t}(0)\right]}{\EE\left[\left|1 + \sum_{k\ge1} \frac{\alpha(p)^k f(p^k)  }{p^{k\sigma_t}}\right|^2\right]}. \]
Note that for each fixed $p$,
\begin{align*}
\lim_{t \to \infty}\EE\left[\exp\left(G_{p, t}(0)\right) \right]  & = \EE\left[\exp\left(G_{p, \infty}(0)\right) \right] \\
 \text{and} \quad 
\lim_{t \to \infty} \EE\left[\left|1 + \sum_{k\ge1} \frac{\alpha(p)^k f(p^k)  }{p^{k\sigma_t}}\right|^2\right]
& = \EE\left[\left|1 + \sum_{k\ge1} \frac{\alpha(p)^k f(p^k)  }{p^{k/2}}\right|^2\right].
\end{align*}
Also, for $p \le y$ we have
\[ \EE\left[ \exp\left(G_{p, t}(0)\right)\right] 
 =  \exp\left(\frac{1}{2} \EE[G_{p, t}(0)^2] + O\left( \left| \frac{f(p)}{p^{\sigma_t}}\right|^3\right) \right)
= \exp\left( \frac{|f(p)|^2}{p^{2\sigma_t}} + O\left(\frac{|f(p)|^3}{p^{3/2}}\right) \right) \]
by \Cref{lem:gf-estimates}, and 
\[ \EE\left[\left|1 + \sum_{k\ge1} \frac{\alpha(p)^k f(p^k)  }{p^{k\sigma_t}}\right|^2\right]
\ge 1 + \frac{|f(p)|^2}{p^{2\sigma_t}} 
= \exp \left(  \frac{|f(p)|^2}{p^{2\sigma_t}} +  O\left(\frac{|f(p)|^4}{p^{2}}\right)\right). \]
Therefore,
\[ \log \left[\frac{\EE\left[ \exp G_{p, t}(0)\right]}{\EE\left[\left|1 + \sum_{k\ge1} \frac{\alpha(p)^k f(p^k)  }{p^{k\sigma_t}}\right|^2\right]}\right]
\ll \frac{|f(p)|^3}{p^{3/2}} \]
is summable in $p$ with an upper bound uniform in $y, t \ge 3$, and by dominated convergence
\[ X_{y, \infty}^{(0)}(0), ~X_{\infty, y}^{(0)}(0) \xrightarrow[y \to \infty]{} \prod_{p} \frac{\EE\left[ \exp G_{p, \infty}(0)\right]}{\EE\left[\left|1 + \sum_{k\ge1} \frac{\alpha(p)^k f(p^k)  }{p^{k/2}}\right|^2\right]}. \]
\item $X_{y, t}^{(1)}$ is a finite product of random functions that are bounded deterministically and uniformly in $y, t \ge 3$, and each of them is converging to some random continuous functions almost surely. Thus $X_{y, t}^{(1)}$ converges almost surely and in $L^r$ for any $r > 1$ (and hence in probability) to $X_{\infty, \infty}^{(1)}$ in $C(I)$.
\item The analysis of $X_{y, t}^{(2)}$ follows essentially that of $X_{y, t}^{(0)}$. Using \Cref{lem:truncate}, we have
\[ \log \left[\left|1 + \sum_{k =1}^2 \frac{\alpha(p)^k f(p^k)  }{p^{k(\sigma_t+ is)}}\right|^{-2}\left|1 + \sum_{k \ge 1} \frac{\alpha(p)^k f(p^k)  }{p^{k(\sigma_t + is)}}\right|^2\right]
\ll  \left|\frac{f(p)}{p^{1/2}}\right|^3 + \left|\frac{f(p^2)}{p}\right|^3 +  \sum_{k \ge 3} \left|\frac{f(p^k)}{p^{k/2}}\right| \]
which is uniformly summable in $p$. In particular, $\|X_{y, t}^{(2)}\|_\infty$ is bounded deterministically and uniformly in $y, t \ge 3$, and by dominated convergence we also have
\[ X_{y, \infty}^{(2)}(s), ~X_{\infty, y}^{(2)} (s)\xrightarrow[y \to \infty]{a.s.} \prod_{p} \left|1 + \sum_{k =1}^2 \frac{\alpha(p)^k f(p^k)  }{p^{k(1/2+ is)}}\right|^{-2}\left|1 + \sum_{k \ge 1} \frac{\alpha(p)^k f(p^k)  }{p^{k(1/2 + is)}}\right|^2 \qquad \text{in $C(I)$}. \]
\item For $X_{y, t}^{(3)}$, note that
\begin{align*}
&\left|1 + \sum_{k=1}^2 \frac{\alpha(p)^k f(p^k)  }{p^{k(\sigma_t+is)}}\right|^2
 = \exp \left\{
2\Re \left(\sum_{j=1}^2 \frac{(-1)^{j-1}}{j} \left(\sum_{k=1}^2 \frac{\alpha(p)^k f(p^k)  }{p^{k(\sigma_t+is)}}\right)^j\right) + O\left(\frac{|f(p)|^3}{p^{3/2}} + \frac{|f(p^2)|^3}{p^3} \right)
 \right\}\\
&\quad  = \exp \left\{ \Re \left[ 2\frac{\alpha(p) f(p)}{p^{\sigma_t + is}} + \frac{\alpha(p)^2 (2f(p^2) - f(p)^2)}{p^{2(\sigma_t +is)}} - 2\frac{\alpha(p)^3f(p)f(p^2)}{p^{3(\sigma_t + is)}}\right] + O\left(\frac{|f(p)|^3}{p^{3/2}} + \frac{|f(p^2)|^2}{p^2}\right)\right\} 
\end{align*}
where the implicit constant is uniform in $p \ge y_0$. It is straightforward to check that the ``error fields" in $O\left(\frac{|f(p)|^3}{p^{3/2}} + \frac{|f(p^2)|^2}{p^2}\right)$ are summable with deterministic bound uniformly in $y, t \ge 3$, and that they converge almost surely to some continuous field using dominated convergence as before.

To conclude our proof, it suffices to check that the field
\[ \Delta_{y, t}(s):=\Re \sum_{y_0 < p \le y} \left[ \frac{\alpha(p)^2 (2f(p^2) - f(p)^2)}{p^{2(\sigma_t + is)}} - 2\frac{\alpha(p)^3f(p)f(p^2)}{p^{3 (\sigma_t +is)}}\right] \]
has uniformly bounded exponential moments of all orders and that $\Delta_{y, \infty}, \Delta_{\infty, y}$ converge in probability in $C(I)$. But \eqref{eq:f-summability} implies
\[ \sum_{y_0 < p}\left\{ \frac{\left[|f(p^2)|^2 + |f(p)|^4\right]\log^2 p}{p^2} + \frac{ |f(p)|^2 |f(p^2)|^2 \log^2 p}{p^3}\right\} < \infty, \]
and our desired claims follow by an application of \Cref{lem:randomFourier}.
\end{itemize}
\end{proof}
In view of \Cref{lem:add-cont-density} and \Cref{lem:cont-density},  \Cref{thm:mc-L1} holds if we can prove the analogous statement for the random measures $\nu_{y, t}$. This is the content of the following lemma.
\begin{lem}\label{lem:reduced-mc}
Under the same setting as \Cref{thm:mc-L1}, there exists a non-trivial random Radon measure $\nu_\infty(ds)$ on $\RR$  such that the following are true: for any bounded interval $I$ and any test function $h \in C(I)$, we have
\begin{equation}\label{eq:reduced-mc-conv}
   (i) \quad \nu_{y, \infty}(h) \xrightarrow[y \to \infty]{L^r} \nu_\infty(h)
    \qquad \text{and} \qquad
  (ii) \quad   \nu_{\infty, t}(h) \xrightarrow[t \to \infty]{L^r} \nu_\infty(h)
\end{equation}
for any $r \in [1, 1/\theta)$. In particular, the above convergence also holds in probability. Moreover,  the limiting measure $\nu_\infty(ds)$ is supported on $\RR$ and non-atomic almost surely.
\end{lem}
Before we move on to various estimates needed for the proof of \Cref{lem:reduced-mc}, let us record the following result for tensorising convergence of random measures. Ultimately this will be used in our modified second moment method.
\begin{lem}\label{lem:tensorise}
Let $I \subset \RR$ be any compact interval. Suppose $(\nu_n)_n, \nu_\infty$ are random Radon measures on $I$ such that $\nu_n \xrightarrow[n \to \infty]{p} \nu_\infty$, then the following are true:
\begin{itemize}
\item[(i)] $\nu_n^{\otimes2} \xrightarrow[n \to \infty]{p} \nu_\infty^{\otimes 2}$ in $\Ma(I \times I)$.
\item[(ii)] Assume the random measure $\nu_\infty$ is non-atomic. Let $K_n(s_1, s_2)$ be a sequence of deterministic continuous functions on $I \times I$ such that  
\[ \sup_{n} \|K_n\|_\infty < \infty
\qquad \text{and} \qquad
K_n(s_1, s_2) \to 0 \qquad \text{uniformly for any $\delta > 0$ and $|s_1 - s_2| > \delta$.} \]
Then $K_n \nu_n^{\otimes 2} \xrightarrow[n \to \infty]{p} 0$ in $\Ma(I \times I)$.
\end{itemize}
\end{lem}
\begin{proof}
By the subsequential limit characterisation of convergence in probability (\Cref{lem:subseq_measure}), we may assume without loss of generality that $\nu_n \to \nu_\infty$ almost surely. In particular, it suffices to establish these claims when $\nu_n, \nu$ are all non-random measures (similar to the proof of \Cref{lem:add-cont-density}). For notational convenience, let us take $I = [0,1]$.

For (i), we want to show that 
\[ \int_{I \times I} h(s_1, s_2) \nu_n(ds_1) \nu_n(ds_2)
\xrightarrow[n \to \infty]{} \int_{I \times I} h(s_1, s_2) \nu_\infty(ds_1) \nu_\infty(ds_2)
\qquad \forall h \in C(I \times I). \]
If $h(s_1, s_2) = h_1(s_1) h_2(s_2)$ where $h_i \in C(I)$, the claim is immediate because $\nu_n(h_1) \nu_n(h_2) \to \nu_\infty(h_1) \nu_\infty(h_2)$. The general case then follows from the fact that the space of functions spanned by $h_1(s_1)h_2(s_2)$ for $h_i \in C(I)$ is dense in $C(I \times I)$ by e.g.~Stone--Weierstrass.

As for (ii), assume without loss of generality that $\sup_n \|K_n\|_\infty \le 1$. Let us take some non-negative function $h_0 \in C(\RR)$ with the property that $\|h_0\|_\infty \le 1$, $h_0|_{[0,1/2]} \equiv 1$ and $\mathrm{supp}(h_0) = [0,1]$. Then for any fixed $\delta > 0$, 
\begin{align*}
& \limsup_{n \to \infty} \int_{I \times I} |K_n(s_1, s_2)| \nu_n(ds_1) \nu_n(ds_2)\\
& \le \limsup_{n \to \infty} \int_{I \times I} h_0(|s_1 - s_2| / \delta) \nu_n(ds_1) \nu_n(ds_2)
+ \limsup_{n \to \infty} \int_{I \times I} \left[1-h_0(|s_1 - s_2| / \delta)\right] |K_n(s_1, s_2)| \nu_n(ds_1) \nu_n(ds_2)\\
& \le \int_{I \times I} h_0(|s_1 - s_2| / \delta) \nu_\infty(ds_1) \nu_\infty(ds_2)
\end{align*}
by weak convergence of measures, and the uniform convergence of $K_n(s_1, s_2) \to 0$ for $|s_1 - s_2| \ge 2\delta$. Now take $\delta = 2^{-m}$, then
\begin{align*}
\nu_\infty^{\otimes 2}\left(\{s_1, s_2 \in I: |s_1 - s_2| \le 2^{-m}\}\right)
& \ll \sum_{j=0}^{2^m - 1} \nu_\infty([j2^{-m}, (j+1)2^{-m}])^2 \\
& \le \nu_\infty([0,1]) \max_{j \le 2^{m} - 1}\nu_\infty \left( [j2^{-m}, (j+1)2^{-m}]\right) \xrightarrow[m \to \infty]{} 0
\end{align*}
by the non-atomicity of $\nu_\infty$. This means $h_0(|s_1 - s_2| /\delta) = h_0(2^m|s_1 - s_2|)$ converges $\nu^{\otimes 2}$-almost everywhere to $0$ and by dominated convergence we obtain
\[ \limsup_{m \to \infty}\int_{I \times I} h_0(2^m|s_1 - s_2| ) \nu_\infty(ds_1) \nu_\infty(ds_2) = 0 \]
which concludes our proof.
\end{proof}

\subsection{Multifractal estimates and uniform integrability} \label{sec:mc-moment}
Before explaining the proof of \Cref{lem:reduced-mc}, we still need to obtain a few more moment estimates. The first one describes the multifractal scaling of $\nu_{y, t}$.
\begin{lem}\label{lem:multifractal}
Let $\delta := \log^{-1}  \min(y, t)$. For any fixed $n, \lambda > 0$,  we have
\[    C(n, \lambda) := \sup_{y, t \ge 3} \EE\left[ \exp \left( \lambda \sup_{s_1, s_2 \in [0, n \delta]} |\Ga_{y, t}(s_1) - \Ga_{y, t}(s_2)|\right)\right] < \infty,\]
and hence
\[    \EE\left[ \exp \left(\lambda \sup_{s \in [0, n \delta]} \Ga_{y, t}(s)\right)\right] \le C\left(n, \frac{\lambda\lambda'}{\lambda' - \lambda}\right)^{1-\lambda/\lambda'} \EE\left[ \exp \left(\lambda' \Ga_{y, t}(0)\right)\right]^{\lambda / \lambda'} \]
for any $\lambda' > \lambda$. In particular, for any $1 \le q < q'$, we have
\[\EE\left[ \nu_{y, t}([0, \OurEpsilon])^q\right] \ll_{q, q'} \OurEpsilon^q \delta^{-\theta q(q'-1)}\]
uniformly for all $y, t \ge 3$ and $0 \le \OurEpsilon \le \delta$.
\end{lem}
\begin{proof}
Recall
\[ \Ga_{y, t}(\delta s_1) - \Ga_{y, t}(\delta s_2)
 = 2 \Re \sum_{p \le y} \frac{f(p)\alpha(p)}{p^{\sigma_t}} \left(p^{-i \delta s_1} - p^{-i\delta s_2}\right) \]
where
\begin{align*}
 \sum_{p \le y} & \left|2 \Re \left[ \frac{f(p)\alpha(p)}{p^{\sigma_t}} \left(p^{-i \delta  s_1} - p^{-i\delta s_2}\right)\right]\right|^2
 \le 4\delta^2 |s_1 - s_2|^2 \sum_{p \le y} \frac{|f(p)|^2}{p^{1 + 1/\log t}} \log^2 p\\
& \ll \delta^2 |s_1 - s_2|^2 \int_{2^-}^{y^+} \frac{\log^2 p}{p^{1 + 1/\log t}} d\left[\theta \mathrm{Li}(p) + \Ea_{|f|^2}(p)\right]\\
& =  \delta^2 |s_1 - s_2|^2 \left\{ \theta \int_{\log 2}^{\log y} u e^{- u / \log t} du + \left[\frac{\log^2 p}{p^{1 + 1/\log t}}\Ea_{|f|^2}(p)\right]_{2^-}^{y^+} + \int_2^y \frac{\Ea_{|f|^2}(p)}{p^2} \left[\frac{\log^2 p - 2\log p}{p^{1/\log t}} \right] dp \right\}\\
& \le \delta^2 |s_1 - s_2|^2 \left\{
\theta \log^2 \min(y, t) + \frac{\log y}{y^{1/\log t}} \frac{\Ea_{|f|^2(y)}}{y / \log y} + \left(\max_{p \le y} \frac{\log^2 p}{p^{1/\log t}}\right) \int_2^\infty  p^{-2}|\Ea_{|f|^2}(p)|dp\right\}\\
& \ll |s_1 - s_2|^2
\end{align*}
since
\[ \frac{\log y}{y^{1/\log t}} \le \delta^{-1}
\qquad \text{and} \qquad
\max_{p \le y} \frac{\log^2 p}{p^{1/\log t}} \le  2 \delta^{-2}. \]
Therefore, by Hoeffding's inequality \cite{Hoeffding} we have
\[ \PP\left(|\Ga_{y, t}(\delta s_1) - \Ga_{y, t}(\delta s_2)| > u \right)
\le 2 \exp\left( - cu^2 / |s_1 - s_2|^2\right) \qquad \forall s_1, s_2 \in [0,n] \]
for some $c \in (0, \infty)$ uniformly in $y, t \ge 3$. By generic chaining bound (see \Cref{thm:chaining} and \Cref{rem:generic-chaining}), we have
\[ \PP\left(\sup_{s_1, s_2 \in [0, n]}|\Ga_{y, t}(\delta s_1) - \Ga_{y, t}(\delta s_2)| > u \right)
\le 2 \exp\left( - c'u^2\right)  \]
for some $c' \in (0, \infty)$ depending only on the previous constant $c$. Thus, for any $n, \lambda > 0$ 
\[\EE\left[ \exp \left( \lambda \sup_{s_1, s_2 \in [0, n \delta]} |\Ga_{y, t}(s_1) - \Ga_{y, t}(s_2)|\right)\right] = \EE\left[ \exp \left( \lambda \sup_{s_1, s_2 \in [0, n]} |\Ga_{y, t}(\delta s_1) - \Ga_{y, t}(\delta s_2)|\right)\right]
\ll\sum_{k \ge 0} e^{\lambda k} e^{-c' k^2}  \]
is uniformly bounded in $y, t \ge 3$.

Finally, recall by \Cref{lem:new-martingaleL2est} and \Cref{lem:gf-estimates} that 
\begin{equation}\label{eq:mgf_actual}
\EE\left[\exp\left(\lambda \Ga_{y, t}(0)\right)\right]
\asymp \exp \left(\frac{\lambda}{2} \sum_{p \le y} \EE\left[G_{p, t}(0)^2\right]\right)
\asymp \left(\log \min(y, t)\right)^{\theta \lambda^2}
= \delta^{-\theta\lambda^2} \qquad \forall \lambda > 0.
\end{equation} 
Since
\[\nu_{y, t}([0, \OurEpsilon])
 = \OurEpsilon \int_0^1  \frac{\exp\left(\Ga_{y, t}(\OurEpsilon s)\right)}{\EE\exp\left(\Ga_{y, t}(\OurEpsilon s)\right)} ds
 \le 
\OurEpsilon \frac{\exp\left(\sup_{s \in [0,\delta]}\Ga_{y, t}(s)\right)}{\EE\exp\left(\Ga_{y, t}(0)\right)},\]
our earlier bound implies
\[ \EE\left[ \nu_{y, t}([0, \OurEpsilon])^q\right]  \ll_{q, q'} \OurEpsilon^q\frac{\EE\left[\exp\left(q'\Ga_{y, t}(0)\right)\right]^{q/q'}}{\EE\left[\exp\left(\Ga_{y, t}(0)\right)\right]^{q}}
\ll_{q, q'} \OurEpsilon^q \delta^{-\theta q(q'-1)} \]
which concludes our proof.
\end{proof}
\begin{lem}\label{lem:ui-L1}
Let $\delta := \log^{-1} \min(y, t)$. For any $1 \le q < q' < \min(1/\theta, 2)$, we have
\begin{equation}\label{eq:multifractal_nu}
    \sup_{y, t \ge 3} \EE[\nu_{y, t}(I)^q] \ll_{q, q'}  r^{q - \theta q(q'-1)}
\end{equation}
uniformly for all interval $I \subset \RR$ of length $r \in [\delta, r_0]$ for any fixed $r_0 \in (0, \infty)$. In particular, with the same assumption on $(q, q')$ we have
\begin{equation}\label{eq:multifractal_m}
    \sup_{y, t \ge 3} \EE[m_{y, t}(I)^q] \ll_{q, q'} r^{q - \theta q(q'-1)}
\end{equation}
uniformly for all interval $I \subset \RR$ of length $r \in [\delta, r_0]$ for any fixed $r_0 \in (0, \infty)$.
\end{lem}
Recall that we focus on the $L^1$-regime $\theta \in [1/2, 1)$ where $\min(1/\theta, 2) = 1/\theta$. The utility of the above lemma is to reduce the problem of $L^r$-convergence to convergence in probability of $\nu_{y, \infty}(h)$ and $\nu_{\infty, t}(h)$ (resp.~$m_{y, \infty}(h)$ and $m_{\infty, t}(h)$) to the same random variable $\nu_\infty(h)$ (resp.~$m_\infty(h)$) as $y, t \to \infty$ based on standard argument of uniform integrability.
\begin{proof}[Proof of \Cref{lem:ui-L1}]
We shall only focus on the moment estimates \eqref{eq:multifractal_nu} for $\nu_{y, t}$ below since the proof of \eqref{eq:multifractal_m} is similar. Without loss of generality, we consider $r_0 \le \frac{1}{100}$ and $I = [0, r]$.

Let $m \ge 2$ be some fixed integer to be chosen later, $\ell \in [\delta, \frac{1}{100}]$, $y_0(\ell):= \exp(\ell^{-1})$ and $I_1, I_2 \subset [0, 2^m\ell]$ be two intervals of length at most $\ell$ and distance at least $\ell$ away from each other. For any $s_1 \in I_1, s_2 \in I_2$, we have
\[\Ga_{y, t}(s_1) + \Ga_{y, t}(s_2) = \sum_{p \le y} \left[ G_{p, t}(s_1) + G_{p, t}(s_2)\right]
\le 2\sup_{s \in [0, 2^m\ell]} \Ga_{y_0(\ell), t}(s) + \sum_{y_0(\ell) < p \le y} \left[ G_{p, t}(s_1) + G_{p, t}(s_2)\right] \]
and thus
\begin{align*}
& \EE\left[ \left( \nu_{y, t}(I_1)\nu_{y, t}(I_2)\right)^{q/2}\right]
= \EE\left[ \left(\iint_{I_1 \times I_2} \frac{\exp\left(\Ga_{y, t}(s_1) + \Ga_{y, t}(s_2)\right)}{\EE\left[\exp\left(\Ga_{y, t}(0)\right)\right]^2}ds_1 ds_2\right)^{q/2}\right]\\
& \le \frac{\EE\left[\exp\left(q\sup_{s \in [0, 2^m\ell]} \Ga_{y_0(\ell), t}(s)\right)\right]}{\EE\left[\exp\left(\Ga_{y_0(\ell), t}(0)\right)\right]^q}
\EE\left[ \left(\iint_{I_1 \times I_2} \prod_{y_0(\ell) < p \le y}\frac{\exp\left(G_{p, t}(s_1) +G_{p, t}(s_2)\right)}{\EE\left[\exp\left(G_{p, t}(0)\right)\right]^2}ds_1 ds_2\right)^{q/2}\right]
\end{align*}
by independence. The first ratio on the right-hand side is
\[    \ll_{q, q'}\frac{\EE\left[\exp\left(q' \Ga_{y_0(\ell), t}(0)\right)\right]^{q/q'}}{\EE\left[\exp\left(\Ga_{y_0(\ell), t}(0)\right)\right]^q}
    \ll_{q, q'} \ell^{-\theta q (q'-1)} \]
by \Cref{lem:multifractal} and \eqref{eq:mgf_actual}. Meanwhile, from \Cref{lem:gf-estimates} and \Cref{lem:new-martingaleL2est} we see that
\begin{align*}
& \sum_{y_0(\ell) < p \le y} \log \frac{\EE\left[\exp\left(G_{p, t}(s_1) +G_{p, t}(s_2)\right)\right]}{\EE\left[\exp\left(G_{p, t}(0)\right)\right]^2}
= 2\sum_{y_0(\ell) < p \le y}\frac{|f(p)|^2 \cos(|s_1 - s_2| \log p)}{p^{2\sigma_t}} + O(1)\\
& \qquad = 2\sum_{ p \le y}\frac{|f(p)|^2 \cos(|s_1 - s_2| \log p)}{p^{2\sigma_t}}
- 2\sum_{p \le y_0(\ell)}\frac{|f(p)|^2 \cos(|s_1 - s_2| \log p)}{p^{2\sigma_t}} + O(1)\\
& \qquad = 2\theta \log\left[\min \left(|s_1 - s_2|^{-1}, \log y, \log t\right)\right] - 2\theta \log\left[\min \left(|s_1 - s_2|^{-1}, \log y_0(\ell), \log t\right)\right]  + O(1)
\end{align*}
is uniformly bounded for $|s_1 - s_2| \ge \ell$. By H\"older's inequality,
\begin{align*}
& \EE\left[ \left(\iint_{I_1 \times I_2} \prod_{y_0(\ell) < p \le y}\frac{\exp\left(G_{p, t}(s_1) +G_{p, t}(s_2)\right)}{\EE\left[\exp\left(G_{p, t}(0)\right)\right]^2}ds_1 ds_2\right)^{q/2}\right] \\
& \qquad \le \left(\iint_{I_1 \times I_2} \prod_{y_0(\ell) < p \le y}\frac{\EE\left[\exp\left(G_{p, t}(s_1) +G_{p, t}(s_2)\right)\right]}{\EE\left[\exp\left(G_{p, t}(0)\right)\right]^2}ds_1 ds_2\right)^{q/2}
 \ll_q \ell^q
\end{align*}
and hence there exists some constant $C(q, q') \in (0, \infty)$ such that
\begin{equation}\label{eq:cross_fractional}
\EE\left[ \left( \nu_{y, t}(I_1)\nu_{y, t}(I_2)\right)^{q/2}\right] \le C(q, q') \ell^{q-\theta q(q'-1)}
\end{equation}
uniform in $y, t \ge 3$ and any pairs of intervals $I_1, I_2$ under the conditions explained at the start.

Let us write $I_{j}(r) := [jr, (j+1) r]$. We now split our interval into $2^m$ sub-intervals of equal length and group them according to parity, i.e.
\begin{align*}
\EE\left[ \nu_{y, t}([0, r])^q \right]
& = \EE\left[ \left(\sum_{j=0}^{2^m - 1}\nu_{y, t}(I_{j}(r2^{-m}))\right)^{q} \right]\\
& \le 2^{q-1} \left\{\EE\left[ \left(\sum_{j=0}^{2^{m-1} - 1}\nu_{y, t}(I_{2j}(r2^{-m}))\right)^{q} \right]
+\EE\left[ \left(\sum_{j=0}^{2^{m-1} - 1}\nu_{y, t}(I_{2j+1}(r2^{-m}))\right)^{q} \right]\right\}
\\
& = 2^{q} \EE\left[ \left(\sum_{j=0}^{2^{m-1} - 1}\nu_{y, t}(I_{2j}(r2^{-m}))\right)^{q} \right]
\end{align*}
where the last equality follows from the translation invariance of $\nu_{y, t}$. Observe that the intervals appearing on the last line are all at least $r2^{-m}$ away from each other. Using sub-additivity of $x \mapsto x^{q/2}$ (since $q/2 \in [0, 1]$), this can be further upper-bounded by
\begin{align*}
& 2^q \EE\left[ \left(\sum_{j, k=0}^{2^{m-1} - 1}\nu_{y, t}(I_{2j}(r2^{-m}))\nu_{y, t}(I_{2k}(r2^{-m}))\right)^{q/2} \right]\\
& \le 2^q \left\{\sum_{j=0}^{2^{m-1} - 1} \EE\left[\nu_{y, t}(I_{2j}(r2^{-m}))^q\right] 
+2 \sum_{1 \le j < k \le 2^{m-1} - 1} \EE\left[\left(\nu_{y, t}(I_{2j}(r2^{-m}))\nu_{y, t}(I_{2k}(r2^{-m}))\right)^{q/2}\right] 
\right\}\\
& \le 2^q \left[ 2^{m-1} \EE\left[\nu_{y, t}([0, r2^{-m}])^q\right]  + 2^{2(m-1)} \max_{1 \le j < k \le 2^{m-1} - 1}\EE\left[\left(\nu_{y, t}(I_{2j}(r2^{-m}))\nu_{y, t}(I_{2k}(r2^{-m}))\right)^{q/2}\right] \right]\\
& =  2^{q + m-1}\EE\left[\nu_{y, t}([0, r2^{-m}])^q\right] + O_{q, q'}(r^{q - \theta q(q'-1)})
\end{align*}
using the estimate \eqref{eq:cross_fractional} established earlier. Iterating this $n$ times with $r2^{-nm} \le \delta < r 2^{-(n-1)m}$, we combine the estimate from \Cref{lem:multifractal} to deduce
\begin{align*}
\EE\left[ \nu_{y, t}([0, r])^q\right]
&\le 2^{n(q+m-1)} \EE\left[ \nu_{y, t}([0, r2^{-nm}])^q\right]
+ O_{q, q'}\left(\sum_{j=0}^{n-1} 2^{j(q+m-1)} (r2^{-jm})^{q - \theta q(q'-1)} \right)\\
& \ll_{q, q'} \sum_{j=0}^{n} 2^{j(q+m-1)} (r2^{-jm})^{q - \theta q(q'-1)}
= r^{q - \theta q(q'-1)}\sum_{j=0}^{n} 2^{j \{q-1 - m[(q-1) - \theta q(q'-1)]\}}.
\end{align*}
Since $\theta q \in (0, 1)$, we can choose $q' > q$ sufficiently close to $q$ such that $(q-1) - \theta q(q'-1) > 0$, and $m$ sufficiently large such that $q-1 - m[(q-1) - \theta q(q'-1)] < 0$ to ensure the boundedness of the geometric sum as $n \to \infty$. Finally, the exponent of $r$ in \eqref{eq:multifractal_nu} gets worse when $q'$ is larger, and this means the statement holds without the need to impose the extra condition $(q-1) - \theta q(q'-1) > 0$. This concludes our proof.
\end{proof}

\subsection{Martingale construction and support properties} \label{sec:mc-martingale-support}
It is straightforward to check that for any prime number $p_0 \le  y$, 
\[ \EE\left[ \nu_{y, \infty}(h) \bigg| \Fa_{p_0}\right]
 = \int_I h(s) \EE\left[ \frac{\exp\left(\Ga_{y, \infty}(s)\right)}{\EE\left[\exp\left(\Ga_{y, \infty}(s)\right)\right]} \bigg| \Fa_{p_0}\right] ds 
 =  \int_I h(s) \frac{\exp\left(\Ga_{p_0, \infty}(s)\right)}{\EE\left[\exp\left(\Ga_{p_0, \infty}(s)\right)\right]} ds = \nu_{p_0, \infty}(h), \]
i.e.~$\left(\nu_{y, \infty}(h)\right)_{y}$ is a martingale with respect to the filtration $(\Fa_y)_y$. By the almost sure martingale convergence theorem, we conclude that $\nu_{y, \infty}(h) \xrightarrow[y \to \infty]{a.s.} \nu_{\infty}(h)$. By \Cref{lem:ui-L1}, the convergence also holds in $L^r$ for any $r \in [1, 1/\theta)$, and in particular the limit $\nu_\infty(h)$ is non-trivial with probability $1$ as long as $h \not \equiv 0$ since uniform integrability implies $\EE[\nu_\infty(h)] = \lim_{y \to \infty} \EE[\nu_{y, \infty}(h)] = \int_I h(u)du$. The same argument applies to the convergence $m_{y, \infty}(h) \xrightarrow[y \to \infty]{a.s., ~ L^r} m_\infty(h)$.

The fact that the support of the limiting measure $\nu_\infty$ (and $m_\infty$) is the entire real line follows from an argument based on Kolmogorov 0--1 law (see Step 4 in \cite[Section 4.1.2]{GW} for the details). For future reference, we now state the non-atomicity of the limiting measure in the following lemma.
\begin{lem}\label{lem:non-atomic}
The limiting measures $\nu_\infty$ and $m_\infty$ are almost surely non-atomic.
\end{lem}
\begin{proof}
It suffices to verify the claim for $\nu_\infty$ on any compact interval $I$, and without loss of generality take $I = [0,1]$ for notational convenience. We first remark that $\EE\left[\nu_\infty(\Sa)\right] = 0$ for any fixed finite/countable set $\Sa \subset I$ (and in particular the set of dyadic points in $[0,1]$). Now, observe that
\[ Q_m := \sum_{j=0}^{2^m-1} \nu_{\infty}\left([j2^{-m}, (j+1)2^{-m}]\right)^2 \]
is a non-increasing sequence in $m \in \NN$ so that $Q_\infty := \lim_{m \to \infty} Q_m$ exists almost surely. Since $\nu_\infty$ is non-atomic if and only if $Q_\infty = 0$, it suffices to check that $\lim_{m \to \infty} \PP(Q_m > u) = 0$ for any $u > 0$. But for any $q \in (1, \min(2, 1/\theta))$, 
\[ \PP(Q_m > u)
\le u^{-q/2} \EE \left[Q_m^{q/2}\right]
\le u^{-q/2} \sum_{j=0}^{2^m - 1} \EE\left[\nu_{\infty}\left([j2^{-m}, (j+1)2^{-m}]\right)^q\right]
= u^{-q/2} 2^m \EE\left[\nu_{\infty}\left([0, 2^{-m}]\right)^q\right] \]
by Markov's inequality and the concavity of $x \mapsto x^{q/2}$. By \Cref{lem:ui-L1},
\[ 2^m \EE\left[\nu_{\infty}\left([0, 2^{-m}]\right)^q\right]
\ll_{q, q'} 2^{-m \left[(q-1) - \theta q (q'-1) \right] } \]
where the exponent $(q-1) - \theta q (q'-1) $ can be made negative by choosing $q'$ sufficiently close to $q$ since $\theta q < 1$, and this concludes our proof.
\end{proof}
\subsection{Approximation away from the critical line} \label{sec:approximate-away-critical}
\subsubsection{Modified second moment method}
Our remaining task is to verify \eqref{eq:reduced-mc-conv}(ii), which will be split into two steps as follows.
\begin{lem}\label{lem:modified-2nd}
Let $h \in C(I)$ be a non-negative function. The following are true.
\begin{itemize}
\item[(i)] We have
\[\nu_{y, \infty}(h) - \nu_{y, y}(h) \xrightarrow[y \to \infty]{p} 0.\]
In particular, $\nu_{y, y}(h) \to \nu_\infty(h)$ in probability.
\item[(ii)] We have
\begin{equation}\label{eq:reduced-mc-2}
\nu_{ \infty, y}(h) - \nu_{y, y}(h) \xrightarrow[y \to \infty]{p} 0.
\end{equation}
In particular, $\nu_{y, \infty}(h) \to \nu_\infty(h)$ in probability.
\end{itemize}
\end{lem}
When $\theta \in (0, \tfrac{1}{2})$, i.e.~$L^2$-regime of multiplicative chaos, the above claims could be established easily by showing $\EE\left[ |\nu_{y, \infty}(h) - \nu_{\infty, y}(h)|^2\right] \to 0$. For $\theta \in [ \tfrac{1}{2}, 1)$, however, the expanded second moments do not separately exist in the limit as $y \to \infty$, and one solution would be to adapt Berestycki's thick point approach \cite{Berestycki} and gain integrability by asking the underlying field not to exceed its typical height. We would instead pursue a different modified second moment method, which only relies on the observation that the limiting measure is non-atomic and does not require refined control of the underlying field at each spatial point. To demonstrate this new approach we first explain:
\begin{proof}[Proof of \Cref{lem:modified-2nd}(i) $\Rightarrow$ (ii)]
We claim that for any $K > 0$,
\begin{equation}\label{eq:claim-mod2nd-1}
\limsup_{y \to \infty} \EE\left[ |\nu_{\infty, y}(h) - \nu_{y,y}(h)|^2 e^{-\nu_{y, y}(I) / K}\right] = 0.
\end{equation}
This is sufficient to conclude \eqref{eq:reduced-mc-2} because for any $\delta > 0$, we have
\begin{align*}
\limsup_{y \to \infty} & ~\PP\left(|\nu_{ \infty, y}(h) - \nu_{y, y}(h)| > \delta\right)\\
& \le\limsup_{y \to \infty} \left\{ \PP\left(\nu_{y, y}(I) > K\right) + \PP\left(|\nu_{ \infty, y}(h) - \nu_{y, y}(h)| > \delta, \nu_{y, y}(I) \le K\right)\right\}\\
& \le \limsup_{y \to \infty}  \PP\left(\nu_{ y, y}(I) > K\right) + \limsup_{y \to \infty}\delta^{-2}\EE\left[|\nu_{ \infty, y}(h) - \nu_{y, y}(h)|^2e^{1 -\nu_{y, y}(I)/K}\right]\\
& \le \PP(\nu_\infty(I) \ge K) \xrightarrow[K \to \infty]{} 0.
\end{align*}
To establish \eqref{eq:claim-mod2nd-1}, we expand the left-hand side and consider
\begin{align*}
\limsup_{y \to \infty} \EE\Bigg[ &e^{- \nu_{y, y}(I) / K} \int_{I \times I} h(s_1) h(s_2) \nu_{y, y}(ds_1) \nu_{y,y}(ds_2)\\
& \qquad \times \EE\left[ \left(\prod_{p > y} \frac{\exp \left( G_{p, y}(s_1)\right)}{\EE[\exp \left( G_{p, y}(s_1)\right)]}- 1 \right)\left(\prod_{p > y} \frac{\exp \left( G_{p, y}(s_2)\right)}{\EE[\exp \left( G_{p, y}(s_2)\right)]}- 1 \right)\bigg| \Fa_y\right] \Bigg].
\end{align*}
By independence, the conditional expectation can be simplified to
\[ \left\{\prod_{p > y} \frac{\EE\left[\exp \left( G_{p, y}(s_1) + G_{p, y}(s_2)\right)\right]}{\EE\left[\exp \left( G_{p, y}(s_1)\right)\right]\EE\left[\exp \left( G_{p, y}(s_2)\right)\right]}\right\} - 1 =: C_y(s_1, s_2). \]
Using \Cref{lem:gf-estimates}, we have
\begin{align*}
C_{y}(s_1, s_2)
&= \prod_{p > y}\exp \left( \EE\left[G_{p, y}(s_1) G_{p, y}(s_2) \right]  + O\left( \left| \frac{f(p)}{p^{\sigma_t}}\right|^3\right)  \right) - 1\\
& = \exp \left(2 \sum_{p > y} \frac{|f(p)|^2\cos(|s_1 - s_2| \log p)}{p^{2\sigma_t}} + O\left( \sum_{p > y} \frac{|f(p)|^3}{p^{3/2}}\right) \right) - 1.
\end{align*}
Moreover, \Cref{lem:martingaleL2est-2} implies that $C_y(s_1, s_2)$ is uniformly bounded in $s_1, s_2 \in I$ and $y \ge 3$, and when $s_1, s_2 \in I$ are bounded away from each other we have the uniform convergence $C_y(s_1, s_2) \to 0$ as $y \to \infty$. Since $\nu_{y, y} \xrightarrow[y \to \infty]{p} \nu_\infty$ where $\nu_\infty$ is non-atomic according to \Cref{lem:non-atomic}, we see that 
\[ \int_{I \times I} h(s_1)h(s_2) \nu_{y, y}(ds_1) \nu_{y, y}(ds_2) C_y(s_1, s_2) \xrightarrow[y \to \infty]{p} 0 \]
by \Cref{lem:tensorise}. Using the elementary inequality $xe^{-x/K} \le K$, we have
\begin{align*}
\left|e^{-\nu_{y, y}(I) / K}\int_{I \times I} h(s_1)h(s_2) \nu_{y, y}(ds_1) \nu_{y, y}(ds_2) C_y(s_1, s_2)\right|
&\le \|h\|_\infty^2 \|C_y\|_\infty e^{-\nu_{y, y}(I)/ K} \nu_{y, y}(I)^2\\
& \le 4K^2 \|h\|_\infty^2 \sup_y\|C_y\|_\infty,
\end{align*}
and hence \eqref{eq:claim-mod2nd-1} follows from dominated convergence.
\end{proof}
It remains to show that \Cref{lem:modified-2nd}(i) holds, and based on similar reason it is sufficient to establish
\[ \lim_{y \to \infty} \EE\left[ |\nu_{y, \infty}(h) - \nu_{y, y}(h)|^2 e^{-\nu_{y, \infty}(I) / K}\right] = 0 \]
for any $K > 0$. Expanding the square, we will verify in the next subsection that:
\begin{lem}
For any fixed $K > 0$, we have
\begin{align}
\EE\left[ \nu_{\infty}(h)^2 e^{-\nu_{\infty}(I) / K}\right] 
\label{eq:mod2-martingale}
&= \lim_{y \to \infty} \EE\left[ \nu_{y, \infty}(h)^2 e^{-\nu_{y, \infty}(I) / K}\right] \\
\label{eq:mod2-cross}
& = \lim_{y \to \infty} \EE\left[ \nu_{y, y}(h) \nu_{y, \infty}(h) e^{-\nu_{y, \infty}(I) / K}\right]\\
\label{eq:mod2-diagonal}
& = \lim_{y \to \infty} \EE\left[ \nu_{y, y}(h)^2 e^{-\nu_{y, \infty}(I) / K}\right].
\end{align}
\end{lem}
Since $\nu_{y, \infty}(h) \to \nu_\infty(h)$ and $\nu_{y, \infty}(I) \to \nu_{\infty}(I)$ by martingale convergence, and $\nu_{y, \infty}(h)^2 e^{-\nu_{y, \infty}(I) / K} \le 4K^2\|h\|_\infty^2$, the claim \eqref{eq:mod2-martingale} follows from dominated convergence. Meanwhile, \eqref{eq:mod2-cross} and \eqref{eq:mod2-diagonal} require additional ingredients, and we shall discuss their proofs in separate subsections below.

\subsubsection{Proof of \texorpdfstring{\eqref{eq:mod2-cross}}{mod2-cross} by weak convergence}\label{sec:weakpf}
Let $r \in (1, \infty)$ and $r'$ be the H\"older conjugate of $r$. Recall the following notion from functional analysis: we say a sequence of (real-valued) random variables $X_n$ converges weakly in $L^r(\Omega, \Fa, \PP)$ to $X_\infty$ if
\begin{equation}\label{eq:weak-limit}
\lim_{n \to \infty} \EE[X_n Y] = \EE[X_\infty Y]
\qquad \forall Y \in L^{r'}(\Omega, \Fa, \PP).
\end{equation}
We will consider one further simplification when $X_n$ is unsigned:
\begin{lem}\label{lem:weak-simple}
Let $(X_n)_n$ and $X_\infty$ be non-negative random variables in $L^r(\Omega, \Fa, \PP)$. If $(\Fa_k)_{k \ge 0}$ is a filtration and $\Fa = \Fa_\infty$, then the condition
\begin{equation}\label{eq:weak-limit2}
\lim_{n\to \infty} \EE\left[X_n Y\right] = \EE\left[X_\infty Y\right] \qquad \forall Y \in L^{\infty}(\Omega, \Fa_k, \PP), \qquad \forall k \in \NN
\end{equation}
(or equivalently, $\lim_{n \to \infty} \EE[X_n |\Fa_k] = \EE[X_\infty | \Fa_k]$ for any $k \in \NN$) implies \eqref{eq:weak-limit}.
\end{lem}
\begin{proof}
By splitting the random variable $Y$ into positive and negative part, we may assume without loss of generality that $Y$ is non-negative.

Suppose $Y \in L^\infty(\Omega, \Fa, \PP)$, and define $Y_k := \EE[Y | \Fa_k]$. Then $(Y_k)$ is a uniformly bounded martingale, and $Y_k \to Y$ almost surely and in $L^{r'}$. In particular,
\begin{align*}
 \limsup_{n \to \infty} \left|\EE\left[(X_n - X_\infty) Y\right]\right|
&\le \limsup_{n \to \infty}  \left|  \EE\left[(X_n-X_\infty) Y_k\right]\right|
+  \limsup_{n \to \infty}  \left| \EE\left[(X_n-X_\infty) (Y_k - Y_\infty)\right]\right|\\
& \le \left[ 2 \sup_n \EE[|X_n|^r]^{1/r}\right] \EE\left[|Y_k-Y_\infty |^{r'}\right]^{1/r'}
\end{align*}
which can be made arbitrarily small as $k \to \infty$. This suggests that \eqref{eq:weak-limit2} implies the same condition but for any $Y \in L^\infty(\Omega, \Fa, \PP)$. Finally, to deduce \eqref{eq:weak-limit} for any non-negative $Y \in L^{r'}(\Omega, \Fa, \PP)$, use truncation and conclude with monotone convergence theorem.
\end{proof}
It is straightforward to check that if $X_n$ converges weakly in $L^r$ to $X_\infty$ and $Y_n$ converges (strongly) in $L^{r'}$ to $Y_\infty$, then $\EE[X_n Y_n] \to \EE[X_\infty Y_\infty]$. Therefore, \eqref{eq:mod2-cross} holds if we can establish the following claim:
\begin{lem}
Let $r  > 1$ be such that $\left(\nu_{y, y}(I)^r\right)_{y \ge 3}$ is uniformly integrable. Then
\begin{itemize}
\item[(i)] $\nu_{y, \infty}(h) e^{-\nu_{y, \infty}(I)/K}$ converges (strongly) in $L^{r'}(\Omega, \Fa, \PP)$ for any $r' > 1$ to $\nu_{\infty}(h) e^{-\nu_\infty(I)/K}$.
\item[(ii)] $\nu_{y, y}(h)$ converges weakly in $L^{r}(\Omega, \Fa, \PP)$ to $\nu_\infty(h)$.
\end{itemize}
In particular, the equality \eqref{eq:mod2-cross} holds.
\end{lem}
\begin{proof}
We only need to verify (ii) for non-negative $h \in C(I)$ (by splitting it into positive and negative parts), and by \Cref{lem:weak-simple} it suffices to check that
\begin{equation}\label{eq:check_weak}
    \lim_{y \to \infty} \EE[\nu_{y, y}(h) Y] = \EE[\nu_\infty(h) Y] \qquad \forall Y \in L^\infty(\Omega, \Fa_k, \PP), \qquad \forall k \in \NN.
\end{equation}
But this is obviously true because
\begin{align*}
\EE[\nu_{y, y}(h) Y]
&=\int_I h(s) ds \EE \left[ \EE\left[\frac{\exp\left(\Ga_{y, y}(s)\right)}{\EE\exp\left(\Ga_{y, y}(s)\right)} Y\bigg|\Fa_k\right]\right]
 = \int_I h(s) ds \EE \left[ \frac{\exp\left(\Ga_{k, y}(s)\right)}{\EE\exp\left(\Ga_{k, y}(s)\right)} Y\right]\\
& \xrightarrow[y \to \infty]{} \int_I h(s) ds \EE \left[ \frac{\exp\left(\Ga_{k, \infty}(s)\right)}{\EE\exp\left(\Ga_{k, \infty}(s)\right)} Y\right]
= \EE[\nu_{k, \infty}(h) Y] = \EE[\nu_{\infty}(h)Y].
\end{align*}
\end{proof}
\begin{rem}
Equation \eqref{eq:check_weak} is actually equivalent to $\lim_{y \to \infty}\EE[\nu_{y, y}(h)|\Fa_k] = \EE[\nu_\infty(h) | \Fa_k]$ for any $k \in \NN$.
\end{rem}

\subsubsection{Proof of \texorpdfstring{\eqref{eq:mod2-diagonal}}{mod2-diagonal} by change of measure}\label{sec:pf-ch-measure}
Note that we can rewrite
\begin{align*}
& \EE\left[ \nu_{y,t}(h)^2 e^{-\nu_{y, \infty}(I) / K}\right]
= \int_{I \times I} h(s_1) h(s_2) ds_1 ds_2 \EE\left[ \frac{\exp\left(\Ga_{y, t}(s_1) + \Ga_{y, t}(s_2)\right)}{\EE\exp\left(\Ga_{y, t}(s_1)\right)\EE\exp\left(\Ga_{y, t}(s_2)\right)}e^{-\nu_{y, \infty}(I) / K}\right]\\
& = \int_{I \times I} ds_1 ds_2 h(s_1) h(s_2) 
\frac{\EE\exp\left(\Ga_{y, t}(s_1) + \Ga_{y, t}(s_2)\right)}{\EE\exp\left(\Ga_{y, t}(s_1)\right)\EE\exp\left(\Ga_{y, t}(s_2)\right)}
\EE\left[ \frac{\exp\left(\Ga_{y, t}(s_1) + \Ga_{y, t}(s_2)\right)}{\EE\exp\left(\Ga_{y, t}(s_1)+\Ga_{y, t}(s_2)\right)}e^{-\nu_{y, \infty}(I) / K}\right]\\
& =: \int_{I \times I} ds_1 ds_2 h(s_1) h(s_2) C_{y, t}(s_1, s_2) \EE_{y, t}^{(s_1, s_2)}\left[e^{-\nu_{y, \infty}(I) / K}\right]
\end{align*}
where
\[ C_{y, t}(s_1, s_2)  := \frac{\EE\exp\left(\Ga_{y, t}(s_1) + \Ga_{y, t}(s_2)\right)}{\EE\exp\left(\Ga_{y, t}(s_1)\right)\EE\exp\left(\Ga_{y, t}(s_2)\right)} \]
and $\EE_{y,t}^{(s_1, s_2)}$ is the expectation with respect to the new probability measure $\PP_{y, t}^{(s_1, s_2)}$ under which $(\alpha(p))_p$ remain jointly independent and that their marginal distributions are characterised by
\[ \EE_{y,t}^{(s_1, s_2)} \left[ g(\alpha(p))\right]
:= \EE \left[ \frac{\exp\left(G_{p, t}(s_1) + G_{p, t}(s_2)\right)}{\EE\exp\left(G_{p, t}(s_1) + G_{p, t}(s_2)\right)}g(\alpha(p))\right] \]
for all $p \le y$, and $\left(\alpha(p)\right)_{p > y}$ remain i.i.d.~and uniformly distributed on the unit circle. In particular,
\begin{equation}\label{eq:mod2-control-diagonal}
\EE\left[ \nu_{y,y}(h)^2 e^{-\nu_{y, \infty}(I) / K}\right]
= \int_{I \times I} ds_1 ds_2 h(s_1)h(s_2)\frac{C_{y, y}(s_1, s_2)}{C_{y, \infty}(s_1, s_2)} {C_{y, \infty}(s_1, s_2)} \EE_{y, y}^{(s_1, s_2)} \left[e^{-\nu_{y, \infty}(I) / K}\right]
\end{equation}
and if we can show (in some suitable sense) that
\begin{itemize}
\item $C_{y, y}(s_1, s_2) / C_{y, \infty}(s_1, s_2) \approx 1$ as $y \to \infty$, and
\item $\EE_{y, y}^{(s_1, s_2)} \left[e^{-\nu_{y, \infty}(I) / K}\right] \approx \EE_{y,\infty}^{(s_1, s_2)} \left[e^{-\nu_{y, \infty}(I) / K}\right]$ as $y \to \infty$,
\end{itemize}
then \eqref{eq:mod2-control-diagonal} would approximately equal $\EE\left[ \nu_{y,\infty}(h)^2 e^{-\nu_{y, \infty}(I) / K}\right]$ and hence have the same limit as $y \to \infty$. 

Note that we will often abuse the notation and write $\EE_{y, t}^{(s_1, s_2)} \left[ A \right]  = \EE_{\infty, t}^{(s_1, s_2)} \left[ A \right]$ when $A$ is $\Fa_y$-measurable random variable, i.e.~it only depends on $\alpha(p)$ for $p \le y$ and thus its distribution is not affected when the law of $(\alpha(p))_{p > y}$ is changed. In particular, $\EE_{y, y}^{(s_1, s_2)} \left[e^{-\nu_{y, \infty}(I) / K}\right] = \EE_{\infty, y}^{(s_1, s_2)} \left[e^{-\nu_{y, \infty}(I) / K}\right]$.

To motivate our analysis below, it may be instructive to first sketch the arguments for the toy problem, namely when $(\alpha(p))_{p}$ are i.i.d.~standard complex Gaussian random variables. In this simplified case, recall the Cameron--Martin--Girsanov theorem: if $(N, X)$ are centred and jointly Gaussian (here $N$ is a real-valued random variable and $X= X(\cdot)$ is a random field/collection of real-valued random variables), then
\[ \EE\left[ \frac{\exp(N)}{\EE[\exp(N)]} g(X(\cdot))\right]
= \EE\left[g\left(X(\cdot) + \EE\left[ N X(\cdot)\right]\right)\right] \]
for any suitable test functions $g$, i.e.~the change of measure by exponential tilting only results in the shifting of mean. If we apply this to
\[    N \leftrightarrow \Ga_{y, t}(s_1) + \Ga_{y, t}(s_2),
    \quad X(\cdot) \leftrightarrow \Ga_{y, \infty}(\cdot) \quad \text{and} \quad
    g(X(\cdot)) \leftrightarrow \exp\left(-\frac{1}{K}\int_I \frac{\exp\left(\Ga_{y, \infty}(s)\right)ds}{\EE\left[\exp\left(\Ga_{y, \infty}(0)\right) \right]}\right)
    = e^{-\nu_{y, \infty}(I) / K}, \]
we can write
\[\EE_{y, t}^{(s_1, s_2)}\left[e^{-\nu_{y, \infty}(I) / K}\right]
= \EE\left[\exp\left(-\frac{1}{K}\int_I \exp\left(\EE\left[ \left(\Ga_{y, t}(s_1)+\Ga_{y, t}(s_2) \right) \Ga_{y, \infty}(s)\right]\right)\nu_{y, \infty}(ds)\right) \right] \]
and in particular
\begin{align*}
\EE_{y, y}^{(s_1, s_2)}\left[e^{-\nu_{y, \infty}(I) / K}\right]
& = \EE\left[\exp\left(-\frac{1}{K}\int_I \exp\left(\EE\left[ \left(\Ga_{y, y}(s_1)+\Ga_{y, y}(s_2) \right) \Ga_{y, \infty}(s)\right]\right)\nu_{y, \infty}(ds)\right) \right]\\
& = \EE_{y, \infty}^{(s_1, s_2)}\left[\exp\left(-\frac{1}{K}\int_I \exp\left(\Ma_y^{(s_1, s_2)}(s)\right)\nu_{y, \infty}(ds)\right) \right]
\end{align*}
where
\begin{align*}
\Ma_y^{(s_1, s_2)}(s)
:=&~ \EE\left[ \left(\Ga_{y, y}(s_1)+\Ga_{y, y}(s_2) \right) \Ga_{y, \infty}(s)\right]
- \EE\left[\left(\Ga_{y, \infty}(s_1)+\Ga_{y, \infty}(s_2) \right) \Ga_{y, \infty}(s)\right]\\
=&~ \EE_{y, y}^{(s_1, s_2)}\left[\Ga_{y, \infty}(s) \right]-\EE_{y, \infty}^{(s_1, s_2)}\left[\Ga_{y, \infty}(s) \right].
\end{align*}
Based on direct computations, one can easily show that $C_{y, y}(s_1, s_2) / C_{y, \infty}(s_1, s_2)$ is uniformly bounded in $s_1, s_2 \in I$ and this ratio converges to $1$ for any fixed $s_1 \ne s_2$ as $y \to \infty$. Similarly, one can verify that $|\Ma_{y}^{(s_1, s_2)}(u)|$ is uniformly bounded in $s, s_1, s_2 \in I$ and that it converges to $0$ as long as $s, s_1, s_2$ are distinct. Since $\nu_{y, \infty}$ and its limit $\nu_\infty$ are non-atomic measures, we observe that
\[\EE_{y, \infty}^{(s_1, s_2)}\left[\exp\left(-\frac{1}{K}\int_I \exp\left(\Ma_y^{(s_1, s_2)}(s)\right)\nu_{y, \infty}(ds)\right) \right]
\xrightarrow[y \to \infty]{}
\EE_{\infty, \infty}^{(s_1, s_2)}\left[e^{-\nu_{ \infty}(I)/K} \right] \]
for any $s_1 \ne s_2$, and hence we can conclude that \eqref{eq:mod2-control-diagonal} has the same limit as $\EE\left[ \nu_{y,\infty}(h)^2 e^{-\nu_{y, \infty}(I) / K}\right]$ when $y \to \infty$ by dominated convergence.

Going back to the actual problem where $(\alpha(p))_p$ are non-Gaussian, we have lost access to Cameron--Martin--Girsanov theorem and many other Gaussian techniques. To proceed with a change-of-measure argument, we will need the following approximate Girsanov theorem which could be of independent interest.
\begin{thm}\label{theo:Girsanov}
Let $\eta$ be a probability measure on $\RR^d$ with compact support in $\overline{B}(0, r)$, and for each $a \in \RR^d$ let $\eta_a$ be the probability measure satisfying $ \eta_a(dx) \propto e^{\langle a, x \rangle}\eta(dx)$. Then there exist a pair of random variables $(U, U_a)$ on some probability space such that $U \sim \eta$, $U_a \sim \eta_a$ and with the properties that
\begin{itemize}
\item[(i)] $\EE\left[U_a - U\right] = \nabla_a \log \EE[\exp\left(\langle a, U\rangle \right)] - \nabla_{a=0} \log \EE[\exp\left(\langle a, U\rangle \right)]$.
\item[(ii)] $\EE\left[|U_a - U|\right] \le C|a|$ and $\EE\left[|U_a - U|^2\right] \le C|a|$ for some $C > 0$.
\end{itemize}
\end{thm}
\begin{proof}
Property (i) follows by differentiation. As for property (ii), since $\EE\left[|U_a - U|^2\right] \le 2r \EE\left[|U_a - U|\right]$ it is sufficient to establish a coupling such that the first moment estimate
\[ \EE\left[|U_a - U|\right] \le C |a| \]
holds for $|a|$ sufficiently small (so that all Taylor series expansions below are justified), or in other words the $1$-Wasserstein distance satisfies $\Wa_1(\eta, \eta_a) \le C|a|$. Recall that the total variation distance between two probability measures can be defined as
\[  d_{\mathrm{TV}}(\eta, \eta_a) := \frac{1}{2} \sup_{f: \, \|f\|_\infty \le 1} \EE\left[ f(U) - f(U_a)\right] \]
where $U \sim \eta$ and $U_a \sim \eta_a$, and that under the maximal coupling we have $d_{\mathrm{TV}}(\eta, \eta_a) = \PP(U \ne U_a)$. Since $\eta, \eta_a$ are equivalent probability measures, we may use the alternative formula for total variation distance
\[ d_{\mathrm{TV}}(\eta, \eta_a) 
 = \frac{1}{2} \int \left|\frac{e^{a \cdot u}}{\EE[e^{a \cdot U}]} - 1\right| \eta(du)
\ll \int \left[|a \cdot (u - \EE[U])| + |a|^2\right] \eta(du)
\ll |a| \]
where the first inequality follows from Taylor series expansion and second inequality from the fact that $|a| \ll 1$. Finally, recall by duality that the $1$-Wasserstein distance can be obtained via
\[ \Wa_1(\eta, \eta_a) := \sup_{f:\, \|f\|_{\mathrm{Lip}} \le 1} \EE\left[ f(U) - f(U_a)\right] \]
where the expression on the right-hand side is invariant under a constant shift in $f$, so we may assume without loss of generality that $f(u) = 0$ for some $u \in B(0, r)$. But such functions also satisfy $\|f\|_\infty \le 2r$ by the Lipschitz condition, and hence $\Wa_1(\eta, \eta_a) \le 2r d_{\mathrm{TV}}(\eta, \eta_a)$ from which the result follows. 
\end{proof}
When our random variables are non-Gaussian, exponential tilting induces not only a shift in mean but also extra random fluctuations as in \Cref{theo:Girsanov}(ii). For this reason, we will pursue a more refined analysis of $\EE_{y, y}^{(s_1, s_2)}\left[e^{-\nu_{y, \infty}(I) / K}\right]$ depending on whether $|s_1 - s_2|$ is small or not (see Step 3 below), but our overall philosophy essentially mirrors that in the toy problem discussed earlier. We are now ready to explain:
\begin{proof}[Proof of \eqref{eq:mod2-diagonal}]
For notational convenience let us only treat $I = [0,1]$ even though the proof can be easily generalised. Given \eqref{eq:mod2-martingale} and \eqref{eq:mod2-cross}, \eqref{eq:mod2-diagonal} follows if we can show that
\[ \limsup_{y \to \infty} \EE\left[ \nu_{y, y}(h)^2 e^{-\nu_{y, \infty}(I) / K}\right]
\le  \EE\left[ \nu_{ \infty}(h)^2 e^{-\nu_{ \infty}(I) / K}\right]. \]
\paragraph{Step 1: coupling.}
Recall
\[ \nu_{y, \infty}(ds) = \nu_{y, \infty}(ds; \alpha) = \frac{\exp\left(\Ga_{y, \infty}(s; \alpha)\right)}{\EE\exp\left(\Ga_{y, \infty}(s)\right)} ds \]
where the notations $\nu_{y, \infty}(ds; \alpha), \Ga_{y, \infty}(s; \alpha)$ emphasise that they may be seen as functions of the vector $(\alpha(p))_p$ (we are not using this notation in the denominator because the expectation is evaluated to a deterministic value and does not depend on the realisation of a random vector). Let us apply \Cref{theo:Girsanov} with
\[ U \leftrightarrow \begin{pmatrix}
\Re \alpha(p) \\ \Im \alpha(p)
\end{pmatrix},
\quad 
U_a \leftrightarrow \begin{pmatrix}
\Re \alpha_y(p) \\ \Im \alpha_y(p)
\end{pmatrix}
\quad \text{and} \quad 
a \leftrightarrow \begin{pmatrix}
\dfrac{2|f(p)| \OurEpsilon_{p, y}}{p^{1/2}} \cos(|s_1 - s_2| \log p)\\
-\dfrac{2|f(p)| \OurEpsilon_{p, y}}{p^{1/2}} \sin(|s_1 - s_2| \log p)
\end{pmatrix} \]
to obtain independent pairs of random variables $(\alpha_y(p), \alpha(p))_{p \le y}$ on a common probability space such that the marginal laws are characterised by
\[ \EE_{y, y}^{(s_1, s_2)}\left[ g(\alpha(p))\right]
= \EE_{y,\infty}^{(s_1, s_2)}\left[ g(\alpha_y(p))\right] \qquad \forall p \le y,\]
and that they satisfy
\[ \EE_{y, \infty}^{(s_1, s_2)}\left[|\alpha_y(p) - \alpha(p)|\right]
+\EE_{y, \infty}^{(s_1, s_2)}\left[|\alpha_y(p) - \alpha(p)|^2\right] \ll |\OurEpsilon_{p, y}| p^{-1/2} \qquad \forall p \le y. \]
Then
\begin{align}
\notag
\EE_{y, y}^{(s_1, s_2)} \left[ e^{-\nu_{y, \infty}(I) / K}\right]
& = \EE_{y, \infty}^{(s_1, s_2)} \left[ \exp \left(-K^{-1}\int_I \frac{\exp\left(\Ga_{y, \infty}(u; \alpha_y)\right)}{\EE\exp\left(\Ga_{y, \infty}(u)\right)} du  \right) \right]\\
\label{eq:coupled_exp}
& = \EE_{y, \infty}^{(s_1, s_2)} \left[ \exp \left(-K^{-1}\int_I \exp\left(\Ga_{y, \infty}(u; \alpha_y - \alpha)  \right)\frac{\exp\left(\Ga_{y, \infty}(u; \alpha)\right)}{\EE\exp\left(\Ga_{y, \infty}(u)\right)} du  \right) \right]
\end{align}
where
\[ \Ga_{y, \infty}(u; \alpha_y - \alpha) 
= \Ga_{y, \infty}(u; \alpha_y)  - \Ga_{y, \infty}(u; \alpha) 
= 2\Re\sum_{p \le y}\frac{[\alpha_y(p) - \alpha(p)] f(p)}{p^{1/2+ iu}}. \]
If we set $\Ma_y^{(s_1, s_2)}(u) := \EE_{y, \infty}^{(s_1, s_2)}\left[\Ga_{y, \infty}(u; \alpha_y - \alpha)\right]$ and $\widetilde{\Ga}_{y, \infty}(u) := \Ga_{y, \infty}(u; \alpha_y - \alpha)  - \Ma_y^{(s_1, s_2)}(u)$, then by the Sobolev inequality (\Cref{thm:sobolev})
\begin{align*}
& \EE_{y, \infty}^{(s_1, s_2)} \left[ \|\widetilde{\Ga}_{y, \infty}\|_\infty^2\right]
=\EE_{\infty, \infty}^{(s_1, s_2)} \left[ \|\widetilde{\Ga}_{y, \infty}\|_\infty^2\right]\\
&  \ll \EE_{\infty, \infty}^{(s_1, s_2)} \left[ \|\widetilde{\Ga}_{y, \infty}\|_{W^{1,2}(I)}^2\right]
\ll \sum_{p \le y} \frac{\log^2 p}{p}  \EE_{\infty, \infty}^{(s_1, s_2)}\left[|\alpha_y(p) - \alpha(p)|^2\right]
\ll \sum_{p \le y} |\OurEpsilon_{p, y}| \frac{|f(p)|^3}{p^{3/2}} \log^2 p \ll \frac{1}{\log y} \xrightarrow[y \to \infty]{} 0
\end{align*}
uniformly in $s_1, s_2 \in I$. In particular, for any fixed $\delta_1 > 0$, we can upper bound \eqref{eq:coupled_exp} by
\begin{align} 
\notag
& \EE_{\infty, \infty}^{(s_1, s_2)} \left[ \exp \left(-K^{-1}\int_I \exp\left( -|\Ma_y^{(s_1, s_2)}(u)|- \delta_1 \right)\nu_{y, \infty}(ds; \alpha)  \right) \right]+ \PP_{\infty, \infty}^{(s_1, s_2)} \left( \|\widetilde{\Ga}_{y, \infty}\|_\infty > \delta_1\right)\\
\label{eq:damping-factor}
&\quad   =: E_y(s_1, s_2; \delta_1) + \PP_\infty^{(s_1, s_2)} \left( \|\widetilde{\Ga}_{y, \infty}\|_\infty > \delta_1\right)
\end{align}
where the last probability is bounded by $ \delta_1^{-2} \EE_{\infty, \infty}^{(s_1, s_2)} \left[ \|\widetilde{\Ga}_{y, \infty}(\cdot)\|_\infty^2\right] = O\left(1 / (\delta_1^2 \log y)\right)$ by Markov's inequality. 
\paragraph{Step 2: estimates for $C_{y, t}(s_1, s_2)$.}
Using \Cref{lem:gf-estimates}, we have
\begin{align*}
 \frac{\EE\exp\left(G_{p, t}(s_1) + G_{p, t}(s_2)\right)}{\EE\exp\left(G_{p, t}(s_1)\right)\EE\exp\left(G_{p, t}(s_2)\right)}
& = \exp \left\{\EE[G_{p, t}(s_1) G_{p,t}(s_2)] + O\left(\frac{|f(p)|^3}{p^{3/2}} \right) \right\}\\
& = \exp \left\{ \frac{2|f(p)|^2}{p^{2\sigma_t}} \cos(|s_1 - s_2| \log p)+ O\left(\frac{|f(p)|^3}{p^{3/2}} \right) \right\}.
\end{align*}
In particular, 
\begin{align*}
C_{y, \infty}(s_1, s_2) 
&= \prod_{p \le y}  \frac{\EE\exp\left(G_{p, \infty}(s_1) + G_{p, \infty}(s_2)\right)}{\EE\exp\left(G_{p, \infty}(s_1)\right)\EE\exp\left(G_{p, \infty}(s_2)\right)}\\
& \ll \exp \left\{\sum_{p \le y} \frac{2|f(p)|^2}{p} \cos(|s_1 - s_2| \log p) \right\}
 \ll \left(|s_1 - s_2| \vee \log^{-1}(y)\right)^{-2\theta}
\end{align*}
where the last estimate follows from  \Cref{lem:new-martingaleL2est}, and this implies
\[ \int_{I \times I}ds_1 ds_2 \frac{C_{y, y}(s_1, s_2)}{C_{y, \infty}(s_1, s_2)} C_{y, \infty}(s_1, s_2)
\ll \int_{I \times I} ds_1 ds_2  \left(|s_1 - s_2| \vee \log^{-1}(y)\right)^{-2\theta}
\ll 1 + \log^{2\theta - 1} y. \]
Combining this with \eqref{eq:damping-factor}, we have
\[ \limsup_{y \to \infty} \EE\left[ \nu_{y, y}(h)^2 e^{-\nu_{y, \infty}(I) / K}\right]
\le \limsup_{y \to \infty} \int_{I \times I}ds_1 ds_2 h(s_1) h(s_2)\frac{C_{y, y}(s_1, s_2)}{C_{y, \infty}(s_1, s_2)} C_{y, \infty}(s_1, s_2) E_y(s_1, s_2; \delta_1). \]
\paragraph{Step 3: treating merging singularities.} Let $L > 0$ be such that
\[ \frac{C_{y, y}(s_1, s_2)}{C_{y, \infty}(s_1, s_2)} \le L \qquad \text{and} \qquad |\Ma_y^{(s_1, s_2)}| \le L \]
uniformly in $y \ge 3$ and $s_1, s_2 \in I$, which is possible by \Cref{lem:tilt_ratios} and \Cref{lem:M_y}. Also let $\bar{h}$ be a non-negative continuous function with the property that $\bar{h} \le 1$, $\bar{h}|_{[0, 1/2]} = 1$ and $\mathrm{supp}(\bar{h}) = [0,1]$. Then for any $m_1 \in \NN$, we have
\begin{align*}
& \int_{I \times I}ds_1 ds_2 h(s_1) h(s_2) \bar{h}(2^{m_1}|s_1 - s_2|)\frac{C_{y, y}(s_1, s_2)}{C_{y, \infty}(s_1, s_2)} {C_{y, \infty}(s_1, s_2)} E_y(s_1, s_2; \delta_1) \\
& \le L \|h\|_\infty^2  \int_{I \times I} ds_1 ds_2 \bar{h}(2^{m_1}|s_1 - s_2|) C_{y, \infty}(s_1, s_2) \EE_\infty^{(s_1, s_2)} \left[ \exp \left(-\frac{1}{Ke^{L + \delta_1}}\int_I \frac{\exp\left(\Ga_{y, \infty}(u; \alpha)\right)}{\EE\exp\left(\Ga_{y, \infty}(u)\right)} du  \right) \right]\\
& = L \|h\|_\infty^2  \EE\left[\int_{I \times I}  \bar{h}(2^{m_1}|s_1 - s_2|) \nu_{y, \infty}(ds_1) \nu_{y, \infty}(ds_2) \exp \left(-\frac{\nu_{y, \infty}(I) }{Ke^{L + \delta_1}}\right) \right]\\
& \le L\|h\|_\infty^2  \EE\left[\left(\max_{j \le 2^{m_1} - 1} \nu_{y, \infty}\left([j2^{-m_1}, (j+1)2^{-m_1}]\right)\right) \nu_{y, \infty}(I)\exp \left(-\frac{\nu_{y, \infty}(I) }{Ke^{L + \delta_1}}\right) \right]\\
& \xrightarrow[y \to \infty]{} L\|h\|_\infty^2  \EE\left[\left(\max_{j \le 2^{m_1} - 1} \nu_{\infty}\left([j2^{-m_1}, (j+1)2^{-m_1}]\right)\right) \nu_{ \infty}(I)\exp \left(-\frac{\nu_{\infty}(I) }{Ke^{L + \delta_1}}\right) \right]
\end{align*}
where the last line again follows from dominated convergence, since $\nu_{y, \infty}(J) \xrightarrow[y \to \infty]{a.s.} \nu_{\infty}(J)$ for any fixed sub-interval $J \subset I$ by martingale convergence theorem, and 
\[ \left(\max_{j \le 2^{m_1} - 1} \nu_{\infty}([j2^{-m_1}, (j+1)2^{-m_1}])\right) \nu_{ \infty}(I)\exp \left(-\frac{\nu_{\infty}(I) }{Ke^{L + \delta_1}}\right) 
\le \nu_{ \infty}(I)^2\exp \left(-\frac{\nu_{\infty}(I) }{Ke^{L + \delta_1}}\right) 
\le  (2Ke^{L + \delta_1})^2 \]
uniformly in $j, m_1$ and $y$.

Now for any $m_2  > m_1$, define $I_{m_2}(s_1, s_2) = I \setminus \left(B(s_1, 2^{-m_2}) \cup B(s_2, 2^{-m_2})\right)$. Using  \Cref{lem:tilt_ratios} and \Cref{lem:M_y}, let $\delta_2 = \delta_2(m_2) > 0$ be such that
\[ \frac{C_{y, y}(s_1, s_2)}{C_{y, \infty}(s_1, s_2)} \le 1+\delta_2 \qquad \text{and} \qquad |\Ma_y^{(s_1, s_2)}(u)| \le \delta_2 \]
uniformly in $|s_1 - s_2| \ge 2^{-m_2}$ and  $u \in I_{m_2}(s_1, s_2)$ for all $y$ sufficiently large (possibly depending on $m_2$). Note that $\delta_2(m)$ can be chosen in a way such that $\delta_2(m_2) \to 0$ as $m_2 \to \infty$. Then
\begin{align*}
 \int_{I \times I} &ds_1 ds_2 h(s_1) h(s_2) \left[1 - \bar{h}(2^{m_1}|s_1 - s_2|)\right]\frac{C_{y, y}(s_1, s_2)}{C_{y, \infty}(s_1, s_2)} {C_{y, \infty}(s_1, s_2)} E_y(s_1, s_2; \delta_1) \\
&\le  (1+\delta_2) \EE\Bigg[ \int_{I \times I} h(s_1) h(s_2) \left[1 - \bar{h}(2^{m_1}|s_1 - s_2|)\right]\nu_{y, \infty}(ds_1) \nu_{y, \infty}(ds_2) \exp \left(-\frac{\nu_{y, \infty}\left(I_{m_2}(s_1, s_2)\right)}{Ke^{\delta_1 + \delta_2}}  \right) \Bigg]\\
&\le  (1+\delta_2) \EE\Bigg[ \int_{I \times I} h(s_1) h(s_2) \left[1 - \bar{h}(2^{m_1}|s_1 - s_2|)\right]\nu_{y, \infty}(ds_1) \nu_{y, \infty}(ds_2) \exp \left(-\frac{\nu_{\infty}\left(I_{m_2}(s_1, s_2)\right)}{Ke^{\delta_1 + \delta_2}}  \right) \Bigg]
\end{align*}
where the last line follows from the martingale property of $\nu_\infty(I_m(s_1, s_2))$ and conditional Jensen's inequality. Since the function
\[ (s_1, s_2) \mapsto \exp \left(-\frac{\nu_{\infty}\left(I_{m_2}(s_1, s_2)\right)}{Ke^{\delta_1 + \delta_2}}  \right) \]
is continuous for $|s_1 -s_2| \ge 2^{-m_2}$ as a consequence of dominated convergence and non-atomicity of $\nu_\infty$, we combine \Cref{lem:tensorise} and \Cref{lem:add-cont-density} to obtain
\begin{align*}
& \lim_{y \to \infty} \EE\Bigg[ \int_{I \times I} h(s_1) h(s_2) \left[1 - \bar{h}(2^{m_1}|s_1 - s_2|)\right]\nu_{y, \infty}(ds_1) \nu_{y, \infty}(ds_2) \exp \left(-\frac{\nu_{\infty}\left(I_{m_2}(s_1, s_2)\right)}{Ke^{\delta_1 + \delta_2}}  \right) \Bigg]\\
& = \EE\Bigg[ \int_{I \times I} h(s_1) h(s_2) \left[1 - \bar{h}(2^{m_1}|s_1 - s_2|)\right]\nu_{\infty}(ds_1) \nu_{\infty}(ds_2) \exp \left(-\frac{\nu_{\infty}\left(I_{m_2}(s_1, s_2)\right)}{Ke^{\delta_1 + \delta_2}}  \right) \Bigg].
\end{align*}
\paragraph{Step 4: conclusion.}
Summarising all the work so far, we have
\begin{align*}
 \limsup_{y \to \infty}& \EE\left[ \nu_{y, y}(h)^2 e^{-\nu_{y, \infty}(I) / K}\right]
 \le L\|h\|_\infty^2  \EE\left[\left(\max_{j \le 2^{m_1} - 1} \nu_{\infty}\left([j2^{-m_1}, (j+1)2^{-m_1}]\right)\right) \nu_{ \infty}(I)\exp \left(-\frac{\nu_{\infty}(I) }{Ke^{L + \delta_1}}\right) \right]\\
& +  \EE\Bigg[ \int_{I \times I} h(s_1) h(s_2) \left[1 - \bar{h}(2^{m_1}|s_1 - s_2|)\right]\nu_{\infty}(ds_1) \nu_{\infty}(ds_2) \exp \left(-\frac{\nu_{\infty}\left(I_{m_2}(s_1, s_2)\right)}{Ke^{\delta_1 + \delta_2}}  \right) \Bigg].
\end{align*}
The first term on the right-hand side vanishes as $m_1 \to \infty$ since
\[ \max_{j \le 2^{m_1} - 1} \nu_{\infty}\left([j2^{-m_1}, (j+1)2^{-m_1}]\right) \xrightarrow[m_1 \to \infty]{a.s.} 0 \]
by the non-atomicity of $\nu_\infty$ (\Cref{lem:non-atomic}). Meanwhile, the second term satisfies
\begin{align*}
&  \EE\Bigg[ \int_{I \times I} h(s_1) h(s_2) \left[1 - \bar{h}(2^{m_1}|s_1 - s_2|)\right]\nu_{\infty}(ds_1) \nu_{\infty}(ds_2) \exp \left(-\frac{\nu_{\infty}\left(I_{m_2}(s_1, s_2)\right)}{Ke^{\delta_1 + \delta_2}}  \right) \Bigg]\\
& \xrightarrow[m_2 \to \infty, \delta_2 \to 0]{}
 \EE\Bigg[ \int_{I \times I} h(s_1) h(s_2) \left[1 - \bar{h}(2^{m_1}|s_1 - s_2|)\right]\nu_{\infty}(ds_1) \nu_{\infty}(ds_2) \exp \left(-\frac{\nu_{\infty}(I)}{Ke^{\delta_1}}  \right) \Bigg]\\
& \xrightarrow[m_1 \to \infty, \delta_1 \to 0]{}
 \EE\Bigg[ \int_{I \times I} h(s_1) h(s_2) \nu_{\infty}(ds_1) \nu_{\infty}(ds_2) \exp \left(-\nu_{\infty}(I) / K  \right) \Bigg]
= \EE\left[ \nu_{ \infty}(h)^2 e^{-\nu_{ \infty}(I) / K}\right]
\end{align*}
again by the non-atomicity of $\nu_\infty$ as well as monotone convergence. This concludes our proof.
\end{proof}
\begin{proof}[Proof of \Cref{cor:mc-L1}]
Suppose $h\colon \RR \to \RR$ is a continuous function with $\mathrm{supp}(h) = I$ for some compact interval $I$. Then
\begin{align*}
\EE|m_{y, t(y)}(h) - m_\infty(h)|
&\le \EE|m_{y, \infty}(h) - m_\infty(h)| + \EE|m_{y, t(y)}(h) - m_{y, \infty}(h)|\\
&= \EE|m_{y, \infty}(h) - m_\infty(h)| + \EE|\EE\left[m_{\infty, t(y)}(h) - m_{\infty}(h) | \Fa_y\right]|\\
& \le \EE|m_{y, \infty}(h) - m_\infty(h)| + \EE|m_{\infty, t(y)}(h) - m_{\infty}(h)|
\end{align*}
where the last step follows from (conditional) Jensen's inequality. The above bound converges to $0$ as $y \to \infty$ by \Cref{thm:mc-L1}.

Now suppose $h \in L^1(\RR)$ instead. Since $C_c(\RR)$ is dense in $L^1(\RR)$, there exists a sequence of functions $h_n \in C_c(\RR)$ such that $\|h-h_n\|_1 \to 0$ as $n \to \infty$. Then
\begin{align*}
\EE|m_{y, t(y)}(h) - m_\infty(h)|
& \le \EE|m_{y, t(y)}(h_n) - m_\infty(h_n)|
+ \EE\left[m_{y, t(y)}(|h - h_n|)\right]
+ \EE\left[m_{\infty}(|h - h_n|)\right]\\
& \le \EE|m_{y, t(y)}(h_n) - m_\infty(h_n)| + 2\|h - h_n\|_{1}
\end{align*}
which vanishes in the limit as $y \to \infty$ and then $n \to \infty$.
\end{proof}

\section{Part II: Distribution of random sums}\label{sec:randomsum}
This section is devoted to the proof of \Cref{thm:summain} and is organised as follows:
\begin{itemize}
\item In \Cref{sec:gen} we state \Cref{thm:summainw}, a generalised version of \Cref{thm:summain} which will be the one we prove. We also recall the notion of stable convergence.
\item In \Cref{sec:mclt}, we recall the (complex) martingale central limit theorem, the main tool used in the proof. 
\item In \Cref{sec:class}, we recall a wide class of multiplicative functions with which it will be convenient to work.
\item In \Cref{sec:struct}, we outline the application of the martingale central limit theorem, in particular defining $S_{x,\OurEpsilon,\delta}$, a truncated variant of $S_{x}$ that is essential to the argument. The proof of \Cref{thm:summainw} will be established up to the verification of three results:
\item In \Cref{sec:trunc} we establish \Cref{lem:negdelta}, which justifies working with $S_{x,\OurEpsilon,\delta}$ instead of $S_x$.
\item In \Cref{sec:Lindeberg} we establish \Cref{lem:lind}, which verifies the Lindeberg condition present in the relevant application of the martingale central limit theorem.
\item In \Cref{sec:bracket} we establish \Cref{prop:bracket}, which is the main technical result of this section. It computes the limit of the relevant bracket process. It is this proposition which relies on the properties of the measure $m_{y,t}(ds)$ established in \Cref{cor:mc-L1}.
\end{itemize}
\subsection{Generalisation}\label{sec:gen}
We prove the following slight generalisation of \Cref{thm:summain}. Given $\weight \colon \RR_{\ge0}\to \CC$ with compact support, let
\[S^{\weight}_x  = S^{\weight}_{x,f}:=\big(\sum_{n\le x} |f(n)|^2\big)^{-\frac{1}{2}}\sum_{n=1}^{\infty} \alpha(n) f(n) \weight\left(\frac{n}{x}\right).\]
For $\weight = \mathbf{1}_{[0,1]}$ this recovers $S_x$. For $\Re s >0$ we define
\[K_{\weight}(s) := \int_{0}^{\infty} \weight(x)x^{s-1}dx.\]
For $\weight=\mathbf{1}_{[0,1]}$ we have $K_{\weight}(s)=1/s$. If $\weight$ is a step function with compact support then $K_{\weight}(1/2+it)\ll 1/|t|$. By Plancherel's theorem, $\sqrt{2\pi} \|  \weight\|_2=\|K_{\weight}(1/2+it)\|_2$.
Let $f \in \Ma$ be a function satisfying the conditions of \Cref{thm:summain} for some $\theta \in (0,1)$. We claim that \[\lim_{x\to \infty}\EE [|S^{\weight}_x|^2]=\int_{0}^{\infty} |\weight(t)|^2 dt\] if $\weight$ is a step function with compact support. To see this, we use \eqref{eq:orth} to get 
\[\EE [|S^{\weight}_x|^2] = \sum_{n\ge 1}|f(n)|^2|\weight(n/x)|^2 / \sum_{n\le x}|f(n)|^2.\]
Necessarily $|\weight|^2$ is a linear combination of functions of the form $\mathbf{1}_{I}$ for compact intervals $I$, so the claim is reduced to $\sum_{n/x \in I}|f(n)|^2 \sim |I| \sum_{n\le x}|f(n)|^2$. This follows from Wirsing's theorem \cite{Wirsing} which says, for our $f$, that $\sum_{n\le x}|f(n)|^2 \sim C_f x(\log x)^{\theta-1}$ for some positive constant $C_f$.
\begin{thm}\label{thm:summainw}
Let $\weight\colon \RR_{\ge0}\to \CC$ be a step function with compact support such that $\|\weight\|_2>0$. Then
\[    S^{\weight}_x \xrightarrow[x \to \infty]{d} \sqrt{V^{\weight}_\infty} \ G \]
where 
\begin{equation}\label{eq:Vinfweight}
V^{\weight}_\infty:= \frac{1}{2\pi}\int_{\RR} \big|K_{\weight}\big(\tfrac{1}{2}+is\big)\big|^2 m_\infty(ds)
\end{equation}
is almost surely finite and strictly positive, and is independent of $G \sim \Na_\CC(0, 1)$. 
Moreover, the convergence in distribution is stable in the sense of \Cref{def:stable}.
\end{thm}
\begin{definition}\label{def:stable}
    Let $(\Omega, \Fa, \PP)$ be a probability space. We say a sequence of random variables $Z_n$ converges stably as $n \to \infty$ if $(Y, Z_n)$ converges in distribution as $n \to \infty$ for any bounded $\Fa$-measurable random variable $Y$.
\end{definition}
Similarly to the discussion in \cite[{\S}B.1]{GW}, \Cref{thm:summainw} can be rephrased as
\begin{equation}\label{eq:stable3w}
    \lim_{x \to \infty} \EE\big[Y \widetilde{h}(S^{\weight}_x)\big]
    = \EE\big[Y\widetilde{h}(\sqrt{V^{\weight}_\infty} G) \big]
\end{equation}
for any bounded continuous function $\widetilde{h}\colon \CC \to \RR$, and any bounded $\Fa_\infty$-measurable random variable $Y$ where $\Fa_\infty  := \sigma(\alpha(p), p=2,3,5,\ldots)$. The rest of the section is devoted to proving \Cref{thm:summainw}.

\subsection{Martingale central limit theorem}\label{sec:mclt}
As in \cite{GW}, our main tool for approaching \Cref{thm:summainw} is the martingale central limit theorem. This choice of tool is inspired by the works of Najnudel, Paquette and Simm \cite{NPS} and  Najnudel, Paquette, Simm and Vu \cite{NPSV} in the setting of  holomorphic multiplicative chaos. 
\begin{lem}\label{lem:mCLT}
For each $n$, let $(M_{n,j})_{j\le k_n}$ be a complex-valued, mean-zero and square integrable martingale with respect to the filtration $(\Fa_{j})_j$, and $\Delta_{n, j} := M_{n, j} - M_{n, j-1}$ be the corresponding martingale differences. Suppose the following conditions are satisfied.
\begin{itemize}
    \item[(a)] The conditional covariances converge: $\sum_{j =1}^{k_n} \EE\left[\Delta_{n, j}^2 | \Fa_{j-1}\right] \xrightarrow[n \to \infty]{p} 0$ and
    \[ \sum_{j =1}^{k_n} \EE\left[|\Delta_{n, j}|^2 | \Fa_{j-1}\right]  \xrightarrow[n \to \infty]{p} V_\infty.\]
    \item[(b)] The conditional Lindeberg condition holds: for any $\delta > 0$,
 $\sum_{j = 1}^{k_n} \EE\left[|\Delta_{n, j}|^2 \mathbf{1}_{\{|\Delta_{n, j}| > \delta\}} | \Fa_{j-1}\right] \xrightarrow[n \to \infty]{p} 0$.
\end{itemize}
Then $M_{n, k_n} \xrightarrow[n \to \infty]{d} \sqrt{V_\infty} \ G$ where $G \sim \Na_{\CC}(0,1)$ is independent of $V_\infty$, and the convergence in distribution is stable.
\end{lem}
See \cite[Theorem 3.2 and Corollary 3.1]{HH1980} for a proof of \Cref{lem:mCLT} for real-valued martingales (in the real case, the requirement $\sum_{j =1}^{k_n} \EE\left[\Delta_{n, j}^2 | \Fa_{j-1}\right] \xrightarrow[n \to \infty]{p} 0$ is omitted) and  \cite[{\S}B.2]{GW} 
for a standard adaptation to the complex case (cf.~\cite[Corollary 4.3]{NPS}). As explained e.g.~in \cite[Section 2.4]{GW}, the conditional Lindeberg condition above is implied by
\[\lim_{n\to \infty}\sum_{j=1}^{k_n}  \EE\left[|\Delta_{n, j}|^4\right] =0.\]
\subsection{A class of functions}\label{sec:class}
To state the precise conditions required by our arguments, it is convenient to introduce a class of functions $\Ma_{\theta}\subseteq \Ma$, which consists of all functions $g \in \Ma$ which satisfy the following mild conditions:
\begin{enumerate}
\item For every $n \in \NN$, $g(n) \ge 0$.
\item As $t \to \infty$, $ \sum_{p\le t} \frac{g(p)\log p}{p} \sim \theta \log t$.
\item For $t \ge 2$, $\sum_{p \le t} g(p) \ll \frac{t}{\log t}$.
\item The series $\sum_{p} \big(\frac{g^2(p)}{p^2}+\sum_{i \ge 2} \frac{g(p^i)}{p^i}\big)$ converges.
\item If $\theta\le 1$ then we also impose that for $t \ge 2$, $\sum_{p, i \ge 2,\, p^i \le t}g(p^i) \ll \frac{t}{\log t}$.
\item As $x \to \infty$,
\begin{equation}\label{eq:flatsum} 
\sum_{n \le x} g(n) \sim \frac{e^{-\gamma\theta}}{\Gamma(\theta)}\frac{x}{\log x}\prod_{p \le x} \left( \sum_{i=0}^{\infty} \frac{g(p^i)}{p^i}\right)
\end{equation}
where $\gamma$ is the Euler--Mascheroni constant.
\end{enumerate}
If $f \in \Ma$ satisfies the conditions in \Cref{thm:summain} then necessarily $|f|^2\in \Ma_{\theta}$ (the only non-trivial condition to verify is \eqref{eq:flatsum}, which is a direct consequence of Wirsing's theorem \cite{Wirsing}).
\subsection{Proof structure}\label{sec:struct}
Let $\weight\colon \RR_{\ge 0}\to \CC$ be a function that is supported on $[0,A]$ for some $A> 1$. Let $f\in \Ma$. Let $P(1) = 1$, and for $n >1$  denote by $P(n)$ the largest prime factor of $n$. We introduce parameters $\OurEpsilon,\delta>0$. For every $k\ge 0$ we define
\[ x_k=x_{k,\OurEpsilon,\delta} := x^{\OurEpsilon+k\delta}, \qquad I_k=I_k(x,\OurEpsilon,\delta) := (x_{k}, x_{k+1}].\]
We let $K$ be the smallest non-negative integer such that $x_{K+1} \ge Ax$. Strictly speaking $K$ depends on $\OurEpsilon$, $\delta$, $A$ and $x$, but once $x$ is sufficiently large (in terms of $\OurEpsilon$, $\delta$ and $A$) $K$ stabilises as $K=\lfloor (1-\OurEpsilon)/\delta \rfloor$. We consider the following truncation of $S^{\weight}_{x}$:
\begin{equation}\label{eq:trunc}
S^{\weight}_{x,\OurEpsilon,\delta} :=\big(\sum_{n \le x} |f(n)|^2 \big)^{-\frac{1}{2}}\sum_{k=0}^{K}\sum_{\substack{ P(n)\in I_k\\ P(n/P(n)) \le x_{k}}} \alpha(n) f(n)\weight\left(\frac{n}{x}\right).
\end{equation}
This definition is inspired by the construction introduced recently in the context of holomorphic multiplicative chaos \cite{NPSV}. In words, we discard the term $\alpha(n)f(n)\weight(n/x)$ from $S^{\weight}_x$ if $P(n)$ is small ($P(n)\le x_0=x^{\OurEpsilon}$), or if $P(n/P(n))$ is close to $P(n)$ (i.e.~$P(n)$ and $P(n/P(n))$ belong to the same $I_k$). Since any $n$ in the inner sum in the right-hand side of \eqref{eq:trunc}  satisfies $P(n)> x_{0}=x^{\OurEpsilon}$ and $P(n)\le Ax$, we may write
\[S^{\weight}_{x,\OurEpsilon,\delta} = \sum_{p>x^{\OurEpsilon}} Z'_{x,p}\]
where, if $p \in I_k$ for some (unique) $k \ge 0$, we define
\[ Z'_{x,p} := \big(\sum_{n \le x} |f(n)|^2\big)^{-\frac{1}{2}} \sum_{\substack{\substack{P(n)=p \\ P(n/P(n)) \le x_{k}}}} \alpha(n) f(n)\weight\left(\frac{n}{x}\right).\] 
Clearly $Z'_{x,p}\equiv 0$ for $p>Ax$.
\begin{lem}\label{lem:negdelta}
Let $\theta>0$. Let $f\in \Ma$ be a function such that $|f|^2 \in \Ma_{\theta} \cap \mathbf{P}_{\theta}$ and
\begin{equation}\label{eq:ppasump}
\sum_{p,i\ge 2,\, p^i \le x}  |f(p^i)|^2 = o\left(  \frac{x}{\log^2 x}\right)
\end{equation}
as $x \to \infty$. Then $\limsup_{\OurEpsilon\to 0^+} \limsup_{\delta\to 0^+}\limsup_{x \to \infty}\EE\big[|S^{\weight}_x - S^{\weight}_{x,\OurEpsilon,\delta}|^2\big]=0$ for every bounded $\weight\colon \RR_{\ge 0}\to \CC$ with compact support.
\end{lem}
\begin{lem}\label{lem:lind}
	Let $\theta >0$. Let $f\in \Ma$ be a function such that $|f|^2 \in \Ma_{\theta}$ and
\begin{align}
\label{eq:sumpi}&\sum_{p} \bigg(\sum_{i \ge 2} \frac{|f(p^i)|^2i}{p^i}\bigg)^2<\infty,\\
\label{eq:passum}&\sum_{p \le x} \left(|f(p^2)|^2+|f(p)|^4\right) \ll x^2(\log x)^{-2\theta-4}.
\end{align}
Then $\limsup_{x \to \infty}\sum_{p>x^{\OurEpsilon}} \EE [|Z'_{x,p}|^4] =0$ for every $\OurEpsilon,\delta>0$ and every bounded  $\weight\colon \RR_{\ge 0}\to \CC$ with compact support.
\end{lem} 
We introduce the $\sigma$-algebra
\[ \Fa_{y^-} := \sigma(\alpha(p), p <y).\]
\begin{proposition}\label{prop:bracket}
Let $\theta \in (0, 1)$. Let $f$ be a function satisfying the conditions of \Cref{thm:summain}. Let $\weight\colon \RR_{\ge 0}\to \CC$ be a step function with compact support. Then
\[T_{x,\OurEpsilon,\delta} :=\sum_{p>x^{\OurEpsilon}} \EE\left[ |Z'_{x,p}|^2 \mid \Fa_{p^-}\right]   \xrightarrow[x \to \infty]{p} C_{\OurEpsilon,\delta} V^{\weight}_\infty\]
where $V^{\weight}_\infty$ is defined in \eqref{eq:Vinfweight} and $\lim_{\OurEpsilon\to0^+}\lim_{\delta\to0^+} C_{\OurEpsilon,\delta}=1$.
\end{proposition}
\begin{proof}[Proof of \Cref{thm:summainw}]
Let $\OurEpsilon,\delta > 0$ be fixed. Let $f\in \Ma$ be a function satisfying the conditions of \Cref{thm:summain}. The reader may verify that this implies that $f$ satisfies the conditions in \Cref{lem:negdelta} and \Cref{lem:lind}. Suppose $\weight$ is a step function with compact support. We claim one may apply \Cref{lem:mCLT} to the martingale sequence associated to $S_{x, \OurEpsilon,\delta}=\sum_{p>x^{\OurEpsilon}}Z'_{x,p}$. Indeed, since $Z'_{x, p}$ is linear in $\alpha(p)$ by construction, for any $x > 0$ we automatically obtain $\sum_{p>x^{\OurEpsilon}} \EE[(Z'_{x, p})^2 | \Fa_{p^-}] =  0$ almost surely. By \Cref{lem:lind}, the conditional Lindeberg condition holds. Combining this with \Cref{prop:bracket} and applying \Cref{lem:mCLT}, we obtain that
$S^{\weight}_{x, \OurEpsilon,\delta} \xrightarrow[x \to \infty]{d}  \sqrt{C_{\OurEpsilon,\delta} V^{\weight}_\infty} G$ where the distributional convergence is also stable.

To establish the stable convergence of $S^{\weight}_{x}$ we use the formulation \eqref{eq:stable3w}: we would like to show that
\[    \lim_{x \to \infty} \EE\big[Y \widetilde{h}(S^{\weight}_x)\big]
    = \EE\big[Y\widetilde{h}(\sqrt{V^{\weight}_\infty} G) \big] \]
holds for any bounded $\Fa_\infty$-measurable random variable $Y$ and any bounded continuous function $\widetilde{h}\colon \CC \to \RR$. By a density argument, it suffices to establish this claim for bounded Lipschitz functions $\widetilde{h}$. Consider
\begin{align*}
    \left|\EE\left[Y \widetilde{h}(S^{\weight}_x)\right] - 
    \EE\left[Y\widetilde{h}(\sqrt{V^{\weight}_\infty} G) \right]\right|
    & \le \left|\EE\left[Y \widetilde{h}(S^{\weight}_{x, \OurEpsilon,\delta})\right] - 
    \EE\left[Y\widetilde{h}(\sqrt{C_{\OurEpsilon,\delta} V^{\weight}_\infty} G) \right]\right|\\
    & \qquad + \EE\left[Y \left|\widetilde{h}(S^{\weight}_x) - \widetilde{h}(S^{\weight}_{x, \OurEpsilon,\delta})\right|\right]
    + \EE\left[Y \left|\widetilde{h}(\sqrt{V^{\weight}_\infty} G) - \widetilde{h}(\sqrt{C_{\OurEpsilon,\delta} V^{\weight}_\infty} G)\right|\right].
\end{align*}
We see that the first term on the right-hand side converges to $0$ as $x \to \infty$ as a consequence of the stable convergence of $S^{\weight}_{x, \OurEpsilon,\delta}$, whereas the third term converges to $0$ as $\delta \to 0^+$ and then $\OurEpsilon \to 0^+$ as a consequence of dominated convergence. Meanwhile, by Cauchy--Schwarz, the middle term is bounded by
\[   \|\widetilde{h}\|_{\mathrm{Lip}} \EE[Y^2]^{\frac{1}{2}} 
     \EE\big[\big|S^{\weight}_x - S^{\weight}_{x, \OurEpsilon,\delta}\big|^2\big]^{\frac{1}{2}} \]
which goes to $0$ in the limit as $x \to \infty$, then $\delta \to 0^+$ and then $\OurEpsilon\to 0^+$ by \Cref{lem:negdelta}. Finally, $V_\infty$ is a.s.~finite because $\EE V_{\infty} = \int_{\RR} ds/(2\pi|1/2+is|^2)<\infty$ (since $\EE m_{\infty}(ds) = ds$), and it is a.s.~positive since $|K_{\weight}(1/2+it)|^2$ is strictly positive in some non-empty open interval $I$ and $m_\infty(I) > 0$ by its support property. This completes the proof.
\end{proof}

\subsection{Truncation: proof of \texorpdfstring{\Cref{lem:negdelta}}{Lemma \ref{lem:negdelta}}}\label{sec:trunc}
Let $\weight\colon \RR_{\ge 0}\to\CC$ be bounded and supported on $[0,A]$ for some $A\ge 1$. We introduce a simpler truncation of $S_x^{\weight}$:
\[S^{\weight}_{x,\OurEpsilon} :=\big(\sum_{n \le x} |f(n)|^2\big)^{-\frac{1}{2}} 
 \sum_{P(n)> x^{\OurEpsilon},\,P(n)^2 \nmid n} \alpha(n) f(n)\weight(n/x).\]
By Cauchy--Schwarz, 
\[\EE\left[|S^{\weight}_x - S^{\weight}_{x,\OurEpsilon,\delta}|^2\right] \ll \EE\left[|S^{\weight}_x - S^{\weight}_{x,\OurEpsilon}|^2\right]+\EE\left[|S^{\weight}_{x,\OurEpsilon} - S^{\weight}_{x,\OurEpsilon,\delta}|^2\right],\]
so it suffices to show that
\begin{equation}\label{eq:lim1}
\limsup_{\OurEpsilon\to 0^+} \limsup_{x \to \infty}\EE\left[|S^{\weight}_x - S^{\weight}_{x,\OurEpsilon}|^2\right]=0
\end{equation}
and
\begin{equation}\label{eq:lim2}
\limsup_{\delta\to 0^+}\limsup_{x \to \infty}\EE\left[|S^{\weight}_{x,\OurEpsilon} - S^{\weight}_{x,\OurEpsilon,\delta}|^2\right]=0
\end{equation}
hold. By \eqref{eq:orth},
\[\EE\left[|S^{\weight}_x - S^{\weight}_{x,\OurEpsilon}|^2\right] = (\sum_{n\le x}|f(n)|^2)^{-1} \sum_{P(n) \le x^{\OurEpsilon}\text{ or } P(n)^2 \mid n}|f(n) \weight(n/x)|^2.\]
By our assumptions on $\weight$,
\begin{equation}\label{eq:Asupp}
\EE\left[|S^{\weight}_x - S^{\weight}_{x,\OurEpsilon}|^2\right] \ll_{\weight}(\sum_{n\le x}|f(n)|^2)^{-1} \sum_{\substack{n \le Ax \\ P(n) \le x^{\OurEpsilon}\text{ or } P(n)^2 \mid n}}|f(n)|^2.
\end{equation}
In \cite[Section 3.1]{GW} it was shown that 
\begin{equation}\label{eq:earliera}
\limsup_{\OurEpsilon\to 0^+}\limsup_{x\to \infty}(\sum_{n\le x}|f(n)|^2)^{-1} \sum_{\substack{n \le x \\ P(n) <x^{\OurEpsilon}\text{ or } P(n)^2 \mid n}}|f(n)|^2=0 
\end{equation}
holds under \eqref{eq:ppasump}. The limit \eqref{eq:lim1} follows from \eqref{eq:Asupp} and \eqref{eq:earliera} once we replace $x$ and $\OurEpsilon$ with $Ax$ and $2\OurEpsilon$ (respectively) in \eqref{eq:earliera} and observe $\sum_{n \le Ax} |f(n)|^2 \ll_A \sum_{n \le x}|f(n)|^2$ by \eqref{eq:flatsum} with $g=|f|^2$. 
We turn to \eqref{eq:lim2}.
By \eqref{eq:orth},
\begin{align*}
\EE\left[|S^{\weight}_{x,\OurEpsilon} - S^{\weight}_{x,\OurEpsilon,\delta}|^2\right] &=
(\sum_{n\le x}|f(n)|^2)^{-1} \sum_{k=0}^{K} \sum_{ x_{k+1}\ge P(n)>P(n/P(n))> x_k}|f(n)\weight(n/x)|^2\\
&\ll (\sum_{n\le x}|f(n)|^2)^{-1} \sum_{k=0}^{K} \sum_{\substack{n\le Ax \\ x_{k+1}\ge P(n)>P(n/P(n))> x_k}}|f(n)|^2=\sum_{k=0}^{K} B_k
\end{align*}
for 
\[ B_{k} :=(\sum_{n\le x}|f(n)|^2)^{-1} \sum_{\substack{n \le Ax\\ x_{k+1}\ge P(n)>P(n/P(n))> x_{k}}}|f(n)|^2.\]
We denote $P(n)$ by $p$, $P(n/P(n))$ by $q$, and express $B_{k}$ as
\[ B_{k}=(\sum_{n\le x}|f(n)|^2)^{-1} \sum_{p\in I_k}\sum_{q \in (x_{k},p)} \sum_{\substack{n \le Ax\\ P(n)=p\\ P(n/p)=q} }|f(n)|^2=B_{k,1}+B_{k,2}\]
where 
\begin{align*}
B_{k,1}&:=(\sum_{n\le x}|f(n)|^2)^{-1} \sum_{p \in I_k} \sum_{q \in (x_{k},p)}\sum_{\substack{n \le Ax \\ P(n)=p\\ P(n/p)=q,\, q^2 \nmid n}}|f(n)|^2,\\
B_{k,2}&:=(\sum_{n\le x}|f(n)|^2)^{-1} \sum_{p \in I_k} \sum_{q \in (x_{k},p)}\sum_{\substack{n \le Ax \\ P(n)=p\\ P(n/p)=q,\,q^2 \mid n}}|f(n)|^2.
\end{align*}
To estimate $B_{k,1}$ we consider several ranges of $k$.  If $x_k\ge (Ax)^{\frac{1}{2}}$ then the conditions $p\in I_k$ and $q\in (x_k,p)$ force $pq > x_k^2\ge Ax$, so $B_{k,1}=B_{k,2}=0$. From now on we assume $x_k < (Ax)^{\frac{1}{2}}$. If we discard the conditions on $P(n/p)$ in $B_{k,1}$ we obtain the bound
\[ B_{k,1} \le (\sum_{n \le x}|f(n)|^2)^{-1}  \sum_{\substack{n\le Ax\\ P(n) \in I_k}} |f(n)|^2=(\sum_{n \le x}|f(n)|^2)^{-1} \big( \sum_{\substack{n\le Ax \\ P(n) \le x_{k+1}}} |f(n)|^2- \sum_{\substack{n\le Ax\\ P(n) \le x_k}} |f(n)|^2\big).\]
By the work of de Bruijn and van Lint \cite{dBvL}, one has
\[ \lim_{x\to \infty} (\sum_{n \le x}|f(n)|^2)^{-1} \sum_{n\le Ax ,\, P(n) \le x^c} |f(n)|^2=A\widetilde{\rho_{\theta}}(1/c)\]
for every $c\in (0,1]$, where $\widetilde{\rho_{\theta}}\colon [1,\infty)\to (0,\infty)$ is a differentiable function with bounded derivative. This implies that  
\[ \lim_{x\to \infty}  (\sum_{n \le x}|f(n)|^2)^{-1} \sum_{\substack{n\le Ax \\ P(n) \in I_k}} |f(n)|^2 =A(\widetilde{\rho_{\theta}}((\OurEpsilon+(k+1)\delta)^{-1}) -\widetilde{\rho_{\theta}}((\OurEpsilon+k\delta)^{-1}))\ll_{\OurEpsilon} A\delta.\]
It follows that $\limsup_{\delta\to0^+} \limsup_{x\to \infty} B_{k,1}=0$ holds for each $\OurEpsilon>0$. In particular, as there are at most $O_A(1)$ values of $k$ such that $(Ax)^{\frac{1}{2}} > x_k\ge x^{\frac{1}{2}-2\delta}$, 
\[ \limsup_{\delta\to0^+} \limsup_{x\to \infty}\sum_{k\ge 0:\, x_k \ge x^{\frac{1}{2}-2\delta}} B_{k,1} =0\]
holds for every $\OurEpsilon>0$. Next we suppose that $x_k < x^{\frac{1}{2}-2\delta}$. We may write $B_{k,1}$ as 
\begin{equation}\label{eq:bk1new}
B_{k,1}=(\sum_{n\le x}|f(n)|^2)^{-1} \sum_{p \in I_k}|f(p)|^2 \sum_{q \in (x_{k},p)} |f(q)|^2 \sum_{\substack{m \le Ax/(pq) \\ P(m)<q}}|f(m)|^2
\end{equation}
where $m$ stands for $n/(pq)$. We apply  \eqref{eq:flatsum} with $g=|f|^2$ to estimate $\sum_{n\le x}|f(n)|^2$, obtaining
\[  (\sum_{n \le x}|f(n)|^2)^{-1} \sum_{m \le Ax/(pq)}|f(m)|^2 \ll\frac{\log (Ax)}{\log (Ax/(pq))}\frac{1}{pq} \ll \frac{\log (Ax)}{\log (Ax/x_{k+1}^2)} \frac{1}{pq}\]
if $p\in I_k$ and $q \in (x_k,p)$. 
Thus, discarding the condition $P(m)<q$ in \eqref{eq:bk1new}, we get
\[	B_{k,1} \ll \frac{\log (Ax)}{\log (Ax/x_{k+1}^2)} \sum_{p \in I_k}\frac{|f(p)|^2}{p}  \sum_{q \in (x_k,p)}\frac{|f(q)|^2}{q}\le \frac{\log (Ax)}{\log (Ax/x_{k+1}^2)}\bigg( \sum_{p \in I_k} \frac{|f(p)|^2}{p}\bigg)^2.\]
Since $|f|^2 \in \mathbf{P}_{\theta}$,
\begin{align*} 
\sum_{p \in I_k} \frac{|f(p)|^2}{p} &=\theta(  \log \log x_{k+1}- \log \log x_k) + o_{x_k\to \infty}(1) \\
& =\theta \log \bigg(1+\frac{\delta}{\OurEpsilon+k\delta}\bigg) + o_{x_k\to \infty}(1)\ll \frac{\delta}{\OurEpsilon}+o_{x_k\to \infty}(1)
\end{align*}
where the terms $o_{x_k\to \infty}(1)$ go to $0$ as $x_k\to \infty$. In particular,
\begin{align*}
\sum_{k\ge 0:\, x_k < x^{\frac{1}{2}-2\delta}} B_{k,1} &\ll    \sum_{k\ge 0:\, x_k < x^{\frac{1}{2}-2\delta}} \frac{\log (Ax)}{\log (Ax/x_{k+1}^2)} \left(  \frac{\delta}{\OurEpsilon} + o_{x_k\to\infty}(1)\right)^2\\
&\ll  o_{x\to\infty}(1)+\frac{\delta^2}{\OurEpsilon^2} \sum_{k\ge 0:\, \OurEpsilon+k\delta < \frac{1}{2}-2\delta} \frac{1}{\frac{1}{2}-\delta(k+1)-\OurEpsilon} \ll o_{x\to\infty}(1) +  \frac{\delta}{\OurEpsilon^2} \log \frac{1}{\delta}
\end{align*}
where the terms $o_{x\to\infty}(1)$ go to $0$ as $x\to \infty$. It follows that
\[ \limsup_{\delta \to 0^+}\limsup_{x\to \infty}\sum_{k\ge 0:\, x_k < x^{\frac{1}{2}-2\delta}} B_{k,1} = 0\]
holds for every $\OurEpsilon>0$. We turn to estimate $B_{k,2}$. We denote the multiplicity of $q$ in $n$ by $i$ (i.e.~$q^i \mid n$ but $q^{i+1} \nmid n$), so
\begin{align*}
	B_{k,2} &= (\sum_{n \le x}|f(n)|^2)^{-1} \sum_{p \in I_k}|f(p)|^2\sum_{q\in (x_k,p)}  \sum_{i \ge 2} |f(q^i)|^2 \sum_{\substack{m \le Ax/(pq^i)\\ P(m)<q}}|f(m)|^2\\
    &\le  (\sum_{n \le x}|f(n)|^2)^{-1} \sum_{p \in I_k}|f(p)|^2\sum_{q\in (x_k,p)}  \sum_{i \ge 2} |f(q^i)|^2 \sum_{m \le Ax/(pq^i)}|f(m)|^2
\end{align*}
where $m$ stands for $n/(pq^i)$, and in the inequality we discarded the condition $P(m)<q$. We apply  \eqref{eq:flatsum} twice with $g=|f|^2$ to estimate $\sum_{n\le x}|f(n)|^2$ (once with $x$ and a second time with $Ax/(pq^i)$), obtaining
\[  (\sum_{n \le x}|f(n)|^2)^{-1} \sum_{m \le Ax/(pq^i)}|f(m)|^2 \ll \frac{\log x}{pq^i}\]
uniformly in $p$, $q$ and $i$. Thus, 
\begin{equation}\label{eq:Bk}
	B_{k,2} \ll \log x \sum_{p \in I_k}\frac{|f(p)|^2}{p}  \sum_{\substack{q \in (x_k,p) \\ i \ge 2: \, pq^i \le Ax}} \frac{|f(q^i)|^2}{q^i}\le \log x \sum_{p \in I_k}\frac{|f(p)|^2}{p} \sum_{\substack{q > x_k\\  i \ge 2: \, q^i \le Ax/x_k}}\frac{|f(q^i)|^2}{q^i}.
\end{equation}
We have $\sum_{p \in I_k} |f(p)|^2/p \ll_{\OurEpsilon,\delta} 1$ since $|f|^2 \in \mathbf{P}_{\theta}$. Moreover, the $q$-sum in the right-hand side of \eqref{eq:Bk} is $o_{x\to\infty}(1/\log x)$ as $x\to \infty$ by \eqref{eq:ppasump}. This means $\lim_{x\to \infty}B_{k,2} =0$ as $x\to \infty$. Since $K\ll_{\OurEpsilon,\delta}1$, this implies \[\lim_{x\to \infty} \sum_{k=0}^{K}B_{k,2} = 0\]
holds for every $\OurEpsilon,\delta>0$. Collecting the results on $B_{k,1}$ and $B_{k,2}$, we conclude that \eqref{eq:lim2} holds.
\subsection{Lindeberg condition: proof of \texorpdfstring{\Cref{lem:lind}}{Lemma \ref{lem:lind}}}
\label{sec:Lindeberg}
Suppose $\weight\colon \RR_{\ge 0}\to\CC$ is bounded and supported on $[0,A]$ for $A \ge 1$. By \eqref{eq:orth},
\[\sum_{p>x^{\OurEpsilon}} \EE\left[|Z'_{x,p}|^4\right]=\big(\sum_{n\le x}|f(n)|^2\big)^{-2}\sum_{k=0}^{K}\sum_{p\in I_k}G(x,p,k)\]
for 
\[ G(x,p,k):=\sum_{\substack{ab=cd\\ P(a)=P(b)=P(c)=P(d)=p\\  P(a/p),P(b/p),P(c/p),P(d/p)\le x_k}} f(a)f(b)\overline{f(c)f(d)}\weight(a/x)\weight(b/x)\overline{\weight(c/x)\weight(d/x)}.\]
One can write $G(x,p,k)$ as
\[G(x,p,k)=\sum_{\substack{P(m)=p,\, p^2 \mid m\\ P(m/p^2)\le x_k}}h(m), \qquad h(m):=\bigg|\sum_{\substack{ab=m\\ P(a)=P(b)}}f(a)f(b)\weight(a/x)\weight(b/x)\bigg|^2,\]
where $m$ stands for the common value of $ab$ and $cd$. We have the pointwise bound
$h \le (|f|*|f|)^2 \cdot (\sup |\weight|)^4$ so that
\[ G(x,p,k) \ll_{\weight} \sum_{\substack{p^2 \mid m \le A^2 x^2\\P(m)=p\\ P(m/p^2)\le x_k}}\fname(m), \quad \fname := (|f|*|f|)^2.\]
It follows that 
\begin{equation}\label{eq:magicupper2}
\begin{split}
\sum_{p>x^{\OurEpsilon}} \EE\left[|Z'_{x,p}|^4\right]&\ll_{\weight,\OurEpsilon,\delta} \big(\sum_{n\le x}|f(n)|^2\big)^{-2}\sum_{k=0}^{K} \sum_{p\in I_k} \sum_{\substack{p^2 \mid m \le A^2 x^2 \\ P(m)=p\\P(m/p^2)\le x_k}} \fname(m)\\
&\le \big(\sum_{n\le x}|f(n)|^2\big)^{-2}\sum_{p>x^{\OurEpsilon}} \sum_{\substack{p^2 \mid \mid m \le A^2 x^2 \\ P(m)=p}} \fname(m),
\end{split}
\end{equation}
where we replaced the condition on $P(m/p^2)$ with the less strict condition $p^2 \mid \mid m$, that indicates that $p^2$ divides $m$ but $p^3$ does not. In \cite[Section 3.2]{GW} it was shown (under \eqref{eq:sumpi}--\eqref{eq:passum}) that
\begin{equation}\label{eq:prevlind} \lim_{x\to \infty} (\sum_{n\le x}|f(n)|^2)^{-2}\sum_{p>x^{\OurEpsilon}} \sum_{\substack{p^2 \mid \mid m \le  x^2 \\ P(m)=p}} \fname(m)=0.
\end{equation}
This implies that the right-hand side of \eqref{eq:magicupper2} tends to $0$ as $x\to \infty$, once we replace $x$ and $\OurEpsilon$ with $Ax$ and $\OurEpsilon/2$ (respectively) in \eqref{eq:prevlind} and recall $\sum_{n \le x}|f(n)|^2 \gg_A \sum_{n \le Ax}|f(n)|^2$ by \eqref{eq:flatsum} with $g=|f|^2$.
\subsection{Bracket process: proof of \texorpdfstring{\Cref{prop:bracket}}{Proposition \ref{prop:bracket}}}\label{sec:bracket}
Suppose $f$ and $\weight$ satisfy the conditions of \Cref{prop:bracket}. By Wirsing's theorem \cite{Wirsing},
\begin{equation}\label{eq:wirsing}
\sum_{n \le x}|f(n)|^2 \sim \frac{C_{|f|^2}}{\Gamma(\theta)} x (\log x)^{\theta-1}
\end{equation}
where
\[ C_g:= \prod_{p} \left(\sum_{i=0}^{\infty} \frac{g(p^i)}{p^i}\right)\left( 1-\frac{1}{p}\right)^{\theta}.\]
The following lemma is proved below.
\begin{lem}\label{lem:tshift}
Let $\theta>0$. Suppose $g\in \mathbf{P}_{\theta}$ is a non-negative multiplicative function satisfying
\[ \sum_{p} \bigg(\frac{g^2(p)}{p^2} +\sum_{i\ge 2}\frac{g(p^i) i}{p^i}\bigg) < \infty.\]
Then, for any $t\ge 0$, we have
\[ \sum_{P(n)\le y} \frac{g(n)}{n^{1+t/\log y}} \sim C_g  \exp\bigg(\theta \gamma -\theta \int_{0}^{t} \frac{1-e^{-s}}{s}ds\bigg) (\log y)^{\theta}, \qquad y\to \infty.\]
\end{lem}
We introduce
\[ s_{t,y} := t^{-\frac{1}{2}}\sum_{P(n) \le y}\alpha(n) f(n)\weight(n/t).\]
We use \Cref{cor:mc-L1} and  \Cref{lem:tshift} to deduce the following result, also proved below.
\begin{lem}\label{lem:plancherelapp}
Let $f$ and $\weight$ be as in \Cref{prop:bracket}. Then, for every $r>0$,
\[C_{|f|^2}^{-1} (\log y)^{-\theta} \int_{0}^{\infty}  \frac{|s_{t,y}|^2 dt}{t^{1+r/\log y}}\xrightarrow[y \to \infty]{p}    \exp\bigg( \theta\gamma -\theta \int_{0}^{r} \frac{1-e^{-s}}{s}ds\bigg) \, V^{\weight}_\infty.\]
\end{lem}
We recall the definition of the $\theta$th Dickman function $\rho_{\theta}\colon (0,\infty)\to(0,\infty)$. For $t\le 1$ it is given by $\rho_{\theta}(t)=t^{\theta-1}/\Gamma(\theta)$ and for $t> 1$ by $t\rho_{\theta}(t)=\theta \int_{t-1}^{t}\rho_{\theta}(v)dv$. It is well known that \cite[Equation (2.5)]{Hensley}\cite[Equation (3.4)]{Smida}
\[ \int_{0}^{\infty} e^{-tv}\rho_{\theta}(v)dv= \exp\left( \theta \gamma - \theta \int_{0}^{t} \frac{1-e^{-s}}{s}ds\right)\]
holds for $t\ge 0$, so \Cref{lem:plancherelapp} may be written equivalently as
\begin{equation}\label{eq:dickmanequiv}
C_{|f|^2}^{-1} (\log y)^{-\theta} \int_{0}^{\infty}  \frac{|s_{t,y}|^2 dt}{t^{1+r/\log y}}\xrightarrow[y \to \infty]{p}   \bigg(\int_{0}^{\infty} e^{-rv}\rho_{\theta}(v)dv \bigg)\, V^{\weight}_\infty.
\end{equation}
For any given $a>0$, we substitute $y=x^a$, $t=x^{1-c}$ and $r=ka$ in \eqref{eq:dickmanequiv}, obtaining
\begin{equation}\label{eq:poly}
C_{|f|^2}^{-1} (\log x)^{1-\theta} \int_{\RR} e^{k(c-1)}|s_{x^{1-c},x^a}|^2 dc \xrightarrow[x \to \infty]{p}   \bigg(a^{\theta}\int_{0}^{\infty} e^{-kav}\rho_{\theta}(v)dv \bigg) V^{\weight}_\infty
\end{equation}
for $k=1,2,\ldots$. By a standard argument (given below) based on the Weierstrass approximation theorem, the following corollary is obtained from \eqref{eq:poly}.
\begin{cor}\label{cor:sandwich}
Let $f$ and $\weight$ be as in \Cref{prop:bracket}. Fix $a>0$ and $0<a_0<a_1$. If $Q\colon (0,\infty) \to [0,\infty)$ is equal to a continuous function times the indicator $\mathbf{1}_{[a_0,a_1]}$, then
\begin{equation}\label{eq:extend}
C_{|f|^2}^{-1} (\log x)^{1-\theta} \int_{\RR} Q(e^{c-1})|s_{x^{1-c},x^a}|^2 dc \xrightarrow[x \to \infty]{p}   \bigg(a^{\theta}\int_{0}^{\infty} Q(e^{-av})\rho_{\theta}(v)dv \bigg) V^{\weight}_\infty. 
\end{equation}
\end{cor}
We now connect \Cref{cor:sandwich} with the bracket process. If $m_1$ and $m_2$ satisfy $P(m_i)<p$ then
\[\EE\left[ \alpha(pm_1 ) \overline{\alpha(pm_2 )}\mid \Fa_{p^-}\right] = \mathbf{1}_{m_1=m_2}\]
and so we can express $T_{x,\OurEpsilon,\delta}$ as
\[T_{x,\OurEpsilon,\delta} = \big(\sum_{n \le x}|f(n)|^2\big)^{-1}\sum_{k=0}^{K}\sum_{p \in I_k} |f(p)|^2 \bigg| \sum_{\substack{P(m)\le x_{k}}}f(m)\alpha(m)\weight\left(\frac{m}{x/p}\right)\bigg|^2=\sum_{k=0}^{K} T_{x,\OurEpsilon,\delta,k}\]
for 
\[T_{x,\OurEpsilon,\delta,k}:=\big(\sum_{n\le x}|f(n)|^2\big)^{-1} \sum_{p \in I_k} |f(p)|^2 \bigg| \sum_{P(m)\le x_{k}}f(m)\alpha(m)\weight\left(\frac{m}{x/p}\right)\bigg|^2.\]
We introduce
\[R_{a,b}:=\sum_{p \in (x^a,x^b]}\frac{|f(p)|^2}{p}|s_{x/p,x^a}|^2.\]
In this notation,
\begin{equation}\label{eq:TR}
T_{x,\OurEpsilon,\delta,k} = \big(\sum_{n\le x}|f(n)|^2\big)^{-1} \cdot x \cdot R_{\OurEpsilon+k\delta,\OurEpsilon+(k+1)\delta}.
\end{equation}
\begin{lem}\label{lem:closel1}
Let $f$ and $\weight$ be as in \Cref{prop:bracket}. Let $0<a<b$. The statement
\[ (\sum_{n\le x}|f(n)|^2)^{-1} \cdot x \cdot R_{a,b} \xrightarrow[x \to \infty]{p}C \, V^{\weight}_\infty \]
for some constant $C$ is equivalent to 
\[(\sum_{n\le x}|f(n)|^2)^{-1} \cdot x \cdot  \int_{x^{a}}^{x^{b}}\theta\frac{|s_{x/t,x^a}|^2}{t\log t} d t \xrightarrow[x \to \infty]{p}C\, V^{\weight}_\infty.\]
\end{lem}
Using \Cref{lem:closel1} and the change of variables $t=x^c$, we find from \eqref{eq:TR} and \eqref{eq:wirsing} that the statement
\[ T_{x,\OurEpsilon,\delta,k} \xrightarrow[x \to \infty]{p}C(k,\OurEpsilon,\delta) \, V^{\weight}_\infty \]
for some $C(k,\OurEpsilon,\delta)$ is equivalent to 
\begin{equation}\label{eq:equiv} \Gamma(\theta)C_{|f|^2}^{-1} (\log  x)^{1-\theta}  \int_{\OurEpsilon+k\delta}^{\OurEpsilon+(k+1)\delta} \theta c^{-1}|s_{x^{1-c},x^{\OurEpsilon+k\delta}}|^2 d c \xrightarrow[x \to \infty]{p}C(k,\OurEpsilon,\delta) \, V^{\weight}_\infty.
\end{equation}
By \Cref{cor:sandwich} with $Q(e^{c-1})=c^{-1}\mathbf{1}_{c\in [\OurEpsilon+k\delta,\OurEpsilon+(k+1)\delta]}$ and $a=\OurEpsilon+k\delta$, \eqref{eq:equiv} holds with
\[ C(k,\OurEpsilon,\delta) = \Gamma(\theta+1) (\OurEpsilon+k\delta)^{\theta} \int_{\max\{\frac{1-(\OurEpsilon+(k+1)\delta)}{\OurEpsilon+k\delta},0\}}^{\max\{\frac{1-(\OurEpsilon+k\delta)}{\OurEpsilon+k\delta},0\}} \rho_{\theta}(v) (1-(\OurEpsilon+k\delta)v)^{-1}dv.\]
It follows, since $T_{x,\OurEpsilon,\delta} = \sum_{k=0}^{K} T_{x,\OurEpsilon,\delta,k}$, that
\[T_{x,\OurEpsilon,\delta}  \xrightarrow[x \to \infty]{p} C_{\OurEpsilon,\delta} \, V^{\weight}_\infty \]
with
\begin{equation}\label{eq:CG}
C_{\OurEpsilon,\delta}:=\sum_{k=0}^{K} C(k,\OurEpsilon,\delta)=\Gamma(\theta+1) \int_{0}^{\frac{1}{\OurEpsilon}-1} \rho_{\theta}(v) G_{\OurEpsilon,\delta}(v) dv
\end{equation}
where
\[G_{\OurEpsilon,\delta}(v):=\sum_{\substack{0\le k \le K\\ \frac{1-\delta}{v+1}< \OurEpsilon+k\delta\le \frac{1}{v+1}  }}\frac{(\OurEpsilon+k\delta)^{\theta}}{1-(\OurEpsilon+k\delta)v}\]
for $v\in (0,1/\OurEpsilon-1)$. We write $\{t\}$ for the fractional part of $t$: $\{t\}=t-\lfloor t\rfloor$. Then
\begin{equation}\label{eq:Gsimp}
G_{\OurEpsilon,\delta}(v)=\frac{\left( \frac{1+O(\delta)}{v+1}\right)^{\theta} }{1-\frac{1+O(\delta)}{v+1}v}\mathbf{1}_{\left\{(\frac{1}{v+1}-\OurEpsilon)/\delta \right\} < \frac{1}{v+1}}=(1+O_{\OurEpsilon}(\delta))(1+v)^{1-\theta}\mathbf{1}_{\left\{(\frac{1}{v+1}-\OurEpsilon)/\delta \right\} < \frac{1}{v+1}}
\end{equation}
for $v\in (0,1/\OurEpsilon-1)$. We need one more lemma in order to estimate $C_{\OurEpsilon,\delta}$.
\begin{lem}\label{lem:equid}
Let $A,a>0$. If $g \colon [0,A]\to \RR$ is continuous then
\[\lim_{\delta\to0^+}\int_{0}^{A} g(v) \left(\mathbf{1}_{\left\{(\frac{1}{v+1}-a)/\delta \right\} < \frac{1}{v+1}} - \frac{1}{v+1}\right)dv =0.\]
\end{lem}
Taking $\delta \to 0^+$, and using \Cref{lem:equid} with $g(v)=\rho_{\theta}(v)(1+v)^{1-\theta}$, we obtain from \eqref{eq:CG} and \eqref{eq:Gsimp} that
\begin{equation}\label{eq:Ceps}
\lim_{\delta\to 0^+}C_{\OurEpsilon,\delta}= \Gamma(\theta+1)\int_{0}^{\frac{1}{\OurEpsilon} - 1}\rho_{\theta}(v)(1+v)^{-\theta}dv.
\end{equation}
The right-hand side of \eqref{eq:Ceps} was shown in \cite[Equation (4.38)]{NPS} to tend to $1$ as $\OurEpsilon\to 0^+$.
\subsubsection{Proof of \texorpdfstring{\Cref{lem:tshift}}{Lemma \ref{lem:tshift}}}
For $t=0$, the claim follows by writing
\[ \sum_{P(n)\le y} \frac{g(n)}{n} = \prod_{p\le y} \left( \sum_{i=0}^{\infty}\frac{g(p^i)}{p^i}\right)\left(1-\frac{1}{p}\right)^{\theta} \prod_{p\le y}\left(1-\frac{1}{p}\right)^{-\theta},\]
and the right-hand side is asymptotic to $C_g e^{\gamma \theta} (\log y)^{\theta}$ by Mertens' theorem \cite[Theorem 2.7]{MV}. To prove the claim for $t>0$, it suffices to show that
\[ \lim_{y \to \infty} \bigg(  \sum_{P(n)\le y} g(n)/n^{1+t/\log y} \bigg)/\bigg(  \sum_{P(n)\le y} g(n)/n\bigg) = \exp\bigg( -\theta \int_{0}^{t} \frac{1-e^{-s}}{s}ds\bigg).\]
Let $a_p(t):=g(p)/p^{1+t}$ and $b_p(t) := \sum_{i \ge 2} g(p^i)/p^{i(1+t)}$, so that we want to prove
\[ \lim_{y\to \infty}\prod_{p\le y} \frac{1+a_p(t/\log y)+b_p(t/\log y)}{1+a_p(0)+b_p(0)}= \exp\bigg( -\theta \int_{0}^{t} \frac{1-e^{-s}}{s}ds\bigg).\]
For $t\ge 0$, $a_p(t)$ and $b_p(t)$ are non-negative decreasing  functions. Moreover, $b_p(0)$ and $a_p(0)$ are uniformly bounded by assumption. Thus
\[\frac{1+a_p(t)+b_p(t)}{1+a_p(0)+b_p(0)} = \exp(a_p(t)-a_p(0) + O(|b_p(t)-b_p(0)|+|a_p(t)-a_p(0)|\cdot  (|a_p(0)|+|b_p(0)|)))\]
uniformly in $p$ and $t\ge 0$. Using $p^{-t}=1+O(t\log p)$ for $t\ge 0$ we have
\[ \sum_{p\le y} |a_p(t)-a_p(0)| \cdot |a_p(0)| \ll  t\sum_{p\le y} \frac{ g^2(p)\log p}{p^2} = o_{y\to \infty}( t\log y) \]
since $\sum_{p} g^2(p)/p^2$ converges. In particular, $\lim_{y\to\infty}\sum_{p\le y} |a_p(t/\log y)-a_p(0)| |a_p(0)| =0$ for fixed $t> 0$. Similarly, 
\[ \sum_{p\le y}|a_p(t)-a_p(0)| \cdot |b_p(0)| \ll t \sum_{p\le  y} \frac{g(p)}{p} \sum_{i\ge 2}\frac{g(p^i) \log p}{p^i}= o_{y\to\infty} (t\log y)\]
since $\sum_{p}g(p)/p \sum_{i\ge 2} g(p^i)/p^i$ converges, and 
\[ \sum_{p\le y}|b_p(t)-b_p(0)| \ll t \sum_{p\le  y} \sum_{i\ge 2}\frac{g(p^i) \log(p^i)}{p^i}= o_{y\to\infty} (t\log y)\]
since $\sum_{p}\sum_{i\ge 2}g(p^i)i/p^i$ converges. It remains to prove that
\[ \lim_{y\to \infty}\sum_{p\le y} (a_p(t/\log y)-a_p(0))= -\theta \int_{0}^{t} \frac{1-e^{-s}}{s}ds\]
for $t\ge 0$. This follows from the fact that $g \in \mathbf{P}_{\theta}$ and integration by parts.

\subsubsection{Proof of \texorpdfstring{\Cref{lem:plancherelapp}}{Lemma \ref{lem:plancherelapp}}}
By \eqref{eq:wirsing}, the expectation of $\int_{0}^{\infty}  |s_{t,y}|^2 t^{-1-r} dt$ is finite and so the integral converges almost surely. By Plancherel’s identity \cite[Theorem 5.4]{MV},
\begin{equation}\label{eq:Planch} \int_{0}^{\infty} \frac{ |s_{t,y}|^2 dt}{t^{1+r}}=\frac{1}{2\pi}\int_{\RR} \big| A_y\big(\tfrac{1}{2}+\tfrac{r}{2}+it\big) K_{\weight}\big(\tfrac{1}{2}+\tfrac{r}{2}+it\big)\big|^2 dt
\end{equation}
holds a.s.~for $r>0$. A direct computation using \eqref{eq:orth} shows that
\begin{equation}\label{eq:directcomp}
\EE [\big|A_y(\tfrac{1}{2}+\tfrac{r}{2}+it\big)|^2] = \sum_{P(n)\le y}\frac{|f(n)|^2}{n^{1+r}}
\end{equation}
for every $t\in \RR$. By \eqref{eq:Planch} and \eqref{eq:directcomp},
\begin{equation}\label{eq:corapp}
\begin{split}
&\bigg(\sum_{P(n)\le y}\frac{|f(n)|^2}{n^{1+r/\log y}}\bigg)^{-1} \int_{0}^{\infty}  \frac{|s_{t,y}|^2 dt}{t^{1+r/\log y}}\\
&\qquad =\frac{1}{2\pi}\int_{\RR} \big|K_{\weight}\big(\sigma_{y^{1/r}}+is\big)\big|^2m_{y,y^{1/r}}(ds)\\
&\qquad = \frac{1}{2\pi}\int_{\RR} \big|K_{\weight}\big(\tfrac{1}{2}+is\big)\big|^2m_{y, y^{1/r}}(ds)+\frac{1}{2\pi}\int_{\RR} \big(\big|K_{\weight}\big(\sigma_{y^{1/r}}+is\big)\big|^2 - \big|K_{\weight}\big(\tfrac{1}{2}+is\big)\big|^2\big)m_{y, y^{1/r}}(ds),
\end{split}
\end{equation}
where $\sigma_t$ and $m_{y,t}$ are introduced in \Cref{thm:mc-L1}. Since $\weight$ is a step function with compact support, for $\Re s > 0$ we have $K_{\weight}(s)=F(s)/s$ where $F$ is a linear combination of exponential functions $a^s$ ($a>0$). In particular, $|K_{\weight}(1/2+it)|^2$ is in $L^1(\RR)$. Thus, \Cref{cor:mc-L1} implies that
\begin{equation}\label{eq:thmcor}
\frac{1}{2\pi}\int_{\RR} \big|K_{\weight}\big(\tfrac{1}{2}+is\big)\big|^2m_{y, y^{1/r}}(ds) \xrightarrow[y \to \infty]{p}   \frac{1}{2\pi}\int_{\RR} \big|K_{\weight}\big(\tfrac{1}{2}+is\big)\big|^2 m_\infty(ds).
\end{equation}
The right-hand side of \eqref{eq:thmcor} is $V^{\weight}_\infty$. The $n$-sum in the left-hand side of \eqref{eq:corapp} can be simplified using \Cref{lem:tshift}. It remains to show that 
\[\int_{\RR} \big(\big|K_{\weight}\big(\sigma_{y^{1/r}}+is\big)\big|^2 - \big|K_{\weight}\big(\tfrac{1}{2}+is\big)\big|^2\big)m_{y, r}(ds) \xrightarrow[y \to \infty]{p}  0.\]
Taking expectations, it suffices to show that
\begin{equation}\label{eq:remain}
\lim_{y \to \infty}\int_{\RR} \left|\big|K_{\weight}\big(\sigma_{y^{1/r}}+is\big)\big|^2 - \big|K_{\weight}\big(\tfrac{1}{2}+is\big)\big|^2\right|ds=0.
\end{equation}
Since $K_{\weight}(s),K'_{\weight}(s)\ll 1/|s|$ for $\Re s \in [1/2,1]$, the mean value theorem yields
\begin{align*} \big|K_{\weight}\big(\sigma_{y^{1/r}}+is\big)\big|^2 - \big|K_{\weight}\big(\tfrac{1}{2}+is\big)\big|^2 &\ll (1+|s|)^{-1}\big(\big|K_{\weight}\big(\sigma_{y^{1/r}}+is\big)\big| - \big|K_{\weight}\big(\tfrac{1}{2}+is\big)\big|\big) \\
 &\ll (1+|s|)^{-2}|\sigma_{y^{1/r}}-1/2| \ll (1+|s|)^{-2}\frac{r}{\log y} 
\end{align*}
which implies \eqref{eq:remain} and concludes the proof of the lemma.
\subsubsection{Proof of \texorpdfstring{\Cref{cor:sandwich}}{Corollary \ref{cor:sandwich}}}
We first reduce to the case $a_1 \le 1$. To do so we need to establish \eqref{eq:extend} when $a_1>a_0\ge 1$. When $Q$ is supported on $[a_0,a_1]$ and $a_0\ge 1$, the right-hand side of \eqref{eq:extend} is $0$ so it suffices to show that the expectation of the absolute value of the left-hand side of \eqref{eq:extend} tends to $0$ as $x\to \infty$. If $a_0\ge 1$ then $Q(e^{c-1})=0$ for $c>1$. Moreover, if $\weight$ is supported on $[0,A]$ then $s_{x^{1-c},x^a}=0$ for $c< 1+ \tfrac{\log A}{\log x}$, and $s_{x^{1-c},x^a}$ is bounded (by a deterministic constant) when $c$ is between $1$ and $1+\tfrac{\log A}{\log x}$. This establishes \eqref{eq:extend} when $a_0\ge 1$. 

From now on suppose $a_1 \le 1$. By \eqref{eq:poly}, \eqref{eq:extend} already holds for any polynomial $Q$ with $Q(0)=0$. For our target $Q$ and each $\OurEpsilon > 0$ consider two continuous functions $Q_\pm$ constructed from $Q$ by linearly interpolating the jumps in an $\OurEpsilon$-neighbourhood of the end points of the interval $[a_0,a_1]$ such that $Q_-(z) \le Q(z) \le Q_+(z)$ and that $Q_\pm(z) = Q(z)=0$ whenever $[z-\OurEpsilon,z+\OurEpsilon]$ does not intersect $[a_0,a_1]$.

Observe that $Q_\pm(z) / z$ is a continuous function (if $\OurEpsilon<a_0$), and by Stone--Weierstrass there exist polynomials $\widetilde{\Pa}_\pm$ such that
\[\left| \widetilde{\Pa}_\pm(z) -  \left(\frac{Q_{\pm}(z)}{z} \pm \OurEpsilon\right) \right| \le \OurEpsilon\]
holds for  $z\in [0,1]$, and in particular the polynomials $\Pa_{\pm}(z) := z \widetilde{\Pa}_\pm(z)$ satisfy $\Pa_{\pm}(0)=0$,
\[\Pa_-(z) \le Q_-(z) \le Q(z) \le Q_+(z) \le \Pa_+(z) \qquad \forall z \in [0,1]
\qquad \text{and} \qquad
\sup_{z \in [0, 1]} |\Pa_\pm(z) - Q_\pm(z)| \le 2\OurEpsilon.\]
Let us now write, for any suitable test functions $g$,
\[\Ia_x( g)
:= C_{|f|^2}^{-1} (\log x)^{1-\theta}\int_{\RR} g(e^{c-1}) |s_{x^{1-c},x^a}|^2 d c,\qquad 
\Ia_\infty(g)
 :=\bigg( a^{\theta}  \int_{0}^{\infty} g(e^{-av}) \rho_{\theta}(v) dv \bigg) V^{\weight}_\infty. \]
For any $\delta > 0$,  we have
\begin{align*}
&  \PP\left(  \Ia_x(Q) - \Ia_\infty(Q)> \delta \right)\\
&\qquad \le  \PP\left(  \Ia_x(\Pa_+) - \Ia_\infty(\Pa_+) + \Ia_\infty(\Pa_+) - \Ia_\infty(\Pa_-)> \delta \right)\\
&\qquad \le  \PP\left(  \Ia_x(\Pa_+) - \Ia_\infty(\Pa_+) > \frac{\delta}{2}\right) + \PP\left( \Ia_\infty(\Pa_+) - \Ia_\infty(\Pa_-)>\frac{ \delta}{2} \right)
 \xrightarrow[x \to \infty]{}
\PP\left( \Ia_\infty(\Pa_+) - \Ia_\infty(\Pa_-) > \frac{\delta}{2}\right)
\end{align*}
since \eqref{eq:extend} holds for $Q=\Pa_{+}$. But then,
\[\Ia_\infty(\Pa_+) - \Ia_\infty(\Pa_-)  \le 4a^{\theta} V^{\weight}_\infty    \int_{0}^{\infty}\left[\OurEpsilon + \|Q\|_{\infty}\left(\mathbf{1}_{|e^{-av} - a_0| \le \OurEpsilon} + \mathbf{1}_{|e^{-av} - a_1| \le \OurEpsilon}\right) \right]\rho_{\theta}(v) dv
\xrightarrow[\OurEpsilon \to 0^+]{a.s.} 0,\]
which means
\[\lim_{x \to \infty}  \PP\left(  \Ia_x(Q) - \Ia_\infty(Q)> \delta \right) = 0
\qquad \text{and similarly} \qquad 
\lim_{x \to \infty} \PP\left(   \Ia_x(Q) - \Ia_\infty(Q) < - \delta \right) = 0,\]
i.e.~$\Ia_x(Q) \xrightarrow[x \to \infty]{p}\Ia_\infty(Q)$, as needed.
\subsection{Proof of \texorpdfstring{\Cref{lem:closel1}}{Lemma \ref{lem:closel1}}}
Let $\weight\colon \RR_{\ge0}\to \CC$ be a step function supported on $[0,A]$. We may assume $b \le 1$ since $s_{x/t,x^a}\equiv 0$ if $t> Ax$, and $s_{x/t,x^a}$ is bounded by a deterministic quantity for $t\in [x,Ax]$ (note $\int_{x}^{Ax}dt/(t\log t)$ and $\sum_{p\in [x,Ax]}|f(p)|^2/p$ are $O(1/\log x)$ since $|f|^2\in \mathbf{P}_{\theta}$).  

It suffices to show that 
\[ \limsup_{x\to \infty}(\log x)^{1-\theta}\EE\bigg| R_{a,b} - \int_{x^a}^{x^b} \theta \frac{|s_{x/t,x^a}|^2}{t\log t}dt\bigg|=0\]
where we used \eqref{eq:wirsing} to simplify the expression $(\sum_{n \le x}|f(n)|^2)^{-1}\cdot x$ which appears in \Cref{lem:closel1}. We show that the last limit is $0$ by fixing an arbitrary $r\in (0,1/2)$ and demonstrating that 
\begin{equation}\label{eq:demonstrate}
\limsup_{x\to \infty}(\log x)^{1-\theta}\EE\bigg| R_{a,b} - \int_{x^a}^{x^b} \theta \frac{|s_{x/t,x^a}|^2}{t\log t}dt\bigg| =O(r).
\end{equation}
We give the proof of \eqref{eq:demonstrate}, where $0<a<b\le 1$ is assumed. We need the following Lipschitz-type property of $s_{x,y}$: when $x_2 \ge x_1$,
\begin{equation}\label{eq:lip}
 |s_{x_1,y}-s_{x_2,y}|\le |s_{x_2,y}|(\sqrt{x_2/x_1}-1) + X_{x_1,x_2}/\sqrt{x_1}
\end{equation}
holds where $X_{x_1,x_2}$ is a non-negative random variable with $\EE[|X_{x_1,x_2}|^2]=\sum_{n}|f(n)|^2 |\weight(n/x_1)-\weight(n/x_2)|^2$. The proof is immediate from the equality \[\sqrt{x_1}s_{x_1,y}-\sqrt{x_2}s_{x_2,y} =\sum_{P(n)\le y} f(n)\alpha(n)(\weight(n/x_1)-\weight(n/x_2)),\]
the triangle inequality and \eqref{eq:orth}. We fix the function 
\[h(t)=rt\] 
and apply \eqref{eq:lip} with $x_1=x/p$ and $x_2=x/t$, obtaining that
\[|s_{x/p,x^a}|^2 = \frac{1}{h(p)} \int_{p-h(p)}^{p}(|s_{x/t,x^a}|^2 +Y_{x,p,t})dt \]
holds for the random variable $Y_{x,p,t}=|s_{x/p,x^a}|^2 - |s_{x/t,x^a}|^2$. We claim that 
\begin{equation}\label{eq:Ybnd}
\EE |Y_{x,p,t}| \ll (\log (2x/p))^{\theta-1} (\sqrt{p/t}-1) + (\log(2x/p))^{\frac{\theta-1}{2}} \sqrt{\sum_{n}|f(n)|^2|\weight(np/x)-\weight(nt/x)|^2/(x/p)}
\end{equation}
for $p-h(p)\le t\le p$. This follows from \eqref{eq:lip} by writing $Y_{x,p,t} = (|s_{x/p,x^a}|+|s_{x/t,x^a}|)(|s_{x/p,x^a}|-|s_{x/t,x^a}|)$, applying Cauchy--Schwarz and recalling
\begin{equation}\label{eq:cs}
\EE [|s_{x,y}|^2]\ll (\log(2+ x))^{\theta-1}
\end{equation}
by \eqref{eq:wirsing}. In fact, the following consequence of \eqref{eq:Ybnd} is sufficient: if $x/p$ is sufficiently large in terms of $r$, then
\begin{equation}\label{eq:Ybnd2}
\EE |Y_{x,p,t}| \ll (\log (2x/p))^{\theta-1} (\sqrt{p/t}-1 +\sqrt{r}) 
\end{equation}
holds uniformly for $t \in [p-h(p),p]$ and $p\le x$ (with implied constant independent of $r$, depending only on $|f|^2$). The bound \eqref{eq:Ybnd2} follows from \eqref{eq:Ybnd} by the fact that $\weight$ is a step function with compact support  and applications of \eqref{eq:wirsing}.

Next, let
\[ M(t) := \sum_{p\in (x^a,x^b]: \, t \in [p-h(p),p]} \frac{|f(p)|^2}{ph(p)} \ge 0\]
so that we have the identity
\begin{equation}\label{eq:Rabeq}
R_{a,b} =\int_{x^{a}-h(x^a)}^{x^b} |s_{x/t,x^a}|^2 M(t) dt + Z_{a,b} \quad \text{where} \quad Z_{a,b}:=\sum_{p\in (x^a,x^b]}\frac{1}{ph(p)}\int_{p-h(p)}^{p}Y_{x,p,t}dt.
\end{equation}
The proof of \eqref{eq:demonstrate} will follow by the triangle inequality from \eqref{eq:Rabeq} by showing the following three limits:
\begin{equation}\label{eq:3ints}
\begin{split}
&\lim_{x\to \infty}(\log x)^{1-\theta} \EE\int_{x^a-h(x^a)}^{x^a}|s_{x/t,x^a}|^2 M(t)dt=0,\\
&\lim_{x\to \infty} (\log x)^{1-\theta} \EE \int_{x^a}^{x^b} |s_{x/t,x^a}|^2 \big|M(t) - \frac{\theta}{t \log t}\big|dt =0,\\
&\limsup_{x\to \infty}(\log x)^{1-\theta}\EE |Z_{a,b}|\ll \sqrt{r}.
\end{split}
\end{equation}
The first  limit in \eqref{eq:3ints} follows from \eqref{eq:cs} and the bound $M(t) \ll 1/(t\log t)$ which holds for $t\in [x^a-h(x^a),x^b]$; this bound follows from the assumption $|f|^2 \in \mathbf{P}_{\theta}$. The second limit in \eqref{eq:3ints} follows (using \eqref{eq:cs}) from $M(t)\ll 1/(t\log t)$ together with the asymptotic formula
\[M(t) \sim \frac{\theta}{t \log t}\]
which holds as $x\to \infty$ uniformly for $t\in (x^{a},x^b(1-r))$, also thanks to the assumption $|f|^2 \in \mathbf{P}_{\theta}$. To treat the last limit in \eqref{eq:3ints} we use \eqref{eq:Ybnd2} (and Chebyshev's bound $\sum_{p\in [x,2x]}1\ll x/\log(2x)$ \cite[Theorem 2.4]{MV} is needed to treat the $p$-s in $Z_{a,b}$ for which $x/p$ is not sufficiently large in terms of $r$).

\subsubsection{Proof of \texorpdfstring{\Cref{lem:equid}}{Lemma \ref{lem:equid}}}
Through the change of variables $u=1/(v+1)$, the claim is the same as 
\[\lim_{\delta\to0^+}\int_{B}^{1} h(u) \left(\mathbf{1}_{\{(u-a)/\delta\}<u}-u\right)du =0\]
for any $0<B<1$ and any continuous $h$. Since Lipschitz functions are dense in $C([B,1])$, we may assume that $h$ is Lipschitz.

For each $k\in \ZZ$ we define $I_k := [B,1] \cap [k\delta+a,(k+1)\delta+a)$. We write $[B,1]$ as the disjoint union $\cup_{k\in J} I_k$ where $J:=\{k \in \ZZ : I_k \neq \emptyset\}$. The size of $J$ is $O(1+\delta^{-1})$. The above integral restricted to a given $I_k$ contributes at most $O(|I_k|)=O(\delta)$ . We have $I_k=[k\delta+a,(k+1)\delta+a)$ for any $k\in J\setminus \{\max J,\min J\}$. To conclude it suffices to show that the integral above restricted to $I_k$ where $k\in J\setminus \{\max J,\min J\}$ is $O(\delta^2)$. For such $k$ we write
\[ \int_{I_k} h(u) \left(\mathbf{1}_{\{(u-a)/\delta\}<u}-u\right)du = h(k\delta+a) \int_{I_k}  \left(\mathbf{1}_{\{(u-a)/\delta\}<u}-u\right)du  + O(\int_{I_k} |h(k\delta+a)-h(u)|du). \]
The first integral is $O(\delta^2)$ by a direct computation and the second one is $O(\delta^2)$ as well since $h$ is  Lipschitz.

\section*{Acknowledgments}
O.G.~is supported by the Israel Science Foundation (grant no.~2088/24). O.G.~is the incumbent of the Rabbi Dr.~Roger Herst Faculty Fellowship, which supported this work. M.D.W.~is supported by the Royal Society Research Grant RG\textbackslash R1\textbackslash 251187.

\appendix
\section{Supremum of stochastic processes}
\begin{thm}[{Generic chaining bound, cf.~\cite[Equation 2.47]{Tal2014}}] \label{thm:chaining}
Let $(X_t)_{t \in T}$ be a collection of zero-mean random variables indexed by elements of a metric space $(T, d)$ satisfying
\[    \PP\left(|X_s - X_t| \ge u\right) \le 2 \exp\left(-\frac{u^2}{2d(s, t)^2}\right) \qquad \forall s, t \in T, \quad u \ge 0.\]
Then there exists some absolute constant $C_1 > 0$ such that
\[ \PP\left(\sup_{s, t \in T} |X_s - X_t| \ge u\right)
\le C_1 \exp\left(-\frac{u^2}{C_1\gamma_2(T, d)^2}\right) \qquad \forall u \ge 0. \]
The special constant $\gamma_2(T, d)$ can be estimated from above by Dudley's entropy bound: there exists some absolute constant $C_2 > 0$ independent of $(T, d)$ such that
\begin{equation}\label{eq:Dudley}
    \gamma_2(T, d) \le C_2 \int_0^\infty  \sqrt{\log N(T, d, r)} dr
\end{equation}
where $N(T, d, r)$ denotes the smallest number of balls of radius $r$ (with respect to $d$) needed to cover $T$.
\end{thm}
\begin{rem}\label{rem:generic-chaining}
When $T$ is some compact interval of $\RR$ with length $|T|$ and $d(s, t) = K|s-t|$ is the Euclidean distance (rescaled by some fixed factor $K > 0$), the cover number with respect to $d$ satisfies $N(T, d, r) \le 1+ \lfloor K|T| / r\rfloor$ for any $r > 0$, and Dudley's entropy bound \eqref{eq:Dudley} yields
\[    \gamma_2(T,d) \ll \int_0^\infty \sqrt{\log  (1+ \lfloor K|T| / r\rfloor)} dr
    = K|T| \int_0^1 \sqrt{\log (1 + \lfloor u^{-1} \rfloor )} du \ll K|T|. \]
\end{rem}

\section{Convergence of abstract random variables}\label{app:abstract}
Throughout this article we often speak of random functions/measures and their convergence, and these concepts can be presented under the unified framework of probability on complete separable metric (Polish) spaces. Before summarising a few abstract probability facts here, let us collect a few function spaces that are used in the main text. In this appendix, $K \subset \RR^d$ denote a compact subset.
\begin{itemize}
\item For any integer $k \ge 0$, denote by $C^k(K)$ the space of functions with uniformly continuous derivatives of all orders $\le k$. This is a Banach space when equipped with the norm
\[    \|u\|_{C^k(K)} := \sum_{ |\alpha| \le k} \sup_{x \in K} |\partial^{\alpha} u(x)| \qquad \text{where} \quad \partial^\alpha u := \frac{\partial^{|\alpha|}u}{\partial x_1^{\alpha_1} \dots \partial x_d^{\alpha_d}}
    \quad  \forall \alpha = (\alpha_1, \dots, \alpha_d) \in \ZZ_{\ge 0}^d,\]
and it is separable (say, by Stone--Weierstrass theorem). When $k=0$ we simply write $C(K) = C^0(K)$.
\item Let $\Ma(K)$ be the space of Radon measures on $K$ equipped with the topology of weak convergence, i.e.~$\nu_n \xrightarrow[n \to \infty]{} \nu_\infty$ in $\Ma(K)$ if $\nu_n(f) \xrightarrow[n \to \infty]{} \nu_\infty(f)$ for all $f \in C(K)$. It is well known that $\Ma(K)$ is a Polish space and the Prokhorov metric completely metrises the topology (see e.g.~\cite[Theorem 4.2 and Lemma 4.3]{Kal2017}).
\end{itemize}
Let $(\Omega, \Fa, \PP)$ be a probability space and $(\Sa, d)$ a complete separable metric space. A random variable is a $\Fa$-measurable function $X\colon \Omega \to \Sa$. Given a collection of ($\Sa$-valued) random variables $(X_n)_{n \ge 1}$ and $X_\infty$, we say:
\begin{itemize}
\item $X_n$ converges almost surely to $X_\infty$ if $\PP\left(\{\omega \in \Omega: \lim_{n \to \infty} X_n(\omega) = X_\infty(\omega)\}\right) = 1$.
\item $X_n$ converges in probability to $X_\infty$ if $\lim_{n \to \infty} \PP\left( d(X_n, X) > \OurEpsilon \right) = 0$ for any $\OurEpsilon > 0$.
\end{itemize}
We now collect some abstract probability results. The first one is the subsequential limit characterisation of convergence in probability.
\begin{lem}[{cf.~\cite[Lemma 5.2]{Kal2021}}]\label{lem:subseq_rv}
Let $X_\infty$, $X_1, X_2, \dots$ be random variables taking values in a separable metric space $(\Sa, d)$. Then $X_n \xrightarrow[n \to \infty]{p} X_\infty$ if and only if the following is true: for any subsequence $(n_k)_k$ there exists a further subsequence $(n_{m_k})_k$ along which we have $X_{n_{m_k}} \xrightarrow[k \to \infty]{a.s.} X_\infty$.
\end{lem}
The next result provides a similar equivalent characterisation of convergence of random measures.
\begin{lem}[{cf.~\cite[Lemma 4.8(iii)]{Kal2017}}]\label{lem:subseq_measure}
Let $\nu_\infty$, $\nu_1, \nu_2, \dots$ be random measures on some compact set $K \subset \RR^d$. Then $\nu_n \xrightarrow[n \to \infty]{p} \nu_\infty$ in $\Ma(K)$ if and only if the following is true: for any subsequence $(n_k)_k$ there exists a further subsequence $(n_{m_k})_k$ along which we have $\nu_{n_{m_k}} \xrightarrow[k \to \infty]{a.s.} \nu_\infty$ in $\Ma(K)$.
\end{lem}

\section{Sobolev inequality}
\begin{thm}\label{thm:sobolev}
Let $I \subset \RR$ be a compact interval. There exists a universal constant $C = C(I)>0$ such that
\[    \sup_{x \in I} |u(x)| \le C \left[ \int_I \left(|u(t)|^2 + |u'(t)|^2 \right) dt \right]^{1/2} =: C\|u\|_{W^{1,2}(I)} \]
for any $u \in C^1(I)$.
\end{thm}
\begin{proof}
This follows from standard Sobolev inequality. For a sketch of proof, recall by Morrey's inequality (see e.g.~\cite[Section 5.6.2]{evans}) that there exists some constant $C_1 > 0$ such that
\[   \|u\|_{C^{0, 1/2}(I)} := \sup_{\substack{x, y \in I \\ x \ne y}} \frac{|u(x) - u(y)|}{|x-y|^{1/2}} \le C_1 \|u\|_{W^{1,2}(I)}. \]
On the other hand, for any $x_0 \in I$ we have
\[ \sup_{x \in I} |u(x)|
\le |u(x_0)| + \sup_{x \in I} |u(x) - u(x_0)|
\ll_I |u(x_0)| + \|u\|_{C^{0, 1/2}(I)}. \]
The desired inequality follows by choosing $x_0 \in I$ such that $|u(x_0)| = \min_{x \in I} |u(x)|$ and noting that $\min_{x \in I} |u(x)| \ll_I \left(\int_I |u(t)|^2 dt\right)^{1/2} \le \|u\|_{W^{1,2}(I)}$.
\end{proof}

\bibliographystyle{abbrv}
\bibliography{references}

\Addresses
\end{document}